\documentclass[11pt, a4paper,reqno]{amsart}
\pdfoutput=1

\usepackage[pagebackref=false,bookmarks=true,colorlinks=true,linktoc=page,citecolor=darkgreen,linkcolor=blue,urlcolor=mediumblue]{hyperref}

\usepackage{tikz,genyoungtabtikz,enumitem,rotating,latexsym,bm,stmaryrd}
 \usepackage{amsmath,amsthm,amsfonts,amssymb,mathrsfs}
 \usepackage{mathtools}
  \usepackage{caption,cases}
 \usepackage{mathtools}

\usetikzlibrary{positioning,intersections,decorations}

\usetikzlibrary{shapes}
\usetikzlibrary{patterns,snakes}

  \usepackage[a4paper,margin=1.3in]{geometry}
 \usepackage{cleveref} 
 \newcommand{\node}{(x,y)}
\usepackage{xcolor}
\definecolor{mediumblue}{rgb}{0.0, 0.0, 0.8}
\colorlet{darkgreen}{green!50!black}

\renewcommand{\geq}{\geqslant}

\renewcommand{\leq}{\leqslant}

\renewcommand{\trianglerighteq}{\trianglerighteqslant}

\tikzset{wei/.style=
{red,double=red,double
distance=1pt}}

\usepackage{ltxtable}
\usepackage{tabularx}

\tikzset{wei2/.style={red,double=red,double
distance=1pt}}

\allowdisplaybreaks
\numberwithin{equation}{section}
\usepackage{scalefnt}

\newtheorem{thm}{Theorem}[section]
\newtheorem{cor}[thm]{Corollary}

\newtheorem*{ack}{Acknowledgements}

\newtheorem{conj}[thm]{Conjecture}

\newtheorem{prop}[thm]{Proposition}
\newtheorem*{prop*}{Proposition}

\newtheorem*{cor*}{Corollary}

\newtheorem*{conj*}{Conjecture}

\theoremstyle{remark}

\newtheorem{rmk}[thm]{Remark}  

\newtheorem{rem}[thm]{Remark}

\newtheorem*{Acknowledgements*}{Acknowledgements}

\theoremstyle{definition}
\newtheorem{defn}[thm]{Definition}
\newtheorem{eg}[thm]{Example}

\newcommand{\PStd}{{\rm PStd}}

\newcommand{\SStd}{{\rm SStd}}

\newcommand{\maxp}{\operatorname{maxp}}
\newcommand{\dom}{\operatorname{dom}}

\newcommand{\Rem}{\operatorname{Rem}}
\renewcommand{\rem}{\operatorname{rem}}
\newcommand{\la}{\lambda}
\newcommand{\al}{\alpha}

 \newcommand{\SSTS}{\mathsf{S}}

\newcommand{\SSTT}{\mathsf{T}}  
\newcommand{\SSTU}{\mathsf{U}}  
\newcommand{\SSTV}{\mathsf{V}}  
\newcommand{\SSTW}{\mathsf{W}}  

\newcommand{\sts}{\mathsf{s}}  
\newcommand{\stt}{\mathsf{t}}  

\newcommand{\ZZ}{{\mathbb Z}}
\newcommand{\NN}{{\mathbb N}}

\let\<=\langle
\let\>=\rangle

\tikzset{
    ultra thin/.style= {line width=0.05pt},
    very thin/.style=  {line width=0.2pt},
    thin/.style=       {line width=0.1pt},
    semithick/.style=  {line width=0.6pt},
    thick/.style=      {line width=0.8pt},
    very thick/.style= {line width=1.2pt},
    ultra thick/.style={line width=1.6pt}
}

\crefname{defn}{Definition}{Definitions}
\crefname{thm}{Theorem}{Theorems}
\crefname{prop}{Proposition}{Propositions}
\crefname{lem}{Lemma}{Lemmas}
\crefname{cor}{Corollary}{Corollaries}
\crefname{conj}{Conjecture}{Conjectures}
\crefname{section}{Section}{Sections}
\crefname{subsection}{Subsection}{Subsections}
\crefname{eg}{Example}{Examples}
\crefname{figure}{Figure}{Figures}
\crefname{rem}{Remark}{Remarks}
\crefname{rmk}{Remark}{Remarks}
\crefname{equation}{equation}{equation}

\Crefname{defn}{Definition}{Definitions}
\Crefname{thm}{Theorem}{Theorems}
\Crefname{prop}{Proposition}{Propositions}
\Crefname{lem}{Lemma}{Lemmas}
\Crefname{cor}{Corollary}{Corollaries}
\Crefname{conj}{Conjecture}{Conjectures}
\Crefname{section}{Section}{Sections}
\Crefname{subsection}{Subsection}{Subsections}
\Crefname{eg}{Example}{Examples}
\Crefname{figure}{Figure}{Figures}
\Crefname{rem}{Remark}{Remarks}
\Crefname{rmk}{Remark}{Remarks}

 \newlength{\mylen}
\setbox1=\hbox{$\bullet$}\setbox2=\hbox{\tiny$\bullet$}
\usepackage{amsmath}
\newcommand\Item[1][]{%
  \ifx\relax#1\relax  \item \else \item[#1] \fi
  \abovedisplayskip=0pt\abovedisplayshortskip=0pt~\vspace*{-\baselineskip}}

\def\Item{\item\abovedisplayskip=0pt\abovedisplayshortskip=5pt~\vspace*{-\baselineskip}}

\begin{document}

\title[Classification of multiplicity-free plethysm products]{The classification of multiplicity-free \\ plethysms of Schur functions}

\author{Christine Bessenrodt}
\address{Institut f\"ur Algebra, Zahlentheorie und Diskrete Mathematik, Leibniz Universit\"at Hannover, 30167 Hannover, Germany}
\email{bessen@math.uni-hannover.de}

\author{Chris Bowman}
\address{Department of Mathematics,  University of  York, Heslington, 
YO10 5DD}
\email{Chris.Bowman-Scargill@york.ac.uk}

\author{Rowena Paget}
\address{School of Mathematics, Statistics and Actuarial Science, University of Kent Canterbury
CT2 7NF, UK}
\email{R.E.Paget@kent.ac.uk}


\begin{abstract}
 We classify and construct all multiplicity-free plethystic products of Schur functions.  We also  compute many new (infinite) families of
 plethysm coefficients, with particular emphasis on those  near maximal  in the dominance ordering and those of small Durfee size.

 \end{abstract}

 \maketitle

\section{Introduction}

In the ring of symmetric functions there
are three ways of ``multiplying" a
pair of functions together in order to obtain a new symmetric function; these are the  outer product,
the Kronecker product,  and the plethysm product.
With $s_\nu$ and $s_\mu$ denoting the Schur functions labelled by the partitions $\nu$ and~$\mu$, the coefficients in the expansion of their outer product $s_\nu \boxtimes s_\mu$ in the basis of Schur functions
are determined by the famous Littlewood--Richardson Rule.
 Richard Stanley identified understanding the  Kronecker and   plethystic products of pairs of Schur functions as two of the most important open problems in algebraic combinatorics \cite[Problems 9 \& 10]{Sta00}; the corresponding expansion coefficients have been described as `perhaps the most challenging, deep and mysterious objects in algebraic combinatorics' \cite{PP1}.
More recently,
 the Kronecker coefficients
have provided the centrepiece of geometric complexity theory, an approach that seeks to settle the $\sf P$ vs $\sf NP$ problem \cite{GCT6};  this approach was recently shown to require not only positivity, but precise information on the coefficients \cite{MR3631001,MR3631002,MR3695867,MR3724215}. The
Kronecker and plethysm coefficients  have also been found to have deep connections with quantum information theory   \cite{ky,MR2197548,MR2421478,MR2745569}.

In 2001, Stembridge  classified the multiplicity-free outer products of  Schur functions \cite{Stembridge01}.
 At a similar time,   Bessenrodt conjectured a classification of
  multi\-pli\-city-free Kronecker products
  of Schur functions.
Multiplicity-free Kronecker products have subsequently been studied in \cite{44,BWZ10,CG1,MR2550164} and Bessenrodt's conjecture was finally    proven in \cite{cbcb}.
Finally, the multiplicity-free plethystic
products have been studied in \cite{cariniremmel,carini}
 and  the well-known   formulas of \cite[Chapter 1, Plethysm]{MR3443860}.
The purpose of this article is to classify and construct all multiplicity-free plethysm products of Schur functions thus completing this picture:

 \begin{thm}\label{conj}
 The plethysm product $s_\nu \circ s_\mu$ is multiplicity-free if and only if one of the following holds:
 \begin{itemize}[leftmargin=*]
 \item[$(i)$] either $\nu$ or $\mu$ is the partition  $(1)$  and the other is arbitrary;
  \item[$(ii)$] $\nu\vdash 2$ and $\mu$ is   $(a^b)$,   $(a+1,a^{b-1}) $, $(a^b,1) $,    $(a^{b-1},a-1)$
 or
a hook;

 \item[$(iii)$]  $\mu\vdash 2$ and $\nu$ is linear or $\nu$ belongs to a small list of exceptions
 $$\nu \in \{(4,1),(3,1),(2,1^a), (2^2),(3^2),(2^2,1)\mid 1\leq a \leq 6 \} ;$$

 \item[$(iv)$]  $\nu$ and $\mu$ belong to a finite list of  small rank exceptional products.  In particular $\nu$ and $\mu$ are both linear and $|\nu|+|\mu|\leq 8$ and $(\nu,\mu)\not\in \{((5),(3))	, ((1^5),(1^3)), $ $((4),(4))	, ((4),(1^4)) \}   $;
 or  $\nu=(1^2)$ and $\mu\in\{ (4,2),(2^2,1^2)\}$;
 or $\nu=(1^3)$ and $\mu \in \{(6), (1^6) , (2^2)\}$;
or $ \nu = (2,1)$ and $\mu\in \{(3),(1^3)\}$.  \end{itemize}
  \end{thm}

The first, and easier, half of the proof is given in \cref{easyhalf} where we show that all the products on the list are, indeed, multiplicity-free and we calculate these decompositions explicitly.
The more difficult half of the theorem (proving that this list is exhaustive) is the subject of \cref{sec:proof} and \cref{sec:endproof}.  The main idea   is to calculate ``seeds" of multiplicity using the combinatorics of
plethystic tableaux
and then to use semigroup properties to ``grow" these seeds and hence show that any product,  $s_\nu \circ s_\mu$, not on the list contains coefficients which are strictly greater than 1.

Finally, during the course of writing this paper we stumbled on the following
new monotonicity property.
 We believe it will be of interest as it is of a different flavour to the known monotonicity properties of plethysm coefficients \cite{plethysmstab1,BPW,brion,MR1190119}. The notation is as defined in \cref{subsec:partitions}.

  \begin{conj}\label{conj2}
For $\nu$ and $\alpha$ arbitrary partitions, we have that  $$
 \langle s_\nu \circ s_{(2)} \mid s_\alpha \rangle
\leq
 \langle s_{\nu \sqcup (1)} \circ s_{(2)} \mid s_{(\alpha+(1))\sqcup (1)} \rangle.
$$
   \end{conj}

\begin{ack}
We would like to thank John Stembridge for making available his SF Package.
The second author would like to thank   the Alexander von Humboldt Foundation
and EPSRC fellowship grant EP/V00090X/1 for financial support and the Leibniz Universit\"at Hannover
 for their ongoing hospitality.
\end{ack}

 \section{Partitions, symmetric functions\\  and maximal terms in plethysm}

 \subsection{Partitions and Young tableaux}\label{subsec:partitions}
  We define a {\sf composition} $\lambda\vDash n$ to be a   finite  sequence  of non-negative integers $ (\lambda_1,\lambda_2, \ldots)$ whose sum, $|\lambda| = \lambda_1+\lambda_2 + \dots$, equals $n$.
  If the sequence $(\lambda_1,\lambda_2, \ldots)$ is  weakly decreasing,  we say that $\la$ is a {\sf partition} and write $\la\vdash n$.       Given  a partition $\lambda $ of $n$, its {\sf Young diagram}   is defined to be the set
\[
[\la]=\{(r,c) \mid  1\leq  c\leq \lambda_r\} 
\]
\color{black} and we refer to each $(r,c)\in [\la]$ as a {\sf node}  or a {\sf box} of the partition.  \color{black} 
 The conjugate partition, $\la^T$, is the partition obtained by interchanging the rows and columns of $\la$.  The number of non-zero parts of a partition $\la$  is called its {\sf length}, $\ell(\lambda)$;
  its largest part $\lambda_1$ is also
 called its {\sf width}, $w(\la)$; the sum $|\lambda|$
 of all the parts of $\lambda$ is called its  size.
We let $\la_{>1}$ denote the partition obtained by removing the first row of $\la$.  
\color{black} 
 We define the {\sf Durfee size} of $\la$ to be the largest value $k$ such that $(k,k)\in [\la]$ and we denote this by $dl(\la)$.  
We say that a node $(r,c)\in [\la]$ is {\sf removable} if $[\la]\setminus \{(r,c)\}$ is itself the Young diagram of a partition.
\color{black} 
We let $\Rem(\la)$
denote the set of all removable
nodes of the partition $\la$,
and set $\rem(\la)=|\Rem(\la)|$.
If $(r,c) \in \Rem(\la)$ then we will write $\la-\varepsilon_r$ for the partition obtained by removing the (unique) removable node in row $r$ from $\la$. Similarly, 
\color{black} 
$(r,c)\notin [\la]$ is {\sf addable} if $[\la]\cup\{(r,c)\}$ is  the Young diagram of a partition, and
\color{black} 
if $(r,c)$ is an addable node of $\la$ then
$\la+\varepsilon_r$ denotes the partition obtained by adding the (unique) addable node in row $r$ to $\la$.

 We now recall the {\sf dominance ordering} on partitions. Let $\la,\mu$ be partitions.
 We  write
 $\la\trianglerighteq \mu$  if
 $$
 \sum_{1\leq i \leq k}\la_i \geq  \sum_{1\leq i \leq k}\mu_i \text{  for all } k\geq 1.
 $$
 If $\la\trianglerighteq \mu$ and $\la\neq \mu$ we write $\la\rhd \mu$.  The dominance ordering  is a partial ordering on the set of  partitions of a given size.  This partial order can be refined into a total ordering as follows:
  we write
$\la\succ \mu$  if
 $$
\text{  $\la_k >\mu_k$ for some $k\geq 1$ and  }
 \la_i = \mu_i \text{ for all } 1\leq i \leq k-1.
 $$
  We refer to   $\succ $  as the {\sf lexicographic ordering}.

Let   $\lambda   $ be  a partition  of $n$.
 A {\sf    Young tableau of shape $  \lambda$} {(\color{black}usually referred to simply a $\la$-tableau, for brevity)}
  may be defined as  a map
 $\stt : [\la] \to \NN.$  Recall that the tableau $\stt$ is {\sf semistandard} if  $\stt(r,c-1)\leq \stt(r,c)$   and
$\stt(r-1,c)< \stt(r,c)$  for all $(r,c)\in [\la]$.
 We let $\stt_k = |\{ (r,c)\in [\la] \mid \stt(r,c)=k\}|$ for $k\in \NN$.
 We refer to the composition $\alpha=(\stt_1,\stt_2,\stt_3,\dots)$ as the {\sf weight} of the tableau~$\stt$.
 We denote the set of all   semistandard
  tableaux of shape $\la$ by $\SStd_\NN(\la)$,  and the subset of those having weight $\alpha$ by $\SStd(\la,\alpha)$.  \color{black}  
 We remark that a necessary condition for 
 $\SStd(\la,\alpha)\neq \emptyset $ is that 
  $\la \trianglerighteq  \alpha$
 \color{black}
  in the dominance order (and so similarly for the lexicographic order).  
 
We write   $\la \subseteq \nu$ if $\la_i \leq \nu_i$ for all $i\geq 1$.  Given $\la \subseteq \nu$, we define the corresponding
 skew partition $\nu \setminus \la$ to be the set difference of the Young diagrams.  
 We extend all the tableaux-theoretic notions above to skew-partitions in the obvious manner, in particular we let 
 $\SStd(\nu\setminus \la,\mu)$ denote the set of skew semistandard tableaux of shape $\nu \setminus \la$  weight $\mu$.   

\color{black}

 Given two partitions $\la $ and $\mu$, we let  $\la + \mu$  and  $\la \sqcup \mu$ denote the partitions obtained by adding the partitions
 horizontally and vertically, respectively.    In more detail,
 $$
 \la+\mu=(\la_1+\mu_1,\la_2+\mu_2,\la_3+\mu_3,\dots )
 $$
 and $\la\sqcup\mu$ is the partition whose multiset of parts is the disjoint union of the multisets of parts of $\la$ and $\mu$. We have that
  $$\la\sqcup\mu= (\la^T+\mu^T)^T.$$
   Going forward, we require the following terminology.  We call the partition $\la$ of~$n$
\begin{itemize}[leftmargin=*]
  \item
   {\sf linear } if  $\la=(n)$ or $(1^n)$;
 \item
 a {\sf 2-line partition} if
   the minimum of $\ell(\la)$ and $w(\la)$ is exactly~2;

\item
  a {\sf fat hook} if $\rem(\la)\leq 2$;

\item
  a {\sf proper fat hook} if $\rem(\la)= 2$, and $\la$ is not a hook or a 2-line partition;

\item
 a {\sf rectangle} if $\la$ is of the form $(a^b)$ for some $a,b \geq 1$;

  \item
  a  {\sf near rectangle} if $\la$ is obtained from a rectangle by adding a single row or column;
  
  \item \color{black} an {\sf almost rectangle} if $\la$ is obtained from a rectangle by adding or removing  a single node.
\end{itemize}

\subsection{Symmetric functions and multiplicity-free products}
  Given $\la$ a partition of $n$, the associated {\sf Schur function}, $s_\la$,  may be defined as follows:
\begin{equation}\label{plethysm3}
   s_\la = \sum_{
   \begin{subarray}c
   \alpha \vDash n
   \end{subarray}
   }
| \SStd_\NN(\la,\alpha)
| x^\alpha
 \qquad
\text{   where} \qquad   x^\alpha= x_1^{\alpha_1} x_2^{\alpha_2} x_3^{\alpha_3}\dots .
\end{equation}
We will also require the elementary and homogenous symmetric functions
$$
e_\la= s_{\la_1^T}s_{\la_2^T}\dots s_{\la_w^T}
\quad
h_\la= s_{\la_1}s_{\la_2}\dots s_{\la_\ell}
$$
for $\la$ a partition of width $w$ and length $\ell$.
  There are three fundamental products on symmetric functions: the outer (Littlewood--Richardson) product $\boxtimes$, the inner (Kronecker) product $\otimes$, and the plethysm product $\circ$ all of which are explicitly defined in \cite[Chapter 1]{MR3443860}.
In 2001, Stembridge  classified the multiplicity-free outer products of symmetric functions
(or equivalently, the outer product of two irreducible
characters of symmetric groups)
as follows:

\begin{thm}[Multiplicity-free outer products of Schur functions
\cite{Stembridge01}\label{thm:mf-outer}]
 An outer product $s_\mu \boxtimes s_\nu$ is multiplicity-free if and only if
one of the following holds:
\begin{itemize}[leftmargin=0pt,itemindent=1.5em]
\item
$\mu$ and $\nu$ are both rectangles,
\item
$\mu$ is a rectangle and $\nu$ is a near-rectangle (up to exchange);
\item
$\mu$ is a 2-line rectangle and $\nu$ is a fat hook (up to exchange);
\item
$\mu$ or $\nu$ is linear (and the other is arbitrary).
\end{itemize}
 \end{thm}

 We will make use of Stembridge's classification in the proof.  At a similar time, Bessenrodt conjectured a classification of all multiplicity-free Kronecker products.  This conjecture was recently proven in \cite{cbcb} and we refer to \cite{cbcb} for the full statement (as it will not be needed here).
 However, we do invite the reader to compare all three classification theorems.
  All three have a trivial case in which one partition is arbitrary and the other is particularly simple
  (linear for the outer and  Kronecker products, or $(1)$ for the plethysm product).
  Except for
  this trivial case, all three classifications satisfy the restraint  that if
  $$\color{black}
  s_\nu \boxtimes s_\mu
  \quad
    s_\nu \otimes s_\mu
    \quad
      s_\nu \circ  s_\mu
  $$ 
  is multiplicity-free, then
 $\rem(\mu)+\rem(\nu)\leq 4$.
 Also, the methods of proof for the Kronecker and plethystic classifications are very similar:
 in both cases a complementary pairing of semigroup properties and
 consideration of near maximal terms (using Dvir recursion in the former and \cref{dvir} in the latter)
 are the key ingredients.

\subsection{Plethysm }The  {\sf plethysm product} of two symmetric functions
 is defined in \cite[Chapter~7,   A2.6]{MR1676282} or   \cite[Chapter I.8]{MR3443860}.
 The plethysm product of two Schur functions  is again a symmetric function  and so can be rewritten as a linear combination of Schur functions.
 For $\nu \vdash n$, $\mu \vdash m$ we have
$$
s_{\nu}\circ s_{\mu}
 =\sum_{\alpha \vdash mn} p(\nu,\mu,\alpha) s_{\alpha}
 $$
 where the coefficients $p(\nu,\mu,\alpha)=\langle  s_{\nu}\circ s_{\mu} \mid s_{\alpha} \rangle$
    may be computed using the Hall inner product;
 they are non-negative as they are representation-theoretic multiplicities.
 We set
  $$
   p(\nu,\mu)
  = \max\{p(\nu,\mu,\alpha) \mid \alpha \vdash mn	\} .
  $$
Given   a total ordering, $>$, on partitions we let
 $$\maxp_{>}(\nu,\mu)$$
denote the  unique  partition $\lambda$  such that $p(\nu,\mu,\lambda)\neq 0$ and  $p(\nu,\mu,\alpha)=0$ for all $\alpha >  \lambda$.

\begin{thm}[\cite{BPW}]
\label{pppppppppp}
\color{black}Let $\mu $,  $\nu $ be partitions of $m$ and $n$ respectively.
The  maximal term of $s_\nu \circ s_\mu$ in the lexicographic order
 is  labelled by the partition
  $$
\maxp_{\succ } (\nu,\mu)  =
(n\mu_1,n\mu_2,\dots, n\mu_{\ell(\mu) -1},n\mu_{\ell(\mu)}-n+\nu_1, \nu_2, \dots ,\nu_{\ell(\nu)}) .
 $$
  Moreover,  the corresponding coefficient is equal to  1.
\end{thm}

We  recall the role  conjugation
 plays   in plethysm (see, for example, \cite[Ex. 1, Chapter I.8]{MR3443860}). For $\mu\vdash m$, $\nu \vdash n$, and $\alpha\vdash mn$  we have that
\begin{equation}\label{conjugate}
p( \nu,\mu,\alpha)=
\begin{cases}
p( \nu,\mu^T,\alpha^T)				&\text{if $m$ is even}\\
p( \nu^T,\mu^T,\alpha^T)				&\text{if $m$ is odd.}
\end{cases}
\end{equation}
  In order to keep track of the effect of this conjugation
  we set
\begin{equation}  \nu^M=\begin{cases}
\nu 		&\text{if $m$ is even}\\
\nu^T &\text{if $m$ is odd.}
\end{cases}
\end{equation}
In particular, we note that
\begin{equation}\label{max-conjugate}
p( \nu,\mu)= p(\nu^M,\mu^T)=
\begin{cases}
p( \nu,\mu^T)				&\text{if $m$ is even}\\
p( \nu^T,\mu^T)				&\text{if $m$ is odd.}
\end{cases}
\end{equation}

 \begin{thm}[\cite{BPW}]\label{PW1}
 For $r\in \NN$ such that $r\geq w(\mu)$,   we have
 $$
p(\nu,(r) \sqcup \mu , (nr) \sqcup \la ) = p(\nu,  \mu , \la ) .
$$
 \end{thm}

  \begin{thm}[\cite{BPW}]\label{PW2}
 For any  $r\in \NN$,
 $$
p(\nu, (1^{r}) + \mu , (n^{r}) +\la )
 \geq
p(\nu, \mu , \la )
 $$
and so by repeated applications of this we obtain
 $$
p(\nu, \al + \mu , n \al +\la )
 \geq
p(\nu, \mu , \la ) .
 $$
 \end{thm}

The following theorem appears explicitly (in the form stated below) in
    {\cite[Proposition 3.6 (R2)]{plethysmstab1}} where it is attributed to earlier  work of Brion  {\cite[Corollary 1, Section 2.6]{brion}}.

\begin{thm}[{\cite{brion}  and \cite{plethysmstab1}}] \label{Brion}
 We have that
 $$\langle s_{\nu+(1)} \circ s_{  \mu }\mid  s_{\la + \mu }\rangle
 \geq
 \langle s_\nu \circ s_{  \mu }\mid  s_{\la}\rangle ,
 $$
 and so by repeated application we obtain
  $$p(\nu +(r) , \mu , \la + r \mu )
 \geq
 p(\nu , \mu , \la ) .
 $$
\end{thm}

We collect together the information on the numbers $p(\nu,\mu)$ obtained from the results above.

\begin{cor}\label{cor:semigroup}
Let $r\in \NN$ and $\al$ be a partition.
Then we have:
\begin{enumerate}
\item
$p(\nu, (r) \sqcup \mu) \geq p(\nu,\mu)$  if $r\geq w(\mu)$.
\item
$p(\nu,  \al + \mu) \geq p(\nu,\mu)$.
\item
$p(\nu + (r),  \mu) \geq p(\nu,\mu)$.
\item
$p(\nu,\mu \sqcup (1)) \geq p(\nu^T,\mu)$.
\end{enumerate}
\end{cor}
\begin{proof}
We only add an argument for the last property which is useful when
the set of partitions $\nu$ under consideration is closed under conjugation.
If $m=|\mu|$ is even, then
$p(\nu,\mu \sqcup (1)) = p(\nu^T,\mu^T+(1)) \geq p(\nu^T,\mu^T) = p(\nu^T,\mu)$.
Similarly, if $m=|\mu|$ is odd, then
$p(\nu,\mu \sqcup (1)) = p(\nu,\mu^T+(1)) \geq p(\nu,\mu^T) = p(\nu^T,\mu)$.
\end{proof}

The properties above imply the following.
\begin{cor}\label{cor:2-important}
Let $\mathcal{N}$ be a set of partitions that is closed under conjugation and
such that $p(\nu, (2)) \geq 2$ for all $\nu \in \mathcal{N}$.
Then for $m>1$ and any $\mu \vdash m$ we have
$p(\nu, \mu) \geq 2$.
\end{cor}

\begin{proof}
\color{black}
Given $\nu\in \mathcal N$  such that $ p(\nu,(2))\geq 2$, 
we have that
$$2\leq p(\nu,(2))\leq p(\nu,(2)+(m-2))=p(\nu,(m))=p(\nu^M,(1^m))$$
 for all $m\geq 2$ by \cref{cor:semigroup}(2); we remark that $\mathcal{N}$ is closed under conjugation and so $\nu^M\in    \mathcal{N}$ and  the result follows for linear partitions $\mu$ of $m\geq 2$.  
Now assume that $\mu  $ is non-linear and so $\mu_1^T\geq2$ and therefore 
$$2\leq p(\nu,(1^{\mu_1^T}))\leq p(\nu,(1^{\mu_1^T})+(1^{\mu_2^T})+\dots )= p(\nu,\mu)$$
by \cref{cor:semigroup}(2).  The result follows.
\end{proof}

\subsection{Plethystic tableaux}

Sometimes we shall use the dominance ordering~$\rhd $ to compare the summands of $s_{\nu}\circ s_{\mu}$, and then there will, in general, be many  (incomparable) maximal partitions. To understand these summands, we require some further definitions.    We place a lexicographic  ordering, $\prec$, on the set of semistandard Young tableaux as follows.
  Let $\sts\neq\stt$ be semistandard $\mu$-tableaux,   and consider the leftmost column in which $\sts$ and $\stt$ differ.  We write
 $\sts \prec \stt$ if  the greatest entry not appearing in both columns lies in $\stt$.  Following \cite[Definition~1.4]{BPW}, we define a {\sf    plethystic  tableau of shape $  \mu^\nu$}  and weight $\alpha$ to be  a map
 $$\SSTT : [\nu] \to \SStd_\NN(\mu)$$ such that the total  number of occurrences of
$k$ in the tableau entries of $\SSTT$ is $\alpha_k$ for each $k$.
 We say that such a tableau is {\sf semistandard} if  $\SSTT(r,c-1) \preceq  \SSTT(r,c)$   and
$\SSTT(r-1,c)  \prec  \SSTT(r,c)$  for all $(r,c)\in [\nu]$. An example follows in Figure~1.
   We denote the set of all plethystic  tableaux of shape $\mu^\nu$ and weight $\alpha$ by $\PStd(\mu^\nu,
  \alpha)$.
By \cite[Section~3]{BPW} we have that
\begin{equation}\label{plethysm1}
s_\nu\circ s_\mu = \sum_{\alpha} |{\rm PStd}(\mu^\nu,\alpha)|x^\alpha .
\end{equation}
This will be a key tool in what follows.

 \!\!\begin{figure}[ht!]
$$  \begin{tikzpicture} [scale=0.6]

\path(0,0) coordinate (origin);
 \begin{scope}{\draw[very thick](origin)--++(0:2)--++(-90:1)--++(180:1)--++(-90:1)--++(180:1)--++(90:2);

   \draw(0.5,-0.5)  node {1};
       \draw(1.5,-0.5)node {1};
              \draw(0.5,-1.5)node {2};

  \clip (origin)--++(0:2)--++(-90:1)--++(180:1)--++(-90:1)--++(180:1)--++(90:2);

  \path(origin) coordinate (origin);
   \foreach \i in {1,...,19}
  {
    \path (origin)++(0:1*\i cm)  coordinate (a\i);
    \path (origin)++(-90:1*\i cm)  coordinate (b\i);
    \path (a\i)++(-90:10cm) coordinate (ca\i);
    \path (b\i)++(0:10cm) coordinate (cb\i);
    \draw[thin] (a\i) -- (ca\i)  (b\i) -- (cb\i); }

    } \end{scope}

  \path(3,0) coordinate (origin);
 \begin{scope}{\draw[very thick](origin)--++(0:2)--++(-90:1)--++(180:1)--++(-90:1)--++(180:1)--++(90:2);
      \path(origin)--++(0:0.5)--++(-90:0.5)  node {1};
       \path(origin)--++(0:1.5)--++(-90:0.5)node {1};
            \path(origin)--++(0:0.5)--++(90:-1.5)node {3};

  \clip (origin)--++(0:1)--++(-90:1)--++(180:1)--++(-90:1)--++(180:1)--++(90:1);
  \path(origin) coordinate (origin);
   \foreach \i in {1,...,19}
  {
    \path (origin)++(0:1*\i cm)  coordinate (a\i);
    \path (origin)++(-90:1*\i cm)  coordinate (b\i);
    \path (a\i)++(-90:10cm) coordinate (ca\i);
    \path (b\i)++(0:10cm) coordinate (cb\i);
    \draw[thin] (a\i) -- (ca\i)  (b\i) -- (cb\i); }

 \path(0.5,-0.5) coordinate (origin);
 \foreach \i in {1,...,19}
  {
    \path (origin)++(0:1*\i cm)  coordinate (a\i);
    \path (origin)++(-90:1*\i cm)  coordinate (b\i);
    \path (a\i)++(-90:1cm) coordinate (ca\i);
        \path (ca\i)++(-90:1cm) coordinate (cca\i);
    \path (b\i)++(0:1cm) coordinate (cb\i);
    \path (cb\i)++(0:1cm) coordinate (ccb\i);
  }} \end{scope}

    \path(6,0) coordinate (origin);
 \begin{scope}{\draw[very thick](origin)--++(0:2)--++(-90:1)--++(180:1)--++(-90:1)--++(180:1)--++(90:2);
         \path(origin)--++(0:0.5)--++(-90:0.5)  node {1};
       \path(origin)--++(0:1.5)--++(-90:0.5)node {1};
            \path(origin)--++(0:0.5)--++(90:-1.5)node {3};

  \clip (origin)--++(0:1)--++(-90:1)--++(180:1)--++(-90:1)--++(180:1)--++(90:1);
  \path(origin) coordinate (origin);
   \foreach \i in {1,...,19}
  {
    \path (origin)++(0:1*\i cm)  coordinate (a\i);
    \path (origin)++(-90:1*\i cm)  coordinate (b\i);
    \path (a\i)++(-90:10cm) coordinate (ca\i);
    \path (b\i)++(0:10cm) coordinate (cb\i);
    \draw[thin] (a\i) -- (ca\i)  (b\i) -- (cb\i); }

 \path(0.5,-0.5) coordinate (origin);
 \foreach \i in {1,...,19}
  {
    \path (origin)++(0:1*\i cm)  coordinate (a\i);
    \path (origin)++(-90:1*\i cm)  coordinate (b\i);
    \path (a\i)++(-90:1cm) coordinate (ca\i);
        \path (ca\i)++(-90:1cm) coordinate (cca\i);
    \path (b\i)++(0:1cm) coordinate (cb\i);
    \path (cb\i)++(0:1cm) coordinate (ccb\i);
  }} \end{scope}

  \path(0,-3) coordinate (origin);
     \path(origin)--++(0:0.5)--++(-90:0.5)  node {1};
       \path(origin)--++(0:1.5)--++(-90:0.5)node {2};
            \path(origin)--++(0:0.5)--++(90:-1.5)node {3};
 \begin{scope}{\draw[very thick](origin)--++(0:2)--++(-90:1)--++(180:1)--++(-90:1)--++(180:1)--++(90:2);

  \clip (origin)--++(0:1)--++(-90:1)--++(180:1)--++(-90:1)--++(180:1)--++(90:1);
  \path(origin) coordinate (origin);
   \foreach \i in {1,...,19}
  {
    \path (origin)++(0:1*\i cm)  coordinate (a\i);
    \path (origin)++(-90:1*\i cm)  coordinate (b\i);
    \path (a\i)++(-90:10cm) coordinate (ca\i);
    \path (b\i)++(0:10cm) coordinate (cb\i);
    \draw[thin] (a\i) -- (ca\i)  (b\i) -- (cb\i); }

 \path(0.5,-0.5) coordinate (origin);
 \foreach \i in {1,...,19}
  {
    \path (origin)++(0:1*\i cm)  coordinate (a\i);
    \path (origin)++(-90:1*\i cm)  coordinate (b\i);
    \path (a\i)++(-90:1cm) coordinate (ca\i);
        \path (ca\i)++(-90:1cm) coordinate (cca\i);
    \path (b\i)++(0:1cm) coordinate (cb\i);
    \path (cb\i)++(0:1cm) coordinate (ccb\i);
  }} \end{scope}

  \path(3,-3) coordinate (origin);
 \begin{scope}{\draw[very thick](origin)--++(0:2)--++(-90:1)--++(180:1)--++(-90:1)--++(180:1)--++(90:2);
      \path(origin)--++(0:0.5)--++(-90:0.5)  node {1};
       \path(origin)--++(0:1.5)--++(-90:0.5)node {1};
            \path(origin)--++(0:0.5)--++(90:-1.5)node {4};

  \clip (origin)--++(0:1)--++(-90:1)--++(180:1)--++(-90:1)--++(180:1)--++(90:1);
  \path(origin) coordinate (origin);
   \foreach \i in {1,...,19}
  {
    \path (origin)++(0:1*\i cm)  coordinate (a\i);
    \path (origin)++(-90:1*\i cm)  coordinate (b\i);
    \path (a\i)++(-90:10cm) coordinate (ca\i);
    \path (b\i)++(0:10cm) coordinate (cb\i);
    \draw[thin] (a\i) -- (ca\i)  (b\i) -- (cb\i); }

 \path(0.5,-0.5) coordinate (origin);
 \foreach \i in {1,...,19}
  {
    \path (origin)++(0:1*\i cm)  coordinate (a\i);
    \path (origin)++(-90:1*\i cm)  coordinate (b\i);
    \path (a\i)++(-90:1cm) coordinate (ca\i);
        \path (ca\i)++(-90:1cm) coordinate (cca\i);
    \path (b\i)++(0:1cm) coordinate (cb\i);
    \path (cb\i)++(0:1cm) coordinate (ccb\i);
  }} \end{scope}

 \draw[very thick] (-0.5,0.5)--(8.5,0.5)--(8.5,-2.5)--(5.5,-2.5)--(5.5,-5.5)--(-0.5,-5.5)--(-0.5,0.5);
 \draw[very thick] (5.5,-2.5)--(-0.5,-2.5);
  \draw[very thick] (5.5,-2.5)--(5.5,0.5);
    \draw[very thick] (2.5,-5.5)--(2.5,0.5);
   \end{tikzpicture}
   \qquad
    \begin{tikzpicture} [scale=0.6]

\path(0,0) coordinate (origin);
 \begin{scope}{\draw[very thick](origin)--++(0:2)--++(-90:1)--++(180:1)--++(-90:1)--++(180:1)--++(90:2);

   \draw(0.5,-0.5)  node {1};
       \draw(1.5,-0.5)node {1};
              \draw(0.5,-1.5)node {2};

  \clip (origin)--++(0:2)--++(-90:1)--++(180:1)--++(-90:1)--++(180:1)--++(90:2);

  \path(origin) coordinate (origin);
   \foreach \i in {1,...,19}
  {
    \path (origin)++(0:1*\i cm)  coordinate (a\i);
    \path (origin)++(-90:1*\i cm)  coordinate (b\i);
    \path (a\i)++(-90:10cm) coordinate (ca\i);
    \path (b\i)++(0:10cm) coordinate (cb\i);
    \draw[thin] (a\i) -- (ca\i)  (b\i) -- (cb\i); }

    } \end{scope}

  \path(3,0) coordinate (origin);
 \begin{scope}{\draw[very thick](origin)--++(0:2)--++(-90:1)--++(180:1)--++(-90:1)--++(180:1)--++(90:2);
      \path(origin)--++(0:0.5)--++(-90:0.5)  node {1};
       \path(origin)--++(0:1.5)--++(-90:0.5)node {1};
            \path(origin)--++(0:0.5)--++(90:-1.5)node {2};

  \clip (origin)--++(0:1)--++(-90:1)--++(180:1)--++(-90:1)--++(180:1)--++(90:1);
  \path(origin) coordinate (origin);
   \foreach \i in {1,...,19}
  {
    \path (origin)++(0:1*\i cm)  coordinate (a\i);
    \path (origin)++(-90:1*\i cm)  coordinate (b\i);
    \path (a\i)++(-90:10cm) coordinate (ca\i);
    \path (b\i)++(0:10cm) coordinate (cb\i);
    \draw[thin] (a\i) -- (ca\i)  (b\i) -- (cb\i); }

 \path(0.5,-0.5) coordinate (origin);
 \foreach \i in {1,...,19}
  {
    \path (origin)++(0:1*\i cm)  coordinate (a\i);
    \path (origin)++(-90:1*\i cm)  coordinate (b\i);
    \path (a\i)++(-90:1cm) coordinate (ca\i);
        \path (ca\i)++(-90:1cm) coordinate (cca\i);
    \path (b\i)++(0:1cm) coordinate (cb\i);
    \path (cb\i)++(0:1cm) coordinate (ccb\i);
  }} \end{scope}

    \path(6,0) coordinate (origin);
 \begin{scope}{\draw[very thick](origin)--++(0:2)--++(-90:1)--++(180:1)--++(-90:1)--++(180:1)--++(90:2);
         \path(origin)--++(0:0.5)--++(-90:0.5)  node {1};
       \path(origin)--++(0:1.5)--++(-90:0.5)node {1};
            \path(origin)--++(0:0.5)--++(90:-1.5)node {2};

  \clip (origin)--++(0:1)--++(-90:1)--++(180:1)--++(-90:1)--++(180:1)--++(90:1);
  \path(origin) coordinate (origin);
   \foreach \i in {1,...,19}
  {
    \path (origin)++(0:1*\i cm)  coordinate (a\i);
    \path (origin)++(-90:1*\i cm)  coordinate (b\i);
    \path (a\i)++(-90:10cm) coordinate (ca\i);
    \path (b\i)++(0:10cm) coordinate (cb\i);
    \draw[thin] (a\i) -- (ca\i)  (b\i) -- (cb\i); }

 \path(0.5,-0.5) coordinate (origin);
 \foreach \i in {1,...,19}
  {
    \path (origin)++(0:1*\i cm)  coordinate (a\i);
    \path (origin)++(-90:1*\i cm)  coordinate (b\i);
    \path (a\i)++(-90:1cm) coordinate (ca\i);
        \path (ca\i)++(-90:1cm) coordinate (cca\i);
    \path (b\i)++(0:1cm) coordinate (cb\i);
    \path (cb\i)++(0:1cm) coordinate (ccb\i);
  }} \end{scope}

  \path(0,-3) coordinate (origin);
     \path(origin)--++(0:0.5)--++(-90:0.5)  node {1};
       \path(origin)--++(0:1.5)--++(-90:0.5)node {2};
            \path(origin)--++(0:0.5)--++(90:-1.5)node {2};
 \begin{scope}{\draw[very thick](origin)--++(0:2)--++(-90:1)--++(180:1)--++(-90:1)--++(180:1)--++(90:2);

  \clip (origin)--++(0:1)--++(-90:1)--++(180:1)--++(-90:1)--++(180:1)--++(90:1);
  \path(origin) coordinate (origin);
   \foreach \i in {1,...,19}
  {
    \path (origin)++(0:1*\i cm)  coordinate (a\i);
    \path (origin)++(-90:1*\i cm)  coordinate (b\i);
    \path (a\i)++(-90:10cm) coordinate (ca\i);
    \path (b\i)++(0:10cm) coordinate (cb\i);
    \draw[thin] (a\i) -- (ca\i)  (b\i) -- (cb\i); }

 \path(0.5,-0.5) coordinate (origin);
 \foreach \i in {1,...,19}
  {
    \path (origin)++(0:1*\i cm)  coordinate (a\i);
    \path (origin)++(-90:1*\i cm)  coordinate (b\i);
    \path (a\i)++(-90:1cm) coordinate (ca\i);
        \path (ca\i)++(-90:1cm) coordinate (cca\i);
    \path (b\i)++(0:1cm) coordinate (cb\i);
    \path (cb\i)++(0:1cm) coordinate (ccb\i);
  }} \end{scope}

  \path(3,-3) coordinate (origin);
 \begin{scope}{\draw[very thick](origin)--++(0:2)--++(-90:1)--++(180:1)--++(-90:1)--++(180:1)--++(90:2);
      \path(origin)--++(0:0.5)--++(-90:0.5)  node {1};
       \path(origin)--++(0:1.5)--++(-90:0.5)node {1};
            \path(origin)--++(0:0.5)--++(90:-1.5)node {3};

  \clip (origin)--++(0:1)--++(-90:1)--++(180:1)--++(-90:1)--++(180:1)--++(90:1);
  \path(origin) coordinate (origin);
   \foreach \i in {1,...,19}
  {
    \path (origin)++(0:1*\i cm)  coordinate (a\i);
    \path (origin)++(-90:1*\i cm)  coordinate (b\i);
    \path (a\i)++(-90:10cm) coordinate (ca\i);
    \path (b\i)++(0:10cm) coordinate (cb\i);
    \draw[thin] (a\i) -- (ca\i)  (b\i) -- (cb\i); }

 \path(0.5,-0.5) coordinate (origin);
 \foreach \i in {1,...,19}
  {
    \path (origin)++(0:1*\i cm)  coordinate (a\i);
    \path (origin)++(-90:1*\i cm)  coordinate (b\i);
    \path (a\i)++(-90:1cm) coordinate (ca\i);
        \path (ca\i)++(-90:1cm) coordinate (cca\i);
    \path (b\i)++(0:1cm) coordinate (cb\i);
    \path (cb\i)++(0:1cm) coordinate (ccb\i);
  }} \end{scope}

 \draw[very thick] (-0.5,0.5)--(8.5,0.5)--(8.5,-2.5)--(5.5,-2.5)--(5.5,-5.5)--(-0.5,-5.5)--(-0.5,0.5);
 \draw[very thick] (5.5,-2.5)--(-0.5,-2.5);
  \draw[very thick] (5.5,-2.5)--(5.5,0.5);
    \draw[very thick] (2.5,-5.5)--(2.5,0.5);
   \end{tikzpicture}   $$
   \caption{Two plethystic semistandard tableaux of shape $ {(2,1)}^{(3,2)}$.
   The former has weight $(9,2,3,1)$ and the latter has weight $(9,5,1)$.
   The latter is maximal in the dominance ordering; the former is not.   }
   \label{pleth-tab}
\end{figure}
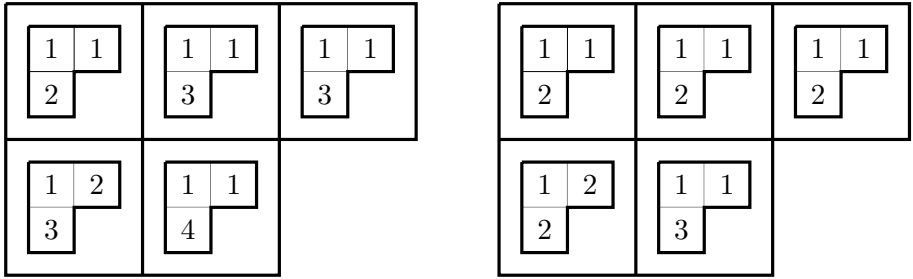

\!\!
\begin{thm}[{\cite[Theorem 1.5]{BPW}}]\label{PW}The maximal partitions $\alpha$ in the dominance order such that $s_\alpha$ is a constituent of $s_\nu \circ s_\mu$ are precisely the maximal weights of the plethystic semistandard tableaux of shape $\mu^\nu$. Moreover, if $\alpha$ is such a maximal partition then $p( \nu,\mu,\alpha)=|\PStd(\mu^\nu, \alpha)|$.
\end{thm}

More generally,  to calculate $p(\nu,\mu,\al ) = \langle s_\nu\circ s_\mu \mid s_\alpha\rangle$ we can proceed by induction on the dominance order (using \cref{plethysm3,plethysm1}).
The following proposition is implicit in \cite{BPW} and can be thought of as the plethystic analogue of
Dvir's recursive method for calculating  Kronecker coefficients \cite{11}
 (as both proceed iteratively by induction along the dominance ordering and cancelling earlier terms).

\begin{prop}
For $\mu$, $\nu$, $\alpha$ an arbitrary triple of partitions, we have that
 \begin{equation}\label{dvir}
 p(\nu,\mu,\al )
=
 |{\rm PStd}(\mu^\nu,\alpha)| -
 \sum_{\beta\rhd \alpha}
 p(\nu,\mu,\beta )
\times |{\rm SStd}( \beta,\alpha)|,
\end{equation}
where the sum  can be restricted to the set of all partitions $\beta\rhd \alpha$
which are less than or equal to
 $\maxp_\succ(\nu,\mu)$
in the lexicographic ordering.
\end{prop}
\begin{proof}
\color{black} Our algorithm is simply an example of what is known as ``highest weight theory".  
We suppose that $f(x_1,x_2,\dots)$ is a symmetric function with integer coefficients which we wish to write in terms of the basis of Schur functions.  
We define the highest weight in    $f(x_1,x_2,\dots)$  to be the
 term $c_\la x^\la$ for some $c_\la\in \ZZ\setminus \{0\}$ for which 
 the partition
\color{black} 
  $\la$ is maximal in the lexicographic ordering (using the notation of  \cref{plethysm3}).  
We claim that the existence of this highest weight implies that  $s_\la$ appears in $f(x_1,x_2,\dots)$ with multiplicity $c_\la$.
To see this, simply note that
\begin{itemize}[leftmargin=*]
\item     if there exists some $s_\mu$ appearing with non-zero coefficient in $f(x_1,\dots, x_n)$  such that $\mu>\la$ in the lexicographic ordering, then 
$s_\mu=x^\mu+\dots$   (by \cref{plethysm3})  which contradicts our maximality assumption on $\la$; 
\item  the highest weight term $x^\la$ cannot appear in any $s_\mu$ for $\mu< \la$ in the lexicographic ordering (by \cref{plethysm3})  as $\SStd(\mu,\la)= \emptyset$ in this case
\color{black}
\end{itemize}
and thus we deduce the existence of the  term  $c_\la s_\la$ in the expansion of $f(x_1,\dots, x_n)$ in the basis of Schur functions.  
One then repeats the above argument for the symmetric function $f(x_1,\dots, x_n)- c_\la s_\la$ et cetera.  
This argument works for any symmetric function, in particular if we set 
\begin{equation}\label{fkghsdjkfgh}f(x_1,x_2,\dots) =
s_\nu\circ s_\mu = \sum_{\alpha} |{\rm PStd}(\mu^\nu,\alpha)|x^\alpha\end{equation}
as in \cref{plethysm1}, then the coefficients are indeed given by the number of relevant plethystic tableaux (appearing in the righthand-side of \cref{fkghsdjkfgh})
 minus the relevant number of semistandard tableaux (appearing in the definition of the Schur function, see   \cref{plethysm3}).  
\end{proof}

This is not an efficient as a general algorithm, however, we  focus on partitions $\alpha$ that are {\em nearly} maximal in the dominance ordering -- this makes calculations manageable.

 \section{The products on the list are multiplicity-free}\label{easyhalf}

In this section we prove that every product on the list is, indeed, multiplicity-free.  For the finite list of exceptional products, this is easily done by computer calculation.
However, the infinite families
require some work.  The ones on our list are $(i)$
$\nu \vdash 2$ and $\mu$ 
 a rectangle or 
\color{black}
an almost rectangle
(i.e., it differs from a rectangle at most by one box)
or a hook, and $(ii)$  $\mu\vdash 2$ and $\nu $ linear.  The  latter case is well-known to be multiplicity-free,  see
\cref{size2,maxminall}.
We have that
\begin{equation}\label{size2}
\langle s_{(n)}\circ s_{(2)}\mid s_\alpha \rangle =
\langle
s_{(n)}\circ s_{(1^2)}
\mid s_{\alpha^T} \rangle =
\begin{cases}
1	&\text{if $\alpha$ has only even parts}		\\
0&\text{otherwise.}
\end{cases}
\end{equation}
 In particular, $p((n),\mu)=1$ for all $n\in \NN$, $\mu\vdash 2$.

 Given $\beta$ a partition of $n$ with distinct parts, we let
$ss[\beta]$ denote the shift symmetric  partition of $2n$  whose leading diagonal hook-lengths  are $2\beta_1, \dots, 2\beta_{\ell(\beta)}$  and whose $i\textsuperscript{th}$ row has length $\beta_i+i$ for  $1\leq i \leq \ell(\beta)$.
We have that
\begin{equation}\label{maxminall}
\langle s_{(1^n)}\circ s_{(2)}\mid s_\alpha \rangle =
\langle
s_{(1^n)}\circ s_{(1^2)}
\mid s_{\alpha^T} \rangle =
\begin{cases}
1	&\text{$\alpha=ss[\beta]$ for some $\beta\vdash n$}		 \\
0&\text{otherwise.}
\end{cases}
\end{equation}
In particular, $p((1^n),\mu)=1$ for all $n\in \NN$ and $\mu\vdash 2$.
Thus case $(ii)$ is covered.

\color{black}
We remark that  the product $ s_\mu \boxtimes s_\mu $ is the character of the tensor square of a simple representation, $\Delta(\mu)$,  of the general linear group. 
 Any tensor square can be decomposed into its symmetric and antisymmetric parts.
 As noted in  \cite[Introduction]{MR1331743},  
 this symmetric/anti-symmetric decomposition of   characters for general linear groups provides us with the well-known identity 
\begin{equation}\label{hghghfkdi}
 s_\mu \boxtimes s_\mu = s_{(2)}\circ s_\mu +
 s_{(1^2)}\circ s_\mu
\end{equation}
 where the first (respectively second) summand is the character of the symmetric (respectively antisymmetric) summand of the tensor product of characters.  
 
\color{black}

\begin{prop}\label{rectangle is mf}
If $\nu \vdash 2$ and $\mu$  is a rectangle, then
 $p(\nu,\mu)=1$.
\end{prop}
\begin{proof}
We have seen that $s_\mu \boxtimes s_\mu$ is multiplicity-free for $\mu$ a rectangle by \cref{thm:mf-outer}. 
 and so the result follows by \cref{hghghfkdi}.
\end{proof}

The remaining products  do  not correspond to summands of products of the form
$s_\mu \boxtimes s_\mu$ on Stembridge's list.   Therefore, we need to show that
these products have maximal multiplicity 2, and when
$$\langle s_\mu \boxtimes s_\mu \mid s_\alpha \rangle=2$$
for some partition $\alpha$, then this coefficient 2 splits into two separate pieces:
\begin{equation}\label{splitters}
  \langle s_{(2)}\circ s_\mu  \mid s_\alpha \rangle=1
\qquad
\text{and} \qquad
 \langle  s_{(1^2)}\circ s_\mu \mid s_\alpha \rangle=1.
\end{equation}
In order to do this, we will require Carr\'e--Leclerc's
 ``domino--Littlewood--Richardson tableaux" algorithm  \cite{MR1331743}  for calculating the decomposition of the products
  $
   s_{(2)}\circ s_\mu    $ and $s_{(1^2)}\circ s_\mu.  $  Given
$\la$ a partition of $n$, we let $[\la]^{2\times2}  $ denote the partition of $4n$ obtained by 
 first doubling the length of every row and then doubling the length of each column.
\color{black}
We define a domino diagram of shape $\la$ as a tiling of $[\la]^{2\times2}$ by means of $2 \times 1$ or $1 \times 2$ rectangles called dominoes.
The {\sf spin-type} of a domino diagram is defined to be half of the total number of $(2)$-dominoes (which is always an integer) modulo $2$.
A domino tableau of shape $\la$ is obtained by labelling  each  domino of the diagram
by a natural number.
We say that the domino tableau is semistandard if these
numbers are weakly increasing along the rows (from left to right), and strictly increasing down the columns.
 Examples are depicted in \cref{0spin,1spin}.

 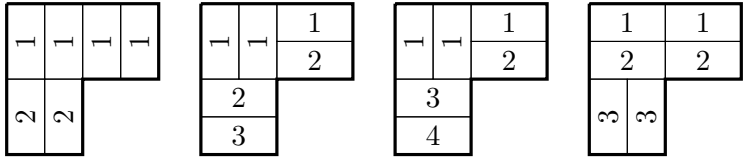
\begin{figure}[ht!]
  \begin{tikzpicture} [scale=1]
 \draw[very thick](0,0)--++(0:2)--++(-90:1)--++(180:1)--++(-90:1)--++(180:1)--++(90:2);

  \clip(0,0)--++(0:2)--++(-90:1)--++(180:1)--++(-90:1)--++(180:1)--++(90:2);

 \draw(0,0) rectangle (0.5,-1) node [midway, rotate=90] {1};

 \draw(0.5,0) rectangle (1,-1) node [midway, rotate=90] {1};

  \draw(1,0) rectangle (1.5,-1) node [midway, rotate=90] {1};

 \draw(1.5,0) rectangle (2,-1) node [midway, rotate=90] {1};

 \draw(0,-1) rectangle (0.5,-2) node [midway, rotate=90] {2};
 \draw(0.5,-1) rectangle (1,-2) node [midway, rotate=90] {2};

 \end{tikzpicture}
\quad
  \begin{tikzpicture} [scale=1]
 \draw[very thick](0,0)--++(0:2)--++(-90:1)--++(180:1)--++(-90:1)--++(180:1)--++(90:2);

  \clip(0,0)--++(0:2)--++(-90:1)--++(180:1)--++(-90:1)--++(180:1)--++(90:2);

 \draw(0,0) rectangle (0.5,-1) node [midway, rotate=90] {1};

 \draw(0.5,0) rectangle (1,-1) node [midway, rotate=90] {1};

  \draw(1,0) rectangle (2,-0.5) node [midway] {1};

 \draw(1,-0.5) rectangle (2,-1) node [midway] {2};

 \draw(0,-1) rectangle (1,-1.5) node [midway] {2};
 \draw(0,-1.5) rectangle (1,-2) node [midway] {3};

 \end{tikzpicture}
\quad
  \begin{tikzpicture} [scale=1]
 \draw[very thick](0,0)--++(0:2)--++(-90:1)--++(180:1)--++(-90:1)--++(180:1)--++(90:2);

  \clip(0,0)--++(0:2)--++(-90:1)--++(180:1)--++(-90:1)--++(180:1)--++(90:2);

 \draw(0,0) rectangle (0.5,-1) node [midway, rotate=90] {1};

 \draw(0.5,0) rectangle (1,-1) node [midway, rotate=90] {1};

  \draw(1,0) rectangle (2,-0.5) node [midway] {1};

 \draw(1,-0.5) rectangle (2,-1) node [midway] {2};

 \draw(0,-1) rectangle (1,-1.5) node [midway] {3};
 \draw(0,-1.5) rectangle (1,-2) node [midway] {4};

 \end{tikzpicture}
 \quad
  \begin{tikzpicture} [scale=1]
 \draw[very thick](0,0)--++(0:2)--++(-90:1)--++(180:1)--++(-90:1)--++(180:1)--++(90:2);

  \clip(0,0)--++(0:2)--++(-90:1)--++(180:1)--++(-90:1)--++(180:1)--++(90:2);

 \draw(0,0) rectangle (1,-0.5) node [midway  ] {1};
 \draw(0,-0.5) rectangle (1,-1) node [midway] {2};

  \draw(1,0) rectangle (2,-0.5) node [midway] {1};

 \draw(1,-0.5) rectangle (2,-1) node [midway] {2};

 \draw(0,-1) rectangle (0.5,-2) node [midway, rotate=90] {3};
 \draw(0.5,-1) rectangle (1,-2) node [midway, rotate=90] {3};

 \end{tikzpicture}
 \caption{The semistandard domino tableaux of shape $(2,1)$ and  even spin type   satisfying the lattice permutation condition 
  (of Definition~\ref{lattice}).  }
 \label{0spin}
 \end{figure}

 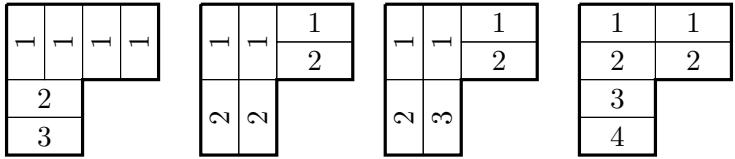
\begin{figure}[ht!]
  \begin{tikzpicture} [scale=1]
 \draw[very thick](0,0)--++(0:2)--++(-90:1)--++(180:1)--++(-90:1)--++(180:1)--++(90:2);

  \clip(0,0)--++(0:2)--++(-90:1)--++(180:1)--++(-90:1)--++(180:1)--++(90:2);

 \draw(0,0) rectangle (0.5,-1) node [midway, rotate=90] {1};

 \draw(0.5,0) rectangle (1,-1) node [midway, rotate=90] {1};

  \draw(1,0) rectangle (1.5,-1) node [midway, rotate=90] {1};

 \draw(1.5,0) rectangle (2,-1) node [midway, rotate=90] {1};

 \draw(0,-1) rectangle (1,-1.5) node [midway] {2};
 \draw(0,-1.5) rectangle (1,-2) node [midway] {3};

 \end{tikzpicture}
\quad  \begin{tikzpicture} [scale=1]
 \draw[very thick](0,0)--++(0:2)--++(-90:1)--++(180:1)--++(-90:1)--++(180:1)--++(90:2);

  \clip(0,0)--++(0:2)--++(-90:1)--++(180:1)--++(-90:1)--++(180:1)--++(90:2);

 \draw(0,0) rectangle (0.5,-1) node [midway, rotate=90] {1};

 \draw(0.5,0) rectangle (1,-1) node [midway, rotate=90] {1};

  \draw(1,0) rectangle (2,-0.5) node [midway] {1};

 \draw(1,-0.5) rectangle (2,-1) node [midway] {2};

 \draw(0,-1) rectangle (0.5,-2) node [midway, rotate=90] {2};
 \draw(0.5,-1) rectangle (1,-2) node [midway, rotate=90] {2};

 \end{tikzpicture}\quad  \begin{tikzpicture} [scale=1]
 \draw[very thick](0,0)--++(0:2)--++(-90:1)--++(180:1)--++(-90:1)--++(180:1)--++(90:2);

  \clip(0,0)--++(0:2)--++(-90:1)--++(180:1)--++(-90:1)--++(180:1)--++(90:2);

 \draw(0,0) rectangle (0.5,-1) node [midway, rotate=90] {1};

 \draw(0.5,0) rectangle (1,-1) node [midway, rotate=90] {1};

  \draw(1,0) rectangle (2,-0.5) node [midway] {1};

 \draw(1,-0.5) rectangle (2,-1) node [midway] {2};

 \draw(0,-1) rectangle (0.5,-2) node [midway, rotate=90] {2};
 \draw(0.5,-1) rectangle (1,-2) node [midway, rotate=90] {3};

 \end{tikzpicture} \quad
  \begin{tikzpicture} [scale=1]
 \draw[very thick](0,0)--++(0:2)--++(-90:1)--++(180:1)--++(-90:1)--++(180:1)--++(90:2);

  \clip(0,0)--++(0:2)--++(-90:1)--++(180:1)--++(-90:1)--++(180:1)--++(90:2);

 \draw(0,0) rectangle (1,-0.5) node [midway  ] {1};
 \draw(0,-0.5) rectangle (1,-1) node [midway] {2};

  \draw(1,0) rectangle (2,-0.5) node [midway] {1};

 \draw(1,-0.5) rectangle (2,-1) node [midway] {2};

 \draw(0,-1) rectangle (1,-1.5) node [midway] {3};
 \draw(0,-1.5) rectangle (1,-2) node [midway] {4};

 \end{tikzpicture}
\caption{The semistandard domino tableaux of shape $(2,1)$ and  odd spin type   satisfying the lattice permutation condition. }
 \label{1spin}
\end{figure}

\color{black}

We associate to a domino tableau, $\SSTT$, of shape $\la$ (as above) a Young tableau, $\stt$, 
of shape  $[\la]^{2\times2}$ in the following way.
   Given a domino $\{(r,c),(r,c+1)\}$ (respectively $\{(r,c),(r+1,c)\}$) labelled by
   $i\in\NN$, we write $\stt(r,c)=i$ and
   $\stt(r,c+1)=i$ (respectively $\stt(r,c)=i$ and $\stt(r+1,c)=i$).
For $k\in \NN$, we let $$t_k = \tfrac{1}{2}|\{ (r,c)   	 \in  [\la]^{2\times2}  \mid \stt (r,c)=k\}|.$$
 We refer to  $\alpha=(t_1, t_2, t_3,\dots)$ as the {\sf weight} of the domino tableau~$\SSTT$.
This is illustrated in \cref{anotherfig1}.

\begin{figure}[ht!]\color{black}
$$
\gyoung(1;1;1;1,1;1;1;1,2;2,3;3)
\qquad
\gyoung(1;1;1;1,1;1;2;2,2;2,2;2)
\qquad
\gyoung(1;1;1;1,1;1;2;2,2;3,2;3)\qquad
\gyoung(1;1;1;1,2;2;2;2,3;3,4;4)
$$

\caption{\color{black}The Young tableaux associated to the domino tableaux of \cref{1spin}.  
Thus the domino  tableaux of \cref{1spin} have weights 
$(4,1,1)$, $(3,3)$, $(3,2,1)$, and $(2^2,1^2)$ respectively.  }
\label{anotherfig1}
\end{figure}

The following definition of good and bad nodes (and lattice permutations) is due to G. D. James and can be found in his original  characteristic-free proof of the Littlewood--Richardson rule in the setting of the symmetric group \cite[4.5 Definition]{j77}.  
 This definition is slightly more complicated than the usual definition of a lattice permutation found in, for example, Sagan's book \cite[Definition 4.9.3]{MR1824028}.  
 This definition keeps track of much more information (it can be seen as a pre-cursor to the theory of crystals) and this information will be needed in our arguments later on in the paper (in particular, we will need to specify a given ``bad node" in a sequence).  
An equivalent formulation of a ``bad node" (see below) is that of a ``Bad Guy" as given within Stembridge's  proof of   \cite[Theorem]{MR1912814} and the reader is invited to 
 use Stembridge's definition if this appeals more to their tastes.  
 In what follows, we will use the grammatical rule for pairing nested parentheses (that is, we proceed from the innermost pairing to the outermost pairing) but
 we tweak  this rule slightly by not requiring that the number of opening parentheses is equal to the number of closing parentheses (any such additional parentheses are left unpaired).  For example in  the following two sequence of parentheses  
 $$(\;\; (\; \;(\;\; )\;\; )\;, \qquad \qquad (\;\;(\;\;)\;\;(\;\;)$$
 the leftmost parenthesis in each sequence is unpaired.  In the first sequence the 3rd and 4th terms are paired and the 2nd and 5th terms are paired.  
 In the second sequence, the 2nd and 3rd terms are paired and the 4th and 5th terms are paired. 
\color{black}

\begin{defn}\label{lattice}
Given a finite sequence, $\Sigma$, of positive integers we let $\Sigma_{(i-1,i)}$ denote the sequence obtained by replacing all occurrences of $i-1$ with an open parenthesis and
all occurrences of $i$ with a closed parenthesis.
We define the quality (good/bad) of each term in $\Sigma$ as follows.
\begin{enumerate}[leftmargin=*]
\item All  terms $1$ are good.
\item A term $i$ is good if and only if  the corresponding closed parenthesis in the sequence
 $\Sigma_{(i-1,i)}$ is partnered with an open parenthesis under the usual rule for nested parentheses.
\end{enumerate}
The  sequence   is  a {\sf lattice permutation} if every term in the sequence is good.
 We shall say the term $i-1$ is supported by the term $i$ whenever they are partnered under the usual rule for parentheses.   \end{defn}

    \begin{eg}
    The following sequence is not a lattice permutation
    $$
    1 , \;
    1 , \; 2 , \; 2 , \; 1 , \; 3 , \; 3 , \; 3 , \; 4 , \; 4 , \; 1 , \;2 , \;4 , \;3,  \;2.
    $$
To see this, we note that the system of parentheses $\Sigma_{(2,3)}$ is   as follows
\begin{align*}
\scalefont{0.9}
   \begin{tikzpicture}[baseline={([yshift=-.7ex]current bounding box.center)},scale=0.55]   \begin{scope}
             \draw (1.5,2.3) arc (0:180:-1.5 and 0.5);
             \draw (0,2.3) arc (0:180:-3 and 0.75);             
                   \draw (13.5,2.3) arc (0:180:-1.5 and 0.5);
                  \draw (0,1.9) node  {$\tiny { ( }$};
          \draw (1.5,1.9) node  {$\tiny { ( }$};
                                \draw (13.5,1.9) node  {$\tiny { ( }$};                \draw (16.5,1.9) node  {$\tiny { ) }$};
                \draw (4.5,1.9) node  {$\tiny { ) }$}; \draw (6,1.9) node  {$\tiny { ) }$};
                 \draw (7.5,1.9) node  {$\tiny { ) }$};         \end{scope}
       \begin{scope}
           \draw (10.5,1.9) node  {$   $};
         \draw (12,1.9) node  {$   $};
                  \draw (13.5,1.9) node  {$   $};
                           \draw (15,1.9) node  {$   $};
                     \draw (16.5,1.9) node  {$   $};
  \draw (-3,1) node  {$\tiny {1}$};
 \draw (-1.5,1) node  {$\tiny {1}$};
  \draw (0,1) node  {$\tiny {2}$};
    \draw (1.5,1) node  {$\tiny {2}$};
      \draw (3,1) node  {$\tiny {1}$};
  \draw (4.5,1) node  {$\tiny { 3}$};
    \draw (6,1) node  {$\tiny { 3}$};
        \draw (7.5,1) node  {$\tiny { 3}$};
      \draw (9,1) node  {$\tiny {4}$};
  \draw (10.5,1) node  {$\tiny { 4}$};
    \draw (12,1) node  {$\tiny { 1}$};
      \draw (13.5,1) node  {$\tiny {2}$};
  \draw (15,1) node  {$\tiny {4}$};
    \draw (16.5,1) node  {$\tiny { 3}$};
        \draw (18,1) node  {$\tiny { 2}$};
            \draw (18,2) node  {$\tiny {(}$};
    \end{scope}
 \end{tikzpicture}
  \end{align*}
\color{black} Thus the 8th integer in the sequence is bad.
    \end{eg}

\begin{defn}\color{black}
We define the {\sf  reading word} $R(\SSTT)$  of a semistandard  tableau  $\SSTT\in \SStd(\nu\setminus \la,\mu)$  to be given by reading the labels of the boxes from top-to-bottom down columns from right-to-left.
We let ${\rm LR}(\nu\setminus \la,\mu)\subseteq \SStd(\nu\setminus \la,\mu)$ denote the set of tableaux whose reading words satisfy the lattice permutation property;  
we refer to such tableaux as {\sf Littlewood--Richardson tableaux}.  

    \end{defn}

\begin{defn}

We define the {\sf  reading word} $R(\SSTT)$  of a domino tableau  $\SSTT$  to be given by reading the labels of the dominoes from top-to-bottom down columns from right-to-left and recording each label exactly once --- as late as possible --- in other words, for a horizontal domino $\{(r,c),(r,c+1)\}$ we record the label upon reading column~$c$.
We say that  a semistandard domino tableau satisfies the {\sf lattice permutation condition} if
  the reading word  is a lattice permutation.
  We let  ${\sf Dom}(\la,\alpha)$     denote the set of all semistandard   domino \color{black} tableaux of shape $\la$ and   weight $\alpha$ satisfying the lattice permutation condition;
  we refer to such tableaux as {\sf Littlewood--Richardson domino tableaux}.  
   We set
  $\dom (\la,\al) = |{\sf Dom}(\la,\alpha)|$, and let $\dom_+(\la,\al)$ and $\dom_-(\la,\al)$
  count the corresponding tableaux of even and odd spin type, respectively.

    \end{defn}

\begin{eg}
The  reading words of the domino tableaux in \cref{0spin} are
$$
(1,1,1,2,1,2) \quad
(1,2,1,  1 ,2,3)\quad
(1,2,1,1, 3,4)\quad
(1,2,3,1,2,3)
$$
and so all the tableaux of \cref{0spin} satisfy the lattice permutation condition.
\end{eg}

{\color{black}
\begin{thm}[Littlewood--Richardson]
 We have that
 $\langle   s_\la\boxtimes s_\mu  \mid s_\nu\rangle =|{\rm LR}(\nu\setminus \la,\mu)|$.  
\end{thm}
}

\begin{thm}[Carr\'e--Leclerc \cite{MR1331743}]\label{gom-tab}
We have that
 $\langle   s_\mu\boxtimes s_\mu  \mid s_\alpha\rangle $
is the number $\dom(\mu,\al)$
of semistandard domino tableaux of shape $\mu$ and weight $\alpha$ satisfying the
lattice permutation condition.  This number decomposes as
$$\langle s_{(2)}\circ s_\mu \mid s_\alpha\rangle + \langle s_{(1^2)}\circ s_\mu \mid s_\alpha\rangle$$
where the former (respectively latter) summand is
equal to the number $\dom_+(\mu,\al)$ (respectively $\dom_-(\mu,\al)$) of tableaux of even (respectively odd) spin type.
\end{thm}

Now, using Carr\'e--Leclerc's refinement of the Littlewood--Richardson rule, we are able (without much ado)  to calculate the multiplicity-free plethystic products $s_{(2)}\circ s_\mu$ and  $s_{(1^2)}\circ s_\mu$  for $\mu $ a hook.

 \begin{prop}
 For $\mu\vdash m$   a hook,  $s_{(2)}\circ s_\mu$ and 
  $s_{(1^2)}\circ s_\mu$
are both multiplicity-free.
 \end{prop}

 \begin{proof}
\color{black}
The coefficient $  \langle s_\mu \boxtimes s_\mu \mid s_\alpha \rangle  $ are {\em not} multiplicity-free in general (we will see that they can be equal to 0, 1, or 2). Thus, following the   discussion around \cref{splitters}, this proof will then proceed to show that the  coefficients   equal to 2 split as $2=1+1$ with multiplicity 1 in each of  
 $s_{(2)}\circ s_\mu$ and 
  $s_{(1^2)}\circ s_\mu$ using \cref{gom-tab} and the combinatorics of domino tableaux.  

Firstly, for $\mu=(a,1^b)$   we have (by the Littlewood--Richardson rule) that 
 \begin{align} \label{=2}
 \langle s_{(a,1^b)} \boxtimes s_{(a,1^b)} \mid s_\alpha \rangle=
 	 \begin{cases}
2 &\text{ if $\alpha_1+\alpha_2=2a+1$ and $dl(\alpha)= 2$} \\
1 &\text{ if $\alpha_1+\alpha_2=2a \text{ or }2a+2$ and ${dl}(\alpha)= 2$} \\
1 &\text{ if $\alpha =(2a,1^{2b})$ or $(2a-1,1^{2b+1})$} \\
0 &\text{ otherwise}.  
   \end{cases}  
\end{align}
In more detail, for  $\alpha$ with $\alpha_1+\alpha_2=2a+1$ and 
 Durfee size
 $dl(\alpha)=2$, 
  the entry $1$ can be placed in either the box $(2,\alpha_2)\in \alpha$ or  
$(b+1,1)\in \alpha$ and all other entries are forced by the semistandard and lattice permutation conditions;   
 we depict indicative examples in \cref{2hooksy}.

 \color{black}
 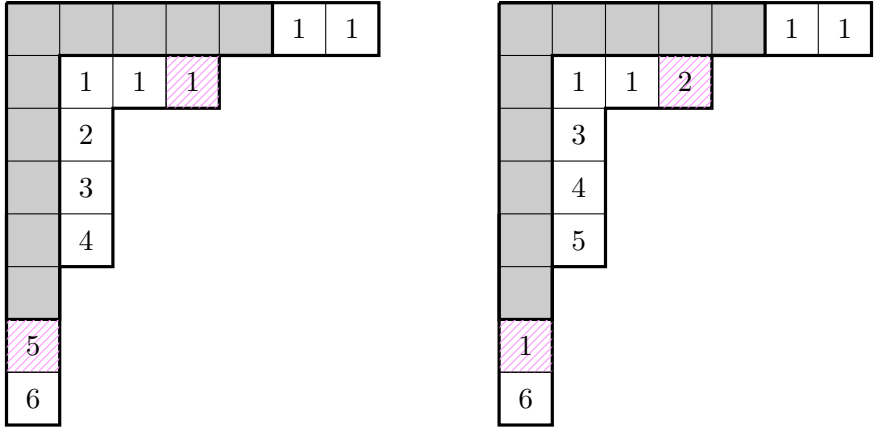
\begin{figure}[ht!]
 $$    \begin{tikzpicture} [scale=0.7]

   \draw[very thick,black,fill=gray!40 ]
   (0,0)--++(0:5)--++(-90:1)
   --++(180:4)--++(-90:5)      --++(180:1)--++(90:6) ;

  \draw[very thick,black,   ]
   (0,0)--++(0:7)--++(-90:1)
   --++(180:3)--++(-90:1)--++(180:2) --++(-90:3)--++(180:1) --++(-90:3) --++(180:1) --++(90:4) ;

\clip   (0,0)--++(0:7)--++(-90:1)
   --++(180:3)--++(-90:1)--++(180:2) --++(-90:3)--++(180:1) --++(-90:3) --++(180:1) --++(90:4) ;

\foreach \i in {0,1,2,3,4,5,6,7,8,9}
{\draw(0,-\i) --++(0:2);
\draw(\i,0)--++(-90:2); }

  \draw[pattern=north east lines, pattern color=magenta!40](3,-1)--++(0:1)--++(-90:1)
  --++(180:1)--++(90:1);

  \draw[pattern=north east lines, pattern color=magenta!40](0,-6)--++(0:1)--++(-90:1)
  --++(180:1)--++(90:1);

\draw(5.5,-0.5)  node {1};
\draw(6.5,-0.5)  node {1};
\draw(1.5,-1.5)  node {1};
\draw(2.5,-1.5)  node {1};

\draw(3.5,-1.5)  node {1};
\draw(1.5,-2.5)  node {2};
\draw(1.5,-3.5)  node {3};
\draw(1.5,-4.5)  node {4};

\draw(0.5,-6.5)  node {5};
\draw(0.5,-7.5)  node {6};

\end{tikzpicture}
\qquad
\qquad
  \begin{tikzpicture} [scale=0.7]

   \draw[very thick,black,fill=gray!40 ]
   (0,0)--++(0:5)--++(-90:1)
   --++(180:4)--++(-90:5)      --++(180:1)--++(90:6) ;

  \draw[very thick,black,   ]
   (0,0)--++(0:7)--++(-90:1)
   --++(180:3)--++(-90:1)--++(180:2) --++(-90:3)--++(180:1) --++(-90:3) --++(180:1) --++(90:4) ;

\clip   (0,0)--++(0:7)--++(-90:1)
   --++(180:3)--++(-90:1)--++(180:2) --++(-90:3)--++(180:1) --++(-90:3) --++(180:1) --++(90:4) ;

\foreach \i in {0,1,2,3,4,5,6,7,8,9}
{\draw(0,-\i) --++(0:2);
\draw(\i,0)--++(-90:2); }

  \draw[pattern=north east lines, pattern color=magenta!40](3,-1)--++(0:1)--++(-90:1)
  --++(180:1)--++(90:1);

  \draw[pattern=north east lines, pattern color=magenta!40](0,-6)--++(0:1)--++(-90:1)
  --++(180:1)--++(90:1);

\draw(5.5,-0.5)  node {1};
\draw(6.5,-0.5)  node {1};
\draw(1.5,-1.5)  node {1};
\draw(2.5,-1.5)  node {1};

\draw(3.5,-1.5)  node {2};
\draw(1.5,-2.5)  node {3};
\draw(1.5,-3.5)  node {4};
\draw(1.5,-4.5)  node {5};

\draw(0.5,-6.5)  node {1};
\draw(0.5,-7.5)  node {6};

\end{tikzpicture}$$
 \caption{ \color{black} 
The elements of $ {\rm LR}(\alpha\setminus \mu,\mu)$ for $\mu=(5,1^5)$ and 
$\alpha=(7,4,2^3,1^3)$.  Note that $\alpha_1+\alpha_2=7+4= 2\times 5+1=2a+1$ and $dl(\alpha)=2$. The only choice is  which pink box we place  the final entry 1 (namely the rightmost box of the second row   or the topmost box of the first column).  All other entries are forced by this choice.} 
 \label{2hooksy}
 \end{figure}

 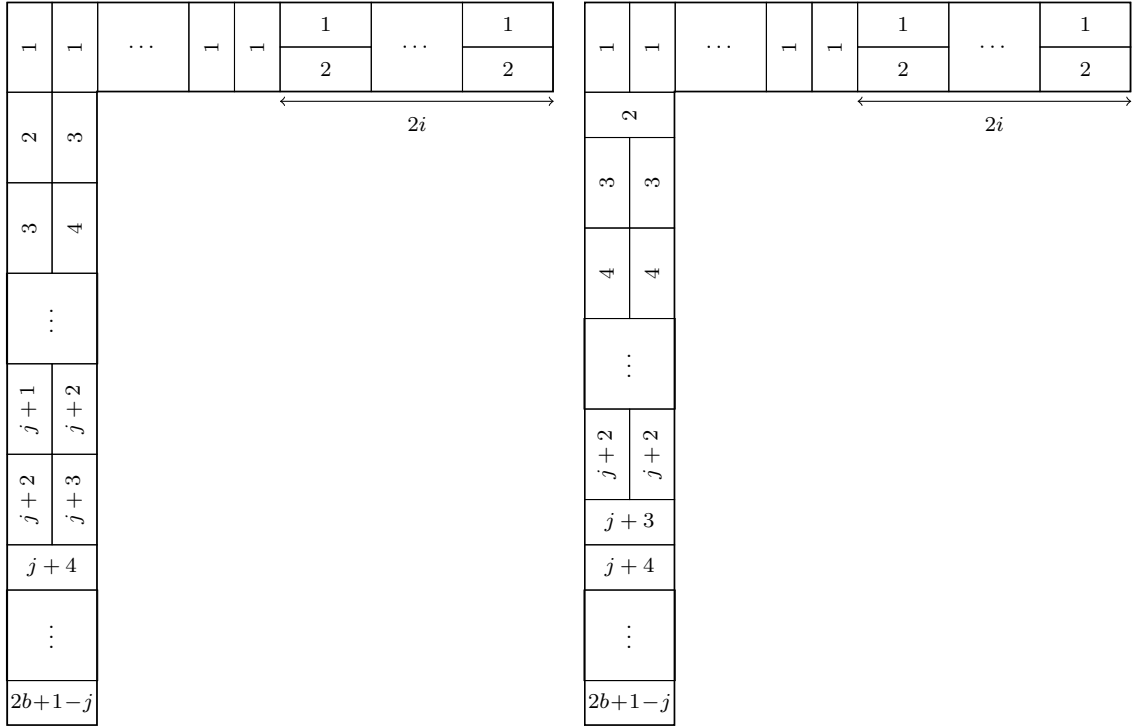
\begin{figure}[ht!]
$$  \begin{tikzpicture} [scale=0.6]
\scalefont{0.7}

      \draw(-2,-6) rectangle  node[midway,rotate=90] {$ \cdots$}(0,-8) ;
       \draw(-2,-13) rectangle  node[midway,rotate=90] {$ \cdots$}(0,-15) ;

 \draw[<->] (4,-2.2)-- (10,-2.2);
 \draw  (7,-2.7) node {$2i$};

\path(0,0) coordinate (origin);

    \clip (-2,0)--++(0:12)--++(-90:2)--++(180:10)--++(-90:14)--++(180:2)--++(90:16);

 \draw (6,0) rectangle  node[midway] {$ \cdots$}(8,-2) ; 
      \draw(0,0) rectangle  node[midway] {$ \dots$}(2,-2) ;

  \draw[very thick]  (-2,0)--++(0:12)--++(-90:2)--++(180:10)--++(-90:14)--++(180:2)--++(90:16);

  \draw(10,-1) rectangle  node[midway] {2}(12,-2) ;
  \draw(8,-1) rectangle  node[midway] {2}(10,-2) ;

  \draw(4,-1) rectangle  node[midway] {2}(6,-2) ;

   \draw(4,-1) rectangle  node[midway] {1}(6,0) ;

  \draw(10,-1) rectangle  node[midway] {1}(12,0) ;
  \draw(8,-1) rectangle  node[midway] {1}(10,0) ;
  \draw(-2,0) rectangle  node[midway, rotate=90] {1}(-1,-2) ;
    \draw(0,0) rectangle  node[midway, rotate=90] {1}(-1,-2) ;
    \draw(3,0) rectangle  node[midway, rotate=90] {1}(4,-2) ;
     \draw(3,0) rectangle  node[midway, rotate=90] {1}(2,-2) ;

    \draw(-2,-2) rectangle  node[midway, rotate=90]  {2}(-1,-4) ;
    \draw(-1,-2) rectangle  node[midway, rotate=90]  {3}(0,-4) ;\draw(-2,-4) rectangle  node[midway, rotate=90]  {3}(-1,-6) ;

      \draw(-1,-6) rectangle  node[midway, rotate=90]  {4}(-0,-4) ;

        \draw(-1,-8) rectangle  node[midway, rotate=90]  {$j+1$}(-2,-10) ;
                \draw(-1,-8) rectangle  node[midway, rotate=90]  {$j+2$}(0,-10) ;

                      \draw(-1,-12) rectangle  node[midway, rotate=90]  {$j+2$}(-2,-10) ;
                \draw(-1,-10) rectangle  node[midway, rotate=90]  {$j+3$}(0,-12) ;

                  \draw(0,-12) rectangle  node[midway]  {$j+4$}(-2,-13) ;

                       \draw(0,-15) rectangle  node[midway]  {$2b\!+\!1\!-\!j$}(-2,-16) ;

   \end{tikzpicture}  \quad 
   \begin{tikzpicture} [scale=0.6]
\scalefont{0.7}

     \draw(-2,-7) rectangle  node[midway,rotate=90] {$ \cdots$}(0,-9) ;
     \draw(-2,-13) rectangle  node[midway,rotate=90] {$ \cdots$}(0,-15) ;

 \draw[<->] (4,-2.2)-- (10,-2.2);
 \draw  (7,-2.7) node {$2i$};

\path(0,0) coordinate (origin);

    \clip (-2,0)--++(0:12)--++(-90:2)--++(180:10)--++(-90:14)--++(180:2)--++(90:16);

 \draw(6,0) rectangle  node[midway] {$ \cdots$}(8,-2) ;

      \draw(0,0) rectangle  node[midway] {$ \dots$}(2,-2) ;

  \draw[very thick]  (-2,0)--++(0:12)--++(-90:2)--++(180:10)--++(-90:14)--++(180:2)--++(90:16);

  \draw(10,-1) rectangle  node[midway] {2}(12,-2) ;
  \draw(8,-1) rectangle  node[midway] {2}(10,-2) ;

  \draw(4,-1) rectangle  node[midway] {2}(6,-2) ;

   \draw(4,-1) rectangle  node[midway] {1}(6,0) ;

  \draw(10,-1) rectangle  node[midway] {1}(12,0) ;
  \draw(8,-1) rectangle  node[midway] {1}(10,0) ;
  \draw(-2,0) rectangle  node[midway, rotate=90] {1}(-1,-2) ;
    \draw(0,0) rectangle  node[midway, rotate=90] {1}(-1,-2) ;
    \draw(3,0) rectangle  node[midway, rotate=90] {1}(4,-2) ;
     \draw(3,0) rectangle  node[midway, rotate=90] {1}(2,-2) ;

     \draw(-2,-2) rectangle  node[midway, rotate=90]  {2}(0,-3) ;  \draw(-1,-3) rectangle  node[midway, rotate=90]  {3}(0,-5) ;\draw(-2,-3) rectangle  node[midway, rotate=90]  {3}(-1,-5) ;

     \draw(-1,-4-1) rectangle  node[midway, rotate=90] {4}(-2,-6-1) ;
     \draw(-1,-4-1) rectangle  node[midway, rotate=90]  {4}(-0,-6-1) ;

     \draw(-1,-8-1) rectangle  node[midway, rotate=90] {$j+2$}(-2,-10-1) ;
     \draw(-1,-8-1) rectangle  node[midway, rotate=90] {$j+2$}(-0,-10-1) ;

          \draw(0,-11) rectangle  node[midway] {$j+3$}(-2,-12) ;
          \draw(0,-12) rectangle  node[midway] {$j+4$}(-2,-13) ;

          \draw(0,-15) rectangle  node[midway] {$2b\!+\!1\!-\!j$}(-2,-16) ;

   \end{tikzpicture} $$
   \caption{The two domino Littlewood--Richardson tableaux of shape $(a,1^b)$ and weight a double hook $\alpha=(2a-i,i+1,2^j,1^{2b-1-2j})$ satisfying $\alpha_1+\alpha_2=2a+1$ and $dl(\alpha)=2$. }
   \label{dominotableaux}
\end{figure}
Thus it remains to show that if  $\alpha$ is such that $\langle s_\mu \boxtimes s_\mu \mid s_\alpha \rangle =2$  (in other words, if $\alpha_1+\alpha_2=2a+1$ and $dl(\alpha)=2$) 
then this coefficient can be  ``split" so that 
$$   \langle s_{(2)}  \circ s_\mu \mid s_\alpha \rangle =1 \qquad   \langle s_{(1^2)} \circ s_\mu \mid s_\alpha \rangle =1 $$
as already discussed in \cref{splitters}.  
  \color{black}  
In other words, we need to describe the domino--Littlewood--Richardson tableaux of this form.  Firstly, we can rewrite $\alpha$ in the form $\alpha=(2a-i,i+1,2^j,1^{2b-1-2j})$
for
$i,j \geq 1$.  With this notation fixed, the pair of domino Littlewood--Richardson  tableaux are depicted in \cref{dominotableaux}.  The signs of these tableaux differ (as the total number of $(2)$-dominoes in the former is 2 
  fewer
\color{black}
 than in the latter) and   the result follows.
  \end{proof}

{\color{black} The remainder of this section is dedicated to the proof that $s_{(2)}\circ s_\mu$ and  $s_{(1^2)}\circ s_\mu$ are both multiplicity-free for the almost rectangles  $\mu=(a^b,1)$ and $(a^b,a-1)$.
We begin by considering the case that $\mu$ is a rectangle in more detail: namely, we  construct  the  elements of ${\sf Dom}((a^b), \la)$ explicitly.
 While this information was not needed to prove that  $ p((2), (a^b))=1$ (as we have already seen in  \cref{rectangle is mf}), this serves as a warm up to our construction of  the domino tableaux of shape    $\mu=(a^b,1)$ and $(a^b,a-1)$ and hence will help the reader with  our proofs that $ p((2), (a^b,1))=p((2), (a^b,a-1))=1$.     }

\begin{eg}
\color{black}
For $\mu=(3,1,1)$ we have the following plethysm products 
\begin{align*}
s_{(2)} \circ s_{(3,1,1)}
&=
s_{(3^2 2, 1^2)} + s_{(4, 2, 1^4)} + s_{(4,2^3)} + s_{(4, 3, 1^3)} + s_{(4, 3, 2, 1)} + s_{(4, 4, 2)} + s_{(5, 2, 1^3)} 
\\	& \quad + s_{(5, 2^2, 1)} + s_{(5, 3, 1^2)} + s_{(6, 1^4)} + s_{(6, 2^2)}
\\
 s_{(1^2)} \circ s_{(3,1,1)}		&=
s_{(3^2 1^4)} + s_{(3^2 2^2)} + s_{(4, 2^2, 1^2)} + s_{(4, 3, 1^3)} + s_{(4, 3, 2, 1)} + s_{(4, 4, 1^2)} + s_{(5, 1^4, 1)} \\	& \quad + s_{(5, 2, 1^3)} + s_{(5, 2^2, 1)} + s_{(5, 3, 2)} + s_{(6, 2, 1^2)}
\end{align*}
and the reader is invited to sum these term-wise to obtain the coefficients for the decomposition of 
$ s_{(3,1)}\boxtimes  s_{(3,1)}$ given in \cref{=2}.  
The partitions which label constituents of both products are $(5,2^2,1)$, $(5,2,1^3)$, $(4,3,2,1)$ and $(4,3,1^3)$;   
 the corresponding domino tableaux for the first 2 of these partitions are given in \cref{extra1}.

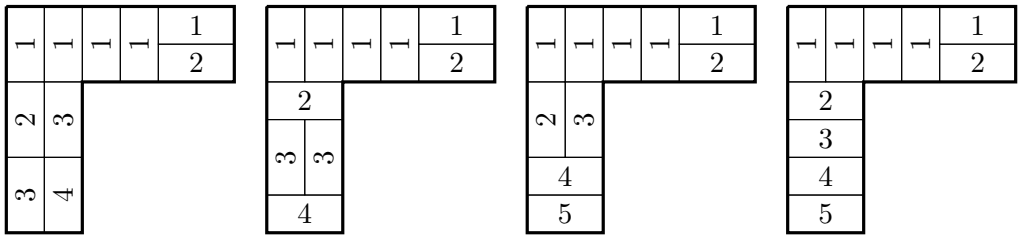
\begin{figure}[ht!]$$
 \begin{tikzpicture} [scale=1]
 \draw[very thick](0,0)--++(0:3)--++(-90:1)--++(180:2)--++(-90:2)--++(180:1)--++(90:3);


 \draw(0,0) rectangle (0.5,-1) node [midway, rotate=90] {1};

 \draw(0.5,0) rectangle (1,-1) node [midway, rotate=90] {1};

  \draw(1,0) rectangle (1.5,-1) node [midway, rotate=90] {1};

 \draw(1.5,0) rectangle (2,-1) node [midway, rotate=90] {1};

 \draw(2,0) rectangle (3,-0.5) node [midway,  ] {1};
 \draw(2,-1) rectangle (3,-0.5) node [midway,  ] {2};

 \draw(0,-1) rectangle (0.5,-2) node [midway, rotate=90] {2};
 \draw(0.5,-1) rectangle (1,-2) node [midway, rotate=90] {3};

 \draw(0,-1-1) rectangle (0.5,-2-1) node [midway, rotate=90] {3};
 \draw(0.5,-1-1) rectangle (1,-2-1) node [midway, rotate=90] {4};

 \end{tikzpicture}
\quad
 \begin{tikzpicture} [scale=1]
 \draw[very thick](0,0)--++(0:3)--++(-90:1)--++(180:2)--++(-90:2)--++(180:1)--++(90:3);


 \draw(0,0) rectangle (0.5,-1) node [midway, rotate=90] {1};

 \draw(0.5,0) rectangle (1,-1) node [midway, rotate=90] {1};

  \draw(1,0) rectangle (1.5,-1) node [midway, rotate=90] {1};

 \draw(1.5,0) rectangle (2,-1) node [midway, rotate=90] {1};

 \draw(2,0) rectangle (3,-0.5) node [midway,  ] {1};
 \draw(2,-1) rectangle (3,-0.5) node [midway,  ] {2};
 
  \draw(0,-1) rectangle (1,-1.5) node [midway ] {2};

 \draw(0,-1-0.5) rectangle (0.5,-2-0.5) node [midway, rotate=90] {3};
 \draw(0.5,-1-0.5) rectangle (1,-2-0.5) node [midway, rotate=90] {3};

   \draw(0,-1-1.5) rectangle (1,-1.5-1.5) node [midway ] {4};

 \end{tikzpicture}
\quad
 \begin{tikzpicture} [scale=1]
 \draw[very thick](0,0)--++(0:3)--++(-90:1)--++(180:2)--++(-90:2)--++(180:1)--++(90:3);


 \draw(0,0) rectangle (0.5,-1) node [midway, rotate=90] {1};

 \draw(0.5,0) rectangle (1,-1) node [midway, rotate=90] {1};

  \draw(1,0) rectangle (1.5,-1) node [midway, rotate=90] {1};

 \draw(1.5,0) rectangle (2,-1) node [midway, rotate=90] {1};

 \draw(2,0) rectangle (3,-0.5) node [midway,  ] {1};
 \draw(2,-1) rectangle (3,-0.5) node [midway,  ] {2};

 \draw(0,-1) rectangle (0.5,-2) node [midway, rotate=90] {2};
 \draw(0.5,-1) rectangle (1,-2) node [midway, rotate=90] {3};

   \draw(0,-1-1.5) rectangle (1,-1.5-1.5) node [midway ] {5};

   \draw(0,-1-1 ) rectangle (1,-1.5-1 ) node [midway ] {4};

 \end{tikzpicture}\quad
 \begin{tikzpicture} [scale=1]
 \draw[very thick](0,0)--++(0:3)--++(-90:1)--++(180:2)--++(-90:2)--++(180:1)--++(90:3);


 \draw(0,0) rectangle (0.5,-1) node [midway, rotate=90] {1};

 \draw(0.5,0) rectangle (1,-1) node [midway, rotate=90] {1};

  \draw(1,0) rectangle (1.5,-1) node [midway, rotate=90] {1};

 \draw(1.5,0) rectangle (2,-1) node [midway, rotate=90] {1};

 \draw(2,0) rectangle (3,-0.5) node [midway,  ] {1};
 \draw(2,-1) rectangle (3,-0.5) node [midway,  ] {2};

 \draw(0,-1) rectangle (1,-1.5) node [midway ] {2};
 \draw(0,-1.5) rectangle (1,-2) node [midway ] {3};

   \draw(0,-1-1.5) rectangle (1,-1.5-1.5) node [midway ] {5};

   \draw(0,-1-1 ) rectangle (1,-1.5-1 ) node [midway ] {4};

 \end{tikzpicture}
 $$
 \caption{\color{black} The  domino tableaux   of shape $(3,1^2)$ and weight $(5,2^2,1)$ and $(5,2,1^3)$.  Compare with \cref{dominotableaux}.   }
 \label{extra1}
 \end{figure}
\end{eg}

\begin{defn}\label{ineedalabel}
 \renewcommand{\alpha}{\lambda}
\renewcommand{\theta}{{\widehat{\la}}}
\renewcommand{\vartheta}{{\widehat{\la}}}

Let $\vartheta=(\vartheta_1,\dots,\vartheta_\ell) \subseteq (a^b)$ be a  partition with
 $\ell =\ell(\vartheta) \leq b$. 
We let $\SSTT^\vartheta$ be the domino tableau constructed in two steps:
\begin{itemize}[leftmargin=*]
\item tile in the region  $[\vartheta]^{2\times 2} $ with unlabelled $(1^2)$-dominoes and the region
		 \linebreak $  [(a^b)]^{2\times 2}\setminus [\vartheta]^{2\times 2}   $ with unlabelled $(2)$-dominoes.
\item label the dominoes down each column with consecutive integers 	beginning with 1.  	
\end{itemize}
We refer to $\SSTT^\vartheta$ as the {\sf admissible} tableau for $\vartheta$ (or simply the   {\sf admissible} $\vartheta$-tableau),
\color{black}
and we call the partition $\vartheta$ a {\sf rectangular weight}.
\color{black}
  \end{defn}

 \begin{figure}[ht!]
$$    \begin{tikzpicture} [scale=1]

   \draw[pattern=north east lines, pattern color=magenta!40](0,0)--++(0:4)--++(-90:1)--++(180:2)--++(-90:1)--++(180:1) --++(-90:1)--++(180:1)  --++(90:3);
 \draw[very thick](0,0)--++(0:6)--++(-90:3)--++(180:6)--++(90:3);
 \clip(0,0)--++(0:6)--++(-90:3)--++(180:6)--++(90:3);

          \draw(0,0) rectangle (0.5,-1) node [midway, rotate=90] {1};
 \draw(0.5,0) rectangle (1,-1) node [midway, rotate=90] {1};
 \draw(1.5,0) rectangle (2,-1) node [midway, rotate=90] {1};
 \draw(1,0) rectangle (1.5,-1) node [midway, rotate=90] {1};
  \draw(2,0) rectangle (2.5,-1) node [midway, rotate=90] {1};
 \draw(0.5+2,0) rectangle (0.5+2.5,-1) node [midway, rotate=90] {1};
 \draw(0.5+0.5+2,0) rectangle (0.5+0.5+2.5,-1) node [midway, rotate=90] {1};
  \draw(0.5+0.5+0.5+2,0) rectangle (0.5+0.5+0.5+2.5,-1) node [midway, rotate=90] {1};

  \draw(0.5+0.5+0.5+2+0.5,0) rectangle (0.5+0.5+0.5+2.5+1,-0.5) node [midway] {1};
    \draw(1+0.5+0.5+0.5+2+0.5,0) rectangle (1+0.5+0.5+0.5+2.5+1,-0.5) node [midway] {1};

 \draw(0.5,0-1) rectangle (1,-1-1) node [midway, rotate=90] {2};
 \draw(1.5,0-1) rectangle (2,-1-1) node [midway, rotate=90] {2};
 \draw(1,0-1) rectangle (1.5,-1-1) node [midway, rotate=90] {2};
  \draw(0.5,0-1) rectangle (0,-1-1) node [midway, rotate=90] {2};

  \draw(-2+0.5+0.5+0.5+2+0.5,0-0.5-0.5) rectangle (-2+0.5+0.5+0.5+2.5+1,-0.5-0.5-0.5) node [midway] {2};
    \draw(-2+1+0.5+0.5+0.5+2+0.5,0-0.5-0.5) rectangle (-2+1+0.5+0.5+0.5+2.5+1,-0.5-0.5-0.5) node [midway] {2};
  \draw(0.5+0.5+0.5+2+0.5,0-0.5) rectangle (0.5+0.5+0.5+2.5+1,-0.5-0.5) node [midway] {2};
    \draw(1+0.5+0.5+0.5+2+0.5,0-0.5) rectangle (1+0.5+0.5+0.5+2.5+1,-0.5-0.5) node [midway] {2};

 \draw(0.5,0-1-1) rectangle (1,-1-1-1) node [midway, rotate=90] {3};
 \draw(0,0-1-1) rectangle (0.5,-1-1-1) node [midway, rotate=90] {3};
 \draw(1,0-1-1) rectangle (2,-1-1-0.5) node [midway] {3};

  \draw(-2+0.5+0.5+0.5+2+0.5,0-0.5-0.5-0.5) rectangle (-2+0.5+0.5+0.5+2.5+1,-0.5-0.5-0.5-0.5) node [midway] {3};
    \draw(-2+1+0.5+0.5+0.5+2+0.5,0-0.5-0.5-0.5) rectangle (-2+1+0.5+0.5+0.5+2.5+1,-0.5-0.5-0.5-0.5) node [midway] {3};
  \draw(0.5+0.5+0.5+2+0.5,0-0.5-0.5) rectangle (0.5+0.5+0.5+2.5+1,-0.5-0.5-0.5) node [midway] {3};
    \draw(1+0.5+0.5+0.5+2+0.5,0-0.5-0.5) rectangle (1+0.5+0.5+0.5+2.5+1,-0.5-0.5-0.5) node [midway] {3};

          \draw(1,-2.5)rectangle (2,-3)node [midway] {4};

  \draw(-2+0.5+0.5+0.5+2+0.5,0-0.5-0.5-0.5-0.5) rectangle (-2+0.5+0.5+0.5+2.5+1,-0.5-0.5-0.5-0.5-0.5) node [midway] {4};
    \draw(-2+1+0.5+0.5+0.5+2+0.5,0-0.5-0.5-0.5-0.5) rectangle (-2+1+0.5+0.5+0.5+2.5+1,-0.5-0.5-0.5-0.5-0.5) node [midway] {4};
  \draw(0.5+0.5+0.5+2+0.5,0-0.5-0.5-0.5) rectangle (0.5+0.5+0.5+2.5+1,-0.5-0.5-0.5-0.5) node [midway] {4};
    \draw(1+0.5+0.5+0.5+2+0.5,0-0.5-0.5-0.5) rectangle (1+0.5+0.5+0.5+2.5+1,-0.5-0.5-0.5-0.5) node [midway] {4};

 \draw(0,-3)rectangle (0.5,-4)node [midway,rotate=90] {4};
   \draw(0.5,-3)rectangle (1,-4)node [midway,rotate=90] {5};

  \draw(-2+0.5+0.5+0.5+2+0.5,0-0.5-0.5-0.5-0.5-0.5) rectangle (-2+0.5+0.5+0.5+2.5+1,-0.5-0.5-0.5-0.5-0.5-0.5) node [midway] {5};
    \draw(-2+1+0.5+0.5+0.5+2+0.5,0-0.5-0.5-0.5-0.5-0.5) rectangle (-2+1+0.5+0.5+0.5+2.5+1,-0.5-0.5-0.5-0.5-0.5-0.5) node [midway] {5};
  \draw(0.5+0.5+0.5+2+0.5,0-0.5-0.5-0.5-0.5) rectangle (0.5+0.5+0.5+2.5+1,-0.5-0.5-0.5-0.5-0.5) node [midway] {5};
    \draw(1+0.5+0.5+0.5+2+0.5,0-0.5-0.5-0.5-0.5) rectangle (1+0.5+0.5+0.5+2.5+1,-0.5-0.5-0.5-0.5-0.5) node [midway] {5};

  \draw(0.5+0.5+0.5+2+0.5,0-0.5-0.5-0.5-0.5-0.5) rectangle (0.5+0.5+0.5+2.5+1,-0.5-0.5-0.5-0.5-0.5-0.5) node [midway] {6};
    \draw(1+0.5+0.5+0.5+2+0.5,0-0.5-0.5-0.5-0.5-0.5) rectangle (1+0.5+0.5+0.5+2.5+1,-0.5-0.5-0.5-0.5-0.5-0.5) node [midway] {6};

             \end{tikzpicture}
\qquad
   \begin{tikzpicture} [scale=1]

   \draw[pattern=north east lines, pattern color=magenta!40](0,0)--++(0:2)--++(-90:2)--++(180:1)--++(-90:1)--++(180:1)  --++(90:3);

 \draw[very thick](0,0)--++(0:3)--++(-90:3)--++(180:3)--++(90:3);
 \clip(0,0)--++(0:3)--++(-90:3)--++(180:3)--++(90:3);

          \draw(0,0) rectangle (0.5,-1) node [midway, rotate=90] {1};
 \draw(0.5,0) rectangle (1,-1) node [midway, rotate=90] {1};
 \draw(1.5,0) rectangle (2,-1) node [midway, rotate=90] {1};
 \draw(1,0) rectangle (1.5,-1) node [midway, rotate=90] {1};
 \draw(2,0) rectangle (3,-0.5) node [midway] {1};
  \draw(2,-0.5) rectangle (3,-1) node [midway] {2};

  \draw(0.5+0.5+2,0) rectangle (0.5+0.5+2.5,-1) node [midway, rotate=90] {1};
  \draw(0.5+0.5+0.5+2,0) rectangle (0.5+0.5+0.5+2.5,-1) node [midway, rotate=90] {1};

  \draw(0.5+0.5+0.5+2+0.5,0) rectangle (0.5+0.5+0.5+2.5+1,-0.5) node [midway] {1};
    \draw(1+0.5+0.5+0.5+2+0.5,0) rectangle (1+0.5+0.5+0.5+2.5+1,-0.5) node [midway] {1};

 \draw(0.5,0-1) rectangle (1,-1-1) node [midway, rotate=90] {2};
 \draw(1.5,0-1) rectangle (2,-1-1) node [midway, rotate=90] {2};
 \draw(1,0-1) rectangle (1.5,-1-1) node [midway, rotate=90] {2};
  \draw (0.5,0-1) rectangle (0,-1-1) node [midway, rotate=90] {2};

  \draw (-2+0.5+0.5+0.5+2+0.5,0-0.5-0.5) rectangle (-2+0.5+0.5+0.5+2.5+1,-0.5-0.5-0.5) node [midway] {3};
    \draw(-2+1+0.5+0.5+0.5+2+0.5,0-0.5-0.5) rectangle (-2+1+0.5+0.5+0.5+2.5+1,-0.5-0.5-0.5) node [midway] {2};
  \draw(0.5+0.5+0.5+2+0.5,0-0.5) rectangle (0.5+0.5+0.5+2.5+1,-0.5-0.5) node [midway] {2};
    \draw(1+0.5+0.5+0.5+2+0.5,0-0.5) rectangle (1+0.5+0.5+0.5+2.5+1,-0.5-0.5) node [midway] {2};

 \draw(0.5,0-1-1) rectangle (1,-1-1-1) node [midway, rotate=90] {3};
 \draw(0,0-1-1) rectangle (0.5,-1-1-1) node [midway, rotate=90] {3};
 \draw (1,0-1-1) rectangle (2,-1-1-0.5) node [midway] {3};

  \draw(-2+0.5+0.5+0.5+2+0.5,0-0.5-0.5-0.5) rectangle (-2+0.5+0.5+0.5+2.5+1,-0.5-0.5-0.5-0.5) node [midway] {4};
    \draw(-2+1+0.5+0.5+0.5+2+0.5,0-0.5-0.5-0.5) rectangle (-2+1+0.5+0.5+0.5+2.5+1,-0.5-0.5-0.5-0.5) node [midway] {3};
  \draw(0.5+0.5+0.5+2+0.5,0-0.5-0.5) rectangle (0.5+0.5+0.5+2.5+1,-0.5-0.5-0.5) node [midway] {3};
    \draw(1+0.5+0.5+0.5+2+0.5,0-0.5-0.5) rectangle (1+0.5+0.5+0.5+2.5+1,-0.5-0.5-0.5) node [midway] {3};

          \draw(1,-2.5)rectangle (2,-3)node [midway] {4};

  \draw (-2+0.5+0.5+0.5+2+0.5,0-0.5-0.5-0.5-0.5) rectangle (-2+0.5+0.5+0.5+2.5+1,-0.5-0.5-0.5-0.5-0.5) node [midway] {5};
    \draw(-2+1+0.5+0.5+0.5+2+0.5,0-0.5-0.5-0.5-0.5) rectangle (-2+1+0.5+0.5+0.5+2.5+1,-0.5-0.5-0.5-0.5-0.5) node [midway] {4};
  \draw(0.5+0.5+0.5+2+0.5,0-0.5-0.5-0.5) rectangle (0.5+0.5+0.5+2.5+1,-0.5-0.5-0.5-0.5) node [midway] {4};
    \draw(1+0.5+0.5+0.5+2+0.5,0-0.5-0.5-0.5) rectangle (1+0.5+0.5+0.5+2.5+1,-0.5-0.5-0.5-0.5) node [midway] {4};

 \draw(0,-3)rectangle (0.5,-4)node [midway,rotate=90] {4};
   \draw(0.5,-3)rectangle (1,-4)node [midway,rotate=90] {5};

  \draw(-2+0.5+0.5+0.5+2+0.5,0-0.5-0.5-0.5-0.5-0.5) rectangle (-2+0.5+0.5+0.5+2.5+1,-0.5-0.5-0.5-0.5-0.5-0.5) node [midway] {6};
    \draw(-2+1+0.5+0.5+0.5+2+0.5,0-0.5-0.5-0.5-0.5-0.5) rectangle (-2+1+0.5+0.5+0.5+2.5+1,-0.5-0.5-0.5-0.5-0.5-0.5) node [midway] {5};
  \draw(0.5+0.5+0.5+2+0.5,0-0.5-0.5-0.5-0.5) rectangle (0.5+0.5+0.5+2.5+1,-0.5-0.5-0.5-0.5-0.5) node [midway] {5};
    \draw(1+0.5+0.5+0.5+2+0.5,0-0.5-0.5-0.5-0.5) rectangle (1+0.5+0.5+0.5+2.5+1,-0.5-0.5-0.5-0.5-0.5) node [midway] {5};

  \draw(0.5+0.5+0.5+2+0.5,0-0.5-0.5-0.5-0.5-0.5) rectangle (0.5+0.5+0.5+2.5+1,-0.5-0.5-0.5-0.5-0.5-0.5) node [midway] {6};
    \draw(1+0.5+0.5+0.5+2+0.5,0-0.5-0.5-0.5-0.5-0.5) rectangle (1+0.5+0.5+0.5+2.5+1,-0.5-0.5-0.5-0.5-0.5-0.5) node [midway] {6};

             \end{tikzpicture}
 $$
 \caption{The unique admissible tableaux
 for $\widehat{\la}=(4,2,1)\subseteq (6^3)$ and
 $\widehat{\la}=(2^2,1)\subseteq (3^3)$
 are of odd and even spin types, respectively \color{black} (see \cref{ineedalabel}).  These tableaux have weights 
 $\la=(10,8,7,5,4,2)$ and  $\la=(5^2,4,2,1^2)$ respectively.  See \cref{hkdsfhksdfhkudsfghkudsfghkuf} and \cref{uiywertyiouwteriuyoterwuyioetwruyetwryo}    for the back-and-forth between $\la$ and $\widehat{\la}$.  
 }
\label{twodomtab}\end{figure}

\begin{defn}\label{hkdsfhksdfhkudsfghkudsfghkuf} \renewcommand{\alpha}{\lambda}
\renewcommand{\theta}{{\widehat{\la}}}
\renewcommand{\vartheta}{{\widehat{\la}}}
\color{black} Given a rectangular weight  $\theta\subseteq (a^b)$  as above
we define   
$$ \alpha_i=
 \begin{cases}
a+\theta_i 					&\text{ for } 1\leq i \leq \ell , \\
 a  							&\text{ for }\ell+1\leq i \leq 2b-\ell \\
 a-\theta_{2b+1-i} 				&\text{ for } 2b-\ell+1\leq i \leq 2b,
 \end{cases}	
$$
and we write ${\sf weight}_{a,b}(\widehat{\la})=\la$.  
\end{defn}

\begin{rmk}\label{ishouldref}  \renewcommand{\alpha}{\lambda}
\renewcommand{\theta}{{\widehat{\la}}}
\renewcommand{\vartheta}{{\widehat{\la}}}
Given  $\theta\subseteq (a^b)$  as above, the weight partition  $\la$ is the weight of  $\SSTT^\vartheta$,  the admissible tableau   for $\vartheta$.
Given $\la={\sf weight}_{a,b}(\widehat{\la})$
for some $\widehat{\la}\subseteq (a^b)$ we can reconstruct
  $\theta\subseteq (a^b)$   by noting that
 ${\widehat{\la}}_i=\tfrac{1}{2}( \lambda_i - \lambda_{2b+1-i})$ for $1\leq   i \leq b$.
 \end{rmk}

\begin{eg}\label{uiywertyiouwteriuyoterwuyioetwruyetwryo}   \color{black} 
For $\widehat{\la}=(4,2,1)\subseteq (6^3)$ 
 depicted in \cref{twodomtab} have 
$$
{\sf  weight}_{6,3}(4,2,1)= (6+4,6+2,6+1,6-1,6-2,6-4)=(10,8,7,5,4,2)
$$
which we have   calculated using \cref{hkdsfhksdfhkudsfghkudsfghkuf}.   One can verify that these are the weights of the tableaux 
in \cref{twodomtab} simply by counting the number of dominoes.  
We can recover the rectangular weight 
as follows 
$$
\widehat{(10,8,7,5,4,2)}= (\tfrac{1}{2}(10-2),\tfrac{1}{2}(8-4), \tfrac{1}{2}(7-5))
 =(4,2,1)
$$
using \cref{ishouldref}.  
 
\end{eg}

\begin{prop}\label{rectangle}
\color{black} Let $\lambda\vdash2 ab$  with $\ell (\lambda)\leq 2b$.
 We have that
$$\langle s_{(a^b)} \boxtimes  s_{(a^b)}\mid s_\lambda\rangle =
\begin{cases}
1	&\text{if $ \la={\sf weight}_{a,b}(\widehat{\la})$ for some  } {\widehat{\la}} \subseteq (a^b) \\
 0		&\text{otherwise.}
\end{cases}
$$
In the former case, the unique element of ${\sf Dom}((a^b),\lambda)$ is given by the admissible tableau
$\SSTT^{\widehat{\la}}$ associated to ${\widehat{\la}}\subseteq (a^b)$.

\end{prop}

\begin{proof}
 \renewcommand{\alpha}{\lambda}
\renewcommand{\theta}{{\widehat{\la}}}
\renewcommand{\vartheta}{{\widehat{\la}}}

Let $\SSTT \in {\sf Dom}(a^b,\alpha)$ for some $\alpha\vdash 2ab$.  Let $R(\SSTT ) $ denote the reading word of $\SSTT$.
In the rightmost column, $R(\SSTT ) $ only reads the labels of $(1^2)$-dominoes.
Thus all $(1^2)$-dominoes occur above $(2)$-dominoes in this column and they  are labelled by consecutive numbers starting from 1.
Thus the reading word for this column is $1,2,\dots,i_{2a}$ for some $i_{2a}\leq b$.
 Before reading $R(\SSTT)$ for the $(2a-1)$th column,  we note that
 adjacent to every $(1^2)$-domino of label $1\leq j \leq i_{2a}$ in column $2a$ we have another
 $(1^2)$-domino of the same label in column $2a-1$
   (by the semistandard condition).
   The remaining rows of the $(2a-1)$th column were all previously  determined to be   $(2)$-dominoes.
   By the lattice permutation condition, these horizontal dominoes have labels $i_{2a}+1,i_{2a}+2,\dots, 2b-i_{2a}$.
We remark that all the  dominoes we have determined so far
belong to a unique square $(r,c)_{2}:=\{ 2r-1,2r\}\times \{2c-1,2c\}$ for some $(r,c)\in (a^b)$ {\color{black}with $c=a$}.
Therefore it makes sense to speak of us having just determined the $a$th double-column.
The reading word of this double column is a prefix of the reading word of $\SSTT$ and is of the form
$$
R_a(\SSTT)=(1,2,3,\dots, i_{2a},1,2,3,\dots, i_{2a},i_{2a}+1 ,i_{2a}+2, \dots,  {2b-i_{2a})}.
$$
 The only numbers $i$ in $R_a(\SSTT)$ which are free to support a subsequent  $ i+1$
 in $R (\SSTT)\setminus R_a(\SSTT)$  under the system of parentheses are $ i_{2a}$ and $ {2 b-i_{2a}}$.
  This is summarised visually in 
  in \cref{below1} below.
 
 \begin{figure}[ht!]  
 $$  \scalefont{0.8}  \begin{tikzpicture} [xscale=1.3,yscale=-1.25]

   \draw[pattern=north east lines, pattern color=cyan!40](0,-0.5)--++(0:1)--++(-90:0.5)--++(180:1)--++(90:0.5) ;
   \draw[pattern=north east lines, pattern color=cyan!40](0+0.5,-3)--++(0:0.5)--++(-90:1)--++(180:0.5)--++(90:1) ;
 \draw[very thick](0,-1 )--(0,-0.5)--++(0:1)--++(-90:0.5 );
  \draw (0,-1)--++(0:1);  
  \draw[densely dotted,very thick](0,-1 )--++(-90:1) (1,-1 )--++(-90:1);
 \draw[very thick](0,-2)--++(-90:2) (1,-2)--++(-90:2);   
  \draw(0,-2)--++(0:1);
    \draw(0,-2.5)--++(0:1);
        \draw(0,-3)--++(0:0.5)--++(-90:1);
            \draw(0,-4)--++(0:1);            \draw(0,-3)--++(0:1);

 \draw[  densely dotted](0.5,-4)--++(-90:1) ;
 \draw[very thick,densely dotted](0,-4)--++(-90:1) (1,-4)--++(-90:1);   
  \draw[very thick] (0,-5)--++(-90:2) --++(0:1)--++(90:2);     
  \draw (0.5,-7)--++(90:2);
  \draw(0,-6)--++(0:1);       \draw(0,-5)--++(0:1);            
  \draw(0.5,-0.75) node {$2b-i_{2a} $ };

  \draw(0.5,-2.25) node {$i_{2a}+2 $ };
  \draw(0.5,-2.75) node {$i_{2a}+1 $ };

              \draw(0.25,-3.5) node[rotate=90]  {$i_{2a}  $ };
                            \draw(0.75,-3.5) node[rotate=90]  {$i_{2a}  $ };

      \draw(0.25,-5.5) node[rotate=90]  {$ 2  $ };
                            \draw(0.75,-5.5) node[rotate=90]  {$2  $ };

      \draw(0.25,-6.5) node[rotate=90]  {$ 1  $ };
                            \draw(0.75,-6.5) node[rotate=90]  {$1  $ };             \end{tikzpicture}$$
 \caption{\color{black}The last two columns (or last ``double column") of an arbitrary tableau in ${\sf Dom}(a^b,\alpha)$.  We have highlighted in blue  the only two dominoes whose  integer entries are free to support subsequent terms in $R (\SSTT)\setminus R_a(\SSTT)$.  Notice that the length of the column (namely, $2b$) and the final entry in a $(1^2)$-domino (namely, $i_{2a}$) together determine the entry in the final $(2)$-domino (namely $2b-i_{2a}$)}
 \label{below1}
 \end{figure}
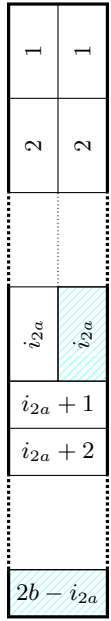

 Before reading $R(\SSTT)$ for the $(2(a-1))$th column,  we note that, 
\color{black} 
 as $\SSTT$ is semistandard,
 \color{black}
 adjacent to every $(1^2)$-domino of label $1\leq j \leq i_{2a}$ in column $2a-1$ we have another
 $(1^2)$-domino of the same label in column $2(a-1)$.  
 (\color{black} In other words, the first 
 $2i_{2a}$ rows 
 \color{black}
 of the penultimate column are the same as those of the the  $2a$th and $(2a-1)$th columns.) \color{black}   
Similarly to how we argued when reading the $2a$th column, we see that all $(1^2)$-dominoes must appear above $(2)$-dominoes
(as all the labels $j$ of these subsequent dominoes are $i_{2a}<j\leq 2b-i_{2a}$ and thus cannot be supported by elements of $R_a(\SSTT)$).
 The labels of these subsequent  $(1^2)$-dominoes are consecutive $i_{2a}+1,\dots,i_{2(a-1)}$.  
  {\color{black}In particular, we note that  $i_{2a}\leq i_{2(a-1)}\leq  b  $}.
 Before reading $R(\SSTT)$ for the $(2a-3)$th column,  we note that
 adjacent to every $(1^2)$-domino of label $1\leq j \leq i_{2(a-1)}$ in column $2(a-1)$ we have another
 $(1^2)$-domino of the same label in column $2a-3$.
 The remaining rows of the $(2a-3)$th column were all previously determined to be   $(2)$-dominoes.
 By the \color{black}
 semistandard property and the
 \color{black}
   lattice permutation condition, these labels are $i_{2(a-1)}+1 ,i_{2(a-1)}+2, \dots    $
 Therefore it makes sense to speak of us having just determined the $(a-1)$th double-column.
The reading word of this double column is a subword of the reading word of $\SSTT$ and is of the form
$$
R_{a-1}(\SSTT)=(1,2,3,\dots, i_{2(a-1)},1,2,3,\dots, i_{2(a-1)},i_{2(a-1)}+1 ,i_{2(a-1)}+2, \dots,  {2b-i_{2(a-1)})}.
$$
Repeating this argument, we deduce that  $\SSTT$ is indeed {\color{black}the admissible $\theta$-tableau   for
 $\theta$ with $\theta^T=(i_2,i_4,\dots, i_{2a})$} with reading word
\[
R_{a}(\SSTT)\circ R_{a-1}(\SSTT)\circ \dots \circ R_{1}(\SSTT). \qedhere
\]
   \end{proof}

\begin{rmk}\label{useful!}
We emphasise that the only numbers in  $R_k(\SSTT)$ which were free to support a subsequent integer
in
 $R (\SSTT) \setminus \cup_{j\leq k}\{ R_k(\SSTT)\}$   were   $ i_{2k}$ and $ {2 b-i_{2k}}$;  however, these integers {\em never did} support any subsequent integer.
 In particular each subword $R_k(\SSTT)$ of $R(\SSTT)$ for $1\leq k \leq a$ was itself a lattice permutation.
\end{rmk}

\begin{prop}\label{3.13}
 For $\nu \vdash 2$,
the products $s_\nu\circ s_{(a^b,1) }$   are   multiplicity-free.
\end{prop}

 \begin{proof}
Let  $\SSTT \in {\sf Dom}((a^b,1),\alpha)$ for some $\alpha\vdash 2ab+2$.
 Proceeding as in the rectangle case, we deduce that any domino $D$ in  $\SSTT$
belongs to a unique square $(r,c)_2=\{ 2r-1,2r\}\times \{2c-1,2c\}$ for some $(r,c)\in (a^b,1)$.  In particular, it makes sense to factorise the reading word as
$$
R_a(\SSTT)\circ R_{a-1}(\SSTT)\circ \dots \circ R_1(\SSTT)
$$
where $R_i(\SSTT)$ is the reading word of the $i$th double column.
Moreover, each $R_i(\SSTT)$ is itself a lattice permutation for $i>1$ just as in \cref{useful!}.
This is not true for $R_1(\SSTT)$ as we see in the example in \cref{cupbox} below, since the dominoes in $(b+1,1)_2$ will,  in general, be  matched with elements of
$R_i(\SSTT)$ for $i\geq1$.

  \begin{figure}[ht!]
    $$
  \begin{tikzpicture} [scale=1]

   \draw[pattern=north east lines, pattern color=magenta!40](0,0)--++(0:4)--++(-90:1)--++(180:2)--++(-90:1)--++(180:1) --++(-90:1)--++(180:1)  --++(90:3);
\draw[pattern=north west lines, pattern color=cyan!60](0,-3)--++(0:1)--++(-90:1)--++(180:1);
 \draw[very thick](0,0)--++(0:6)--++(-90:3)--++(180:5)--++(-90:1)--++(180:1)--++(90:4);

          \draw(0,0) rectangle (0.5,-1) node [midway, rotate=90] {1};
 \draw(0.5,0) rectangle (1,-1) node [midway, rotate=90] {1};
 \draw(1.5,0) rectangle (2,-1) node [midway, rotate=90] {1};
 \draw(1,0) rectangle (1.5,-1) node [midway, rotate=90] {1};
  \draw(2,0) rectangle (2.5,-1) node [midway, rotate=90] {1};
 \draw(0.5+2,0) rectangle (0.5+2.5,-1) node [midway, rotate=90] {1};
 \draw(0.5+0.5+2,0) rectangle (0.5+0.5+2.5,-1) node [midway, rotate=90] {1};
  \draw(0.5+0.5+0.5+2,0) rectangle (0.5+0.5+0.5+2.5,-1) node [midway, rotate=90] {1};

  \draw(0.5+0.5+0.5+2+0.5,0) rectangle (0.5+0.5+0.5+2.5+1,-0.5) node [midway] {1};
    \draw(1+0.5+0.5+0.5+2+0.5,0) rectangle (1+0.5+0.5+0.5+2.5+1,-0.5) node [midway] {1};

 \draw(0.5,0-1) rectangle (1,-1-1) node [midway, rotate=90] {2};
 \draw(1.5,0-1) rectangle (2,-1-1) node [midway, rotate=90] {2};
 \draw(1,0-1) rectangle (1.5,-1-1) node [midway, rotate=90] {2};
  \draw(0.5,0-1) rectangle (0,-1-1) node [midway, rotate=90] {2};

  \draw(-2+0.5+0.5+0.5+2+0.5,0-0.5-0.5) rectangle (-2+0.5+0.5+0.5+2.5+1,-0.5-0.5-0.5) node [midway] {2};
    \draw(-2+1+0.5+0.5+0.5+2+0.5,0-0.5-0.5) rectangle (-2+1+0.5+0.5+0.5+2.5+1,-0.5-0.5-0.5) node [midway] {2};
  \draw(0.5+0.5+0.5+2+0.5,0-0.5) rectangle (0.5+0.5+0.5+2.5+1,-0.5-0.5) node [midway] {2};
    \draw(1+0.5+0.5+0.5+2+0.5,0-0.5) rectangle (1+0.5+0.5+0.5+2.5+1,-0.5-0.5) node [midway] {2};

 \draw(0.5,0-1-1) rectangle (1,-1-1-1) node [midway, rotate=90] {{3}};
 \draw(0,0-1-1) rectangle (0.5,-1-1-1) node [midway, rotate=90] {\underline{3}};
 \draw(1,0-1-1) rectangle (2,-1-1-0.5) node [midway] {3};

  \draw(-2+0.5+0.5+0.5+2+0.5,0-0.5-0.5-0.5) rectangle (-2+0.5+0.5+0.5+2.5+1,-0.5-0.5-0.5-0.5) node [midway] {3};
    \draw(-2+1+0.5+0.5+0.5+2+0.5,0-0.5-0.5-0.5) rectangle (-2+1+0.5+0.5+0.5+2.5+1,-0.5-0.5-0.5-0.5) node [midway] {3};
  \draw(0.5+0.5+0.5+2+0.5,0-0.5-0.5) rectangle (0.5+0.5+0.5+2.5+1,-0.5-0.5-0.5) node [midway] {3};
    \draw(1+0.5+0.5+0.5+2+0.5,0-0.5-0.5) rectangle (1+0.5+0.5+0.5+2.5+1,-0.5-0.5-0.5) node [midway] {3};

     \draw(0,-3)rectangle (1,-3.5)node [midway] {4};
          \draw(1,-2.5)rectangle (2,-3)node [midway] {4};

  \draw(-2+0.5+0.5+0.5+2+0.5,0-0.5-0.5-0.5-0.5) rectangle (-2+0.5+0.5+0.5+2.5+1,-0.5-0.5-0.5-0.5-0.5) node [midway] {4};
    \draw(-2+1+0.5+0.5+0.5+2+0.5,0-0.5-0.5-0.5-0.5) rectangle (-2+1+0.5+0.5+0.5+2.5+1,-0.5-0.5-0.5-0.5-0.5) node [midway] {4};
  \draw(0.5+0.5+0.5+2+0.5,0-0.5-0.5-0.5) rectangle (0.5+0.5+0.5+2.5+1,-0.5-0.5-0.5-0.5) node [midway] {4};
    \draw(1+0.5+0.5+0.5+2+0.5,0-0.5-0.5-0.5) rectangle (1+0.5+0.5+0.5+2.5+1,-0.5-0.5-0.5-0.5) node [midway] {4};

   \draw(0,-3-0.5)rectangle (1,-3.5-0.5)node [midway] {6};

  \draw(-2+0.5+0.5+0.5+2+0.5,0-0.5-0.5-0.5-0.5-0.5) rectangle (-2+0.5+0.5+0.5+2.5+1,-0.5-0.5-0.5-0.5-0.5-0.5) node [midway] {\underline5};
    \draw(-2+1+0.5+0.5+0.5+2+0.5,0-0.5-0.5-0.5-0.5-0.5) rectangle (-2+1+0.5+0.5+0.5+2.5+1,-0.5-0.5-0.5-0.5-0.5-0.5) node [midway] {{5}};
  \draw(0.5+0.5+0.5+2+0.5,0-0.5-0.5-0.5-0.5) rectangle (0.5+0.5+0.5+2.5+1,-0.5-0.5-0.5-0.5-0.5) node [midway] {5};
    \draw(1+0.5+0.5+0.5+2+0.5,0-0.5-0.5-0.5-0.5) rectangle (1+0.5+0.5+0.5+2.5+1,-0.5-0.5-0.5-0.5-0.5) node [midway] {5};

  \draw(0.5+0.5+0.5+2+0.5,0-0.5-0.5-0.5-0.5-0.5) rectangle (0.5+0.5+0.5+2.5+1,-0.5-0.5-0.5-0.5-0.5-0.5) node [midway] {6};
    \draw(1+0.5+0.5+0.5+2+0.5,0-0.5-0.5-0.5-0.5-0.5) rectangle (1+0.5+0.5+0.5+2.5+1,-0.5-0.5-0.5-0.5-0.5-0.5) node [midway] {6};

             \end{tikzpicture}\qquad
              \begin{tikzpicture} [scale=1]
   \draw[pattern=north east lines, pattern color=magenta!40](0,0)--++(0:4)--++(-90:1)--++(180:2)--++(-90:1)--++(180:1) --++(-90:1)--++(180:1)  --++(90:3);
 \draw[very thick](0,0)--++(0:6)--++(-90:3)--++(180:5)--++(-90:1)--++(180:1)--++(90:4);
   \draw[pattern=north west lines, pattern color=cyan!60](0,-3)--++(0:1)--++(-90:1)--++(180:1);

          \draw(0,0) rectangle (0.5,-1) node [midway, rotate=90] {1};
 \draw(0.5,0) rectangle (1,-1) node [midway, rotate=90] {1};
 \draw(1.5,0) rectangle (2,-1) node [midway, rotate=90] {1};
 \draw(1,0) rectangle (1.5,-1) node [midway, rotate=90] {1};
  \draw(2,0) rectangle (2.5,-1) node [midway, rotate=90] {1};
 \draw(0.5+2,0) rectangle (0.5+2.5,-1) node [midway, rotate=90] {1};
 \draw(0.5+0.5+2,0) rectangle (0.5+0.5+2.5,-1) node [midway, rotate=90] {1};
  \draw(0.5+0.5+0.5+2,0) rectangle (0.5+0.5+0.5+2.5,-1) node [midway, rotate=90] {1};

  \draw(0.5+0.5+0.5+2+0.5,0) rectangle (0.5+0.5+0.5+2.5+1,-0.5) node [midway] {1};
    \draw(1+0.5+0.5+0.5+2+0.5,0) rectangle (1+0.5+0.5+0.5+2.5+1,-0.5) node [midway] {1};

 \draw(0.5,0-1) rectangle (1,-1-1) node [midway, rotate=90] {2};
 \draw(1.5,0-1) rectangle (2,-1-1) node [midway, rotate=90] {2};
 \draw(1,0-1) rectangle (1.5,-1-1) node [midway, rotate=90] {2};
  \draw(0.5,0-1) rectangle (0,-1-1) node [midway, rotate=90] {2};

  \draw(-2+0.5+0.5+0.5+2+0.5,0-0.5-0.5) rectangle (-2+0.5+0.5+0.5+2.5+1,-0.5-0.5-0.5) node [midway] {2};
    \draw(-2+1+0.5+0.5+0.5+2+0.5,0-0.5-0.5) rectangle (-2+1+0.5+0.5+0.5+2.5+1,-0.5-0.5-0.5) node [midway] {2};
  \draw(0.5+0.5+0.5+2+0.5,0-0.5) rectangle (0.5+0.5+0.5+2.5+1,-0.5-0.5) node [midway] {2};
    \draw(1+0.5+0.5+0.5+2+0.5,0-0.5) rectangle (1+0.5+0.5+0.5+2.5+1,-0.5-0.5) node [midway] {2};

 \draw(0.5,0-1-1) rectangle (1,-1-1-1) node [midway, rotate=90] {3};
 \draw(0,0-1-1) rectangle (0.5,-1-1-1) node [midway, rotate=90] {\underline{3}};
 \draw(1,0-1-1) rectangle (2,-1-1-0.5) node [midway] {3};

  \draw(-2+0.5+0.5+0.5+2+0.5,0-0.5-0.5-0.5) rectangle (-2+0.5+0.5+0.5+2.5+1,-0.5-0.5-0.5-0.5) node [midway] {3};
    \draw(-2+1+0.5+0.5+0.5+2+0.5,0-0.5-0.5-0.5) rectangle (-2+1+0.5+0.5+0.5+2.5+1,-0.5-0.5-0.5-0.5) node [midway] {3};
  \draw(0.5+0.5+0.5+2+0.5,0-0.5-0.5) rectangle (0.5+0.5+0.5+2.5+1,-0.5-0.5-0.5) node [midway] {3};
    \draw(1+0.5+0.5+0.5+2+0.5,0-0.5-0.5) rectangle (1+0.5+0.5+0.5+2.5+1,-0.5-0.5-0.5) node [midway] {3};

          \draw(1,-2.5)rectangle (2,-3)node [midway] {4};

  \draw(-2+0.5+0.5+0.5+2+0.5,0-0.5-0.5-0.5-0.5) rectangle (-2+0.5+0.5+0.5+2.5+1,-0.5-0.5-0.5-0.5-0.5) node [midway] {4};
    \draw(-2+1+0.5+0.5+0.5+2+0.5,0-0.5-0.5-0.5-0.5) rectangle (-2+1+0.5+0.5+0.5+2.5+1,-0.5-0.5-0.5-0.5-0.5) node [midway] {4};
  \draw(0.5+0.5+0.5+2+0.5,0-0.5-0.5-0.5) rectangle (0.5+0.5+0.5+2.5+1,-0.5-0.5-0.5-0.5) node [midway] {4};
    \draw(1+0.5+0.5+0.5+2+0.5,0-0.5-0.5-0.5) rectangle (1+0.5+0.5+0.5+2.5+1,-0.5-0.5-0.5-0.5) node [midway] {4};

 \draw(0,-3)rectangle (0.5,-4)node [midway,rotate=90] {4};
   \draw(0.5,-3)rectangle (1,-4)node [midway,rotate=90] {6};

  \draw(-2+0.5+0.5+0.5+2+0.5,0-0.5-0.5-0.5-0.5-0.5) rectangle (-2+0.5+0.5+0.5+2.5+1,-0.5-0.5-0.5-0.5-0.5-0.5) node [midway] {\underline5};
    \draw(-2+1+0.5+0.5+0.5+2+0.5,0-0.5-0.5-0.5-0.5-0.5) rectangle (-2+1+0.5+0.5+0.5+2.5+1,-0.5-0.5-0.5-0.5-0.5-0.5) node [midway] {{5}};
  \draw(0.5+0.5+0.5+2+0.5,0-0.5-0.5-0.5-0.5) rectangle (0.5+0.5+0.5+2.5+1,-0.5-0.5-0.5-0.5-0.5) node [midway] {5};
    \draw(1+0.5+0.5+0.5+2+0.5,0-0.5-0.5-0.5-0.5) rectangle (1+0.5+0.5+0.5+2.5+1,-0.5-0.5-0.5-0.5-0.5) node [midway] {5};

  \draw(0.5+0.5+0.5+2+0.5,0-0.5-0.5-0.5-0.5-0.5) rectangle (0.5+0.5+0.5+2.5+1,-0.5-0.5-0.5-0.5-0.5-0.5) node [midway] {6};
    \draw(1+0.5+0.5+0.5+2+0.5,0-0.5-0.5-0.5-0.5-0.5) rectangle (1+0.5+0.5+0.5+2.5+1,-0.5-0.5-0.5-0.5-0.5-0.5) node [midway] {6};

             \end{tikzpicture}
 $$
\caption{The two tableaux of ${\sf Dom}((6^3,1),(10,8,7,6,4,3))$.  The first 6 rows are common to both tableaux and are uniquely determined by the weight. The colouring   highlights  the partition $(4,2,1)\subseteq (6^3) \subset (6^3,1) $ and the final double-row.
\color{black} We have underlined the   elements paired with the dominoes  in $(b+1,1)_2$ under the reading word (note that these underlined entries all belong to the final double-row of the rectangle).}\label{cupbox}
\end{figure}

We now consider the word $R_1(\SSTT)$ in more detail.
 The two dominoes $D$ and $D'$ belonging to $(b+1,1)_2$ have labels $d\leq d'$ respectively, both of which are strictly greater than any other label in
$R_1(\SSTT)$. Thus we can remove the integers $d$ and $d'$ from $R_1(\SSTT)$ without affecting the system of parentheses.
Therefore the 
  domino tableau \color{black}
 $\SSTT_{\leq 2b}=\SSTT \setminus \{D,D'\}$ is of shape $(a^b)$,
  weight $\la:=\alpha-\varepsilon_{d}-\varepsilon_{d'}$,
 and its reading word is a lattice permutation.  In particular $\SSTT_{\leq 2b}$ is the unique admissible ${\widehat{\la}}$-tableau for some ${\widehat{\la}}\subseteq (a^b)$.

\renewcommand{\theta}{{\widehat{\la}}}
\renewcommand{\vartheta}{{\widehat{\la}}}

The partition $\widehat{\lambda}$ and the labels $d,d'$ are uniquely determined by the weight $\alpha$. To see this observe that as $d,d'>b$ then $\widehat{\lambda}_i={\lambda}_i-a=\alpha_i-a$ for $1 \le i \le b$ by \cref{ishouldref}. Then $\widehat{\lambda}$ determines $\lambda$, from which we can read off the values of $d,d'$.  All that remains to determine is whether the dominoes of $(b+1,1)_2$ are both $(1^2)$-dominoes or both $(2)$-dominoes. If both possibilities satisfy the lattice condition there are  two resulting domino tableaux of weight $\alpha$  which have opposite signs, or otherwise there is a unique domino tableau of this weight.
  \end{proof}
  
  \begin{rmk}
  \color{black}
 In the above proof, we assumed that there existed a   $\SSTT \in {\sf Dom}((a^b,1),\alpha)$ and we proved that under this assumption this was the unique element of ${\sf Dom}((a^b,1),\alpha)$ 
 {\em of this 
  given spin-type}.  
We did this by showing that (1) any two tableaux from ${\sf Dom}((a^b,1),\alpha)$  must coincide within the rectangular region and (2) noticing that this implied that they must differ by rotating the dominoes within $(b+1,1)_2$, thus having distinct spin-types. 
We emphasise that rearranging these dominoes will always change the spin-type, but it might also break the semistandard or lattice permutation conditions (this is to be expected as not all coefficients in the product $s_\mu \boxtimes s_\mu$ have coefficient 2). 
For example, if the two dominoes $D$ and $D'$ within $(b+1,1)_2$ have the same label $d=d'$, then they must both be $(1^2)$-dominoes.  
   \end{rmk}

 \begin{rmk}\label{usefulremarkweneed}
 We remark that the
  two dominoes $D$ and $D'$ must be either $(a)$ supported by integers
  $ i_{2k}$ or  $ {2 b-i_{2k}}$ for some $1\leq k \leq a$ as in \cref{useful!},
	or
  $(b)$  $D$ is supported by such an integer and $D'$ is supported by $D$.
  However, $i_{2a}\leq i_{2(a-1)}\leq \dots \leq i_{2}\leq 2 b-i_{2}$  	
   so in actual fact $D$ and $D'$ (respectively $D$ in case $(b)$) must be supported by
  some integers $ 2 b-i_{2k}  $ for $1\leq k \leq a$ which are precisely the labels of the dominoes which intersect the $2b$th row.
  To summarise, the dominoes $D$ and $D'$ are paired (under the system of parentheses) with dominoes of the form
 $\{(2b-1,c),(2b,c)\}$ or  $\{(2b,c-1),(2b,c )\}$ for some $1\leq c \leq 2a$,
   or $D'$ is paired with $D$, and $D$ is paired with such a domino.
     \end{rmk}

\begin{prop}
 For $\nu \vdash 2$,
the products  $s_\nu \circ s_{(a^b,a-1) }$ are   multiplicity-free.
\end{prop}

 \begin{proof}
\renewcommand{\theta}{{\widehat{\la}}}
\renewcommand{\vartheta}{{\widehat{\la}}}

 Let  $\SSTT \in {\sf Dom}((a^b,a-1),\alpha)$ for some $\alpha\vdash 2ab+2a-2$.
 Proceeding as in the rectangle case, we deduce that any domino $D$ in  $\SSTT$
belongs to a unique square $(r,c)_2=\{ 2r-1,2r\}\times \{2c-1,2c\}$ for some $(r,c)\in (a^b)\subset(a^b,a-1)$.
 {\em However} this is not true for the final double-row, i.e., $(r,c)\in ((a^b,a-1)\setminus  (a^b))$.
 Namely, there can exist dominoes of the form $\{(2b+1,2c),(2b+1,2c+1)\}$ or $\{(2b+2,2c),(2b+2,2c+1)\}$ for $1\leq c < a$. An example is depicted in the rightmost tableau in \cref{cupbox2} below.
 Let $D$ be a domino from the final double-row ($\{(x,y)\mid 1\leq y \leq 2a, x > 2b\}$) with label $d$
 and let
 $D'$ be a domino from the first $b$ double-rows  ($\{(x,y)\mid 1\leq y \leq 2a, x \leq  2b\}$) with label $d'$.
If   $d<d'$ , then by the semistandard property, we have that $d$ occurs {\em after} $d'$ in the reading word of $\SSTT$.
Thus  $\SSTT_{\leq 2b} =\SSTT \cap \{(x,y)\mid 1\leq y \leq 2a, x \leq  2b\}$
is itself a semistandard    domino \color{black}  tableau and satisfies the lattice permutation condition.
Thus $\SSTT_{\leq 2b} = \SSTT^{\widehat{\la}}$ for   $\la= \alpha - \varepsilon_{d_1}- \varepsilon_{d_2}-\dots - \varepsilon_{d_{2a-2}} $,
the partition obtained by removing the labels of the dominoes from the final double-row.

  \begin{figure}[ht!]
    $$
              \begin{tikzpicture} [scale=1]
    \draw[pattern=north east lines, pattern color=magenta!40](0,0)--++(0:4)--++(-90:1)--++(180:2)--++(-90:1)--++(180:1) --++(-90:1)--++(180:1)  --++(90:3);
    \draw[pattern=north west lines, pattern color=cyan!60](0,-3)--++(0:5)--++(-90:1)--++(180:5);
    \draw[very thick](0,0)--++(0:6)--++(-90:3)--++(180:1)--++(-90:1)--++(180:5) --++(90:4);
           \draw(0,0) rectangle (0.5,-1) node [midway, rotate=90] {1};
 \draw(0.5,0) rectangle (1,-1) node [midway, rotate=90] {1};
 \draw(1.5,0) rectangle (2,-1) node [midway, rotate=90] {1};
 \draw(1,0) rectangle (1.5,-1) node [midway, rotate=90] {1};
  \draw(2,0) rectangle (2.5,-1) node [midway, rotate=90] {1};
 \draw(0.5+2,0) rectangle (0.5+2.5,-1) node [midway, rotate=90] {1};
 \draw(0.5+0.5+2,0) rectangle (0.5+0.5+2.5,-1) node [midway, rotate=90] {1};
  \draw(0.5+0.5+0.5+2,0) rectangle (0.5+0.5+0.5+2.5,-1) node [midway, rotate=90] {1};

  \draw(0.5+0.5+0.5+2+0.5,0) rectangle (0.5+0.5+0.5+2.5+1,-0.5) node [midway] {1};
    \draw(1+0.5+0.5+0.5+2+0.5,0) rectangle (1+0.5+0.5+0.5+2.5+1,-0.5) node [midway] {1};

 \draw(0.5,0-1) rectangle (1,-1-1) node [midway, rotate=90] {2};
 \draw(1.5,0-1) rectangle (2,-1-1) node [midway, rotate=90] {2};
 \draw(1,0-1) rectangle (1.5,-1-1) node [midway, rotate=90] {2};
  \draw(0.5,0-1) rectangle (0,-1-1) node [midway, rotate=90] {2};

  \draw(-2+0.5+0.5+0.5+2+0.5,0-0.5-0.5) rectangle (-2+0.5+0.5+0.5+2.5+1,-0.5-0.5-0.5) node [midway] {2};
    \draw(-2+1+0.5+0.5+0.5+2+0.5,0-0.5-0.5) rectangle (-2+1+0.5+0.5+0.5+2.5+1,-0.5-0.5-0.5) node [midway] {2};
  \draw(0.5+0.5+0.5+2+0.5,0-0.5) rectangle (0.5+0.5+0.5+2.5+1,-0.5-0.5) node [midway] {2};
    \draw(1+0.5+0.5+0.5+2+0.5,0-0.5) rectangle (1+0.5+0.5+0.5+2.5+1,-0.5-0.5) node [midway] {2};

 \draw(0.5,0-1-1) rectangle (1,-1-1-1) node [midway, rotate=90] {3};
 \draw(0,0-1-1) rectangle (0.5,-1-1-1) node [midway, rotate=90] {3};
 \draw(1,0-1-1) rectangle (2,-1-1-0.5) node [midway] {3};

  \draw(-2+0.5+0.5+0.5+2+0.5,0-0.5-0.5-0.5) rectangle (-2+0.5+0.5+0.5+2.5+1,-0.5-0.5-0.5-0.5) node [midway] {3};
    \draw(-2+1+0.5+0.5+0.5+2+0.5,0-0.5-0.5-0.5) rectangle (-2+1+0.5+0.5+0.5+2.5+1,-0.5-0.5-0.5-0.5) node [midway] {3};
  \draw(0.5+0.5+0.5+2+0.5,0-0.5-0.5) rectangle (0.5+0.5+0.5+2.5+1,-0.5-0.5-0.5) node [midway] {3};
    \draw(1+0.5+0.5+0.5+2+0.5,0-0.5-0.5) rectangle (1+0.5+0.5+0.5+2.5+1,-0.5-0.5-0.5) node [midway] {3};

          \draw(1,-2.5)rectangle (2,-3)node [midway] {4};

  \draw(-2+0.5+0.5+0.5+2+0.5,0-0.5-0.5-0.5-0.5) rectangle (-2+0.5+0.5+0.5+2.5+1,-0.5-0.5-0.5-0.5-0.5) node [midway] {4};
    \draw(-2+1+0.5+0.5+0.5+2+0.5,0-0.5-0.5-0.5-0.5) rectangle (-2+1+0.5+0.5+0.5+2.5+1,-0.5-0.5-0.5-0.5-0.5) node [midway] {4};
  \draw(0.5+0.5+0.5+2+0.5,0-0.5-0.5-0.5) rectangle (0.5+0.5+0.5+2.5+1,-0.5-0.5-0.5-0.5) node [midway] {4};
    \draw(1+0.5+0.5+0.5+2+0.5,0-0.5-0.5-0.5) rectangle (1+0.5+0.5+0.5+2.5+1,-0.5-0.5-0.5-0.5) node [midway] {4};

  \draw(-2+0.5+0.5+0.5+2+0.5,0-0.5-0.5-0.5-0.5-0.5) rectangle (-2+0.5+0.5+0.5+2.5+1,-0.5-0.5-0.5-0.5-0.5-0.5) node [midway] {5};
    \draw(-2+1+0.5+0.5+0.5+2+0.5,0-0.5-0.5-0.5-0.5-0.5) rectangle (-2+1+0.5+0.5+0.5+2.5+1,-0.5-0.5-0.5-0.5-0.5-0.5) node [midway] {5};
  \draw(0.5+0.5+0.5+2+0.5,0-0.5-0.5-0.5-0.5) rectangle (0.5+0.5+0.5+2.5+1,-0.5-0.5-0.5-0.5-0.5) node [midway] {5};
    \draw(1+0.5+0.5+0.5+2+0.5,0-0.5-0.5-0.5-0.5) rectangle (1+0.5+0.5+0.5+2.5+1,-0.5-0.5-0.5-0.5-0.5) node [midway] {5};

  \draw(0.5+0.5+0.5+2+0.5,0-0.5-0.5-0.5-0.5-0.5) rectangle (0.5+0.5+0.5+2.5+1,-0.5-0.5-0.5-0.5-0.5-0.5) node [midway] {6};
    \draw(1+0.5+0.5+0.5+2+0.5,0-0.5-0.5-0.5-0.5-0.5) rectangle (1+0.5+0.5+0.5+2.5+1,-0.5-0.5-0.5-0.5-0.5-0.5) node [midway] {6};

 \draw(0,-3)rectangle (0.5,-4)node [midway,rotate=90] {4};
   \draw (0.5,-3)rectangle (1,-4)node [midway,rotate=90] {4};

  \draw (1,-3)rectangle (2,-3.5)node [midway] {5};
    \draw (2,-3)rectangle (3,-3.5)node [midway] {6};
        \draw (3,-3)rectangle (4,-3.5)node [midway] {6};
\draw (1,-4)rectangle (2,-3.5)node [midway] {6};
    \draw (2,-4)rectangle (3,-3.5)node [midway] {7};
        \draw (3,-4)rectangle (4,-3.5)node [midway] {7};

   \draw (4,-3)rectangle (5,-3.5)node [midway] {7};
   \draw (4,-4)rectangle (5,-3.5)node [midway] {8};
             \end{tikzpicture}
             \qquad\qquad
                  \begin{tikzpicture} [scale=1]
 \draw[very thick](0,0)--++(0:6)--++(-90:2);
   \draw[pattern=north east lines, pattern color=magenta!40](0,0)--++(0:4)--++(-90:1)--++(180:2)--++(-90:1)--++(180:1) --++(-90:1)--++(180:1)  --++(90:3);
    \draw[pattern=north west lines, pattern color=cyan!60](0,-3)--++(0:5)--++(-90:1)--++(180:5);
    \draw[very thick](0,0)--++(0:6)--++(-90:3)--++(180:1)--++(-90:1)--++(180:5) --++(90:4);

          \draw(0,0) rectangle (0.5,-1) node [midway, rotate=90] {1};
 \draw(0.5,0) rectangle (1,-1) node [midway, rotate=90] {1};
 \draw(1.5,0) rectangle (2,-1) node [midway, rotate=90] {1};
 \draw(1,0) rectangle (1.5,-1) node [midway, rotate=90] {1};
  \draw(2,0) rectangle (2.5,-1) node [midway, rotate=90] {1};
 \draw(0.5+2,0) rectangle (0.5+2.5,-1) node [midway, rotate=90] {1};
 \draw(0.5+0.5+2,0) rectangle (0.5+0.5+2.5,-1) node [midway, rotate=90] {1};
  \draw(0.5+0.5+0.5+2,0) rectangle (0.5+0.5+0.5+2.5,-1) node [midway, rotate=90] {1};

  \draw(0.5+0.5+0.5+2+0.5,0) rectangle (0.5+0.5+0.5+2.5+1,-0.5) node [midway] {1};
    \draw(1+0.5+0.5+0.5+2+0.5,0) rectangle (1+0.5+0.5+0.5+2.5+1,-0.5) node [midway] {1};

 \draw(0.5,0-1) rectangle (1,-1-1) node [midway, rotate=90] {2};
 \draw(1.5,0-1) rectangle (2,-1-1) node [midway, rotate=90] {2};
 \draw(1,0-1) rectangle (1.5,-1-1) node [midway, rotate=90] {2};
  \draw(0.5,0-1) rectangle (0,-1-1) node [midway, rotate=90] {2};

  \draw(-2+0.5+0.5+0.5+2+0.5,0-0.5-0.5) rectangle (-2+0.5+0.5+0.5+2.5+1,-0.5-0.5-0.5) node [midway] {2};
    \draw(-2+1+0.5+0.5+0.5+2+0.5,0-0.5-0.5) rectangle (-2+1+0.5+0.5+0.5+2.5+1,-0.5-0.5-0.5) node [midway] {2};
  \draw(0.5+0.5+0.5+2+0.5,0-0.5) rectangle (0.5+0.5+0.5+2.5+1,-0.5-0.5) node [midway] {2};
    \draw(1+0.5+0.5+0.5+2+0.5,0-0.5) rectangle (1+0.5+0.5+0.5+2.5+1,-0.5-0.5) node [midway] {2};

 \draw(0.5,0-1-1) rectangle (1,-1-1-1) node [midway, rotate=90] {3};
 \draw(0,0-1-1) rectangle (0.5,-1-1-1) node [midway, rotate=90] {3};
 \draw(1,0-1-1) rectangle (2,-1-1-0.5) node [midway] {3};

  \draw(-2+0.5+0.5+0.5+2+0.5,0-0.5-0.5-0.5) rectangle (-2+0.5+0.5+0.5+2.5+1,-0.5-0.5-0.5-0.5) node [midway] {3};
    \draw(-2+1+0.5+0.5+0.5+2+0.5,0-0.5-0.5-0.5) rectangle (-2+1+0.5+0.5+0.5+2.5+1,-0.5-0.5-0.5-0.5) node [midway] {3};
  \draw(0.5+0.5+0.5+2+0.5,0-0.5-0.5) rectangle (0.5+0.5+0.5+2.5+1,-0.5-0.5-0.5) node [midway] {3};
    \draw(1+0.5+0.5+0.5+2+0.5,0-0.5-0.5) rectangle (1+0.5+0.5+0.5+2.5+1,-0.5-0.5-0.5) node [midway] {3};

          \draw(1,-2.5)rectangle (2,-3)node [midway] {4};

  \draw(-2+0.5+0.5+0.5+2+0.5,0-0.5-0.5-0.5-0.5) rectangle (-2+0.5+0.5+0.5+2.5+1,-0.5-0.5-0.5-0.5-0.5) node [midway] {4};
    \draw(-2+1+0.5+0.5+0.5+2+0.5,0-0.5-0.5-0.5-0.5) rectangle (-2+1+0.5+0.5+0.5+2.5+1,-0.5-0.5-0.5-0.5-0.5) node [midway] {4};
  \draw(0.5+0.5+0.5+2+0.5,0-0.5-0.5-0.5) rectangle (0.5+0.5+0.5+2.5+1,-0.5-0.5-0.5-0.5) node [midway] {4};
    \draw(1+0.5+0.5+0.5+2+0.5,0-0.5-0.5-0.5) rectangle (1+0.5+0.5+0.5+2.5+1,-0.5-0.5-0.5-0.5) node [midway] {4};

  \draw(-2+0.5+0.5+0.5+2+0.5,0-0.5-0.5-0.5-0.5-0.5) rectangle (-2+0.5+0.5+0.5+2.5+1,-0.5-0.5-0.5-0.5-0.5-0.5) node [midway] {5};
    \draw(-2+1+0.5+0.5+0.5+2+0.5,0-0.5-0.5-0.5-0.5-0.5) rectangle (-2+1+0.5+0.5+0.5+2.5+1,-0.5-0.5-0.5-0.5-0.5-0.5) node [midway] {5};
  \draw(0.5+0.5+0.5+2+0.5,0-0.5-0.5-0.5-0.5) rectangle (0.5+0.5+0.5+2.5+1,-0.5-0.5-0.5-0.5-0.5) node [midway] {5};
    \draw(1+0.5+0.5+0.5+2+0.5,0-0.5-0.5-0.5-0.5) rectangle (1+0.5+0.5+0.5+2.5+1,-0.5-0.5-0.5-0.5-0.5) node [midway] {5};

  \draw(0.5+0.5+0.5+2+0.5,0-0.5-0.5-0.5-0.5-0.5) rectangle (0.5+0.5+0.5+2.5+1,-0.5-0.5-0.5-0.5-0.5-0.5) node [midway] {6};
    \draw(1+0.5+0.5+0.5+2+0.5,0-0.5-0.5-0.5-0.5-0.5) rectangle (1+0.5+0.5+0.5+2.5+1,-0.5-0.5-0.5-0.5-0.5-0.5) node [midway] {6};

 \draw(0,-3)rectangle (0.5,-4)node [midway,rotate=90] {4};
   \draw (0.5,-3)rectangle (1,-4)node [midway,rotate=90] {4};

  \draw (1,-3)rectangle (2,-3.5)node [midway] {5};
    \draw (1,-4)rectangle (2,-3.5)node [midway] {6};

    \draw (2,-3)rectangle (2.5,-4)node [midway,rotate=90] {6};

   \draw (3.5,-3)rectangle (4,-4)node [midway] {7};

  \draw (1,-3)rectangle (2,-3.5)node [midway] {5};
    \draw (2.5,-3)rectangle (3.5,-3.5)node [midway] {6};

     \draw (2+0.5,-4)rectangle (3+0.5,-3.5)node [midway] {7};

   \draw (4,-3)rectangle (5,-3.5)node [midway] {7};
   \draw (4,-4)rectangle (5,-3.5)node [midway] {8};
             \end{tikzpicture}
 $$
\caption{Example of a pair of   domino \color{black} tableaux $\SSTS$ and $\SSTT$
    of shape $(6^3,5)$  \color{black} and weight $(10,8,7^2,5^2,3,1)$ \color{black}.  The colouring   highlights  the partition $\widehat{\la} = (4,2,1)\subseteq (6^3) \subset (6^3,5) $ and the final double-row.
 }\label{cupbox2}
\end{figure}

 The partition $\widehat{\lambda}$ and the 
 \color{black} multiset of
 labels of the dominoes in the final double row $\mathcal{D}$
 are uniquely determined by the weight $\alpha$.
 \color{black}
  To see this, observe that since $d>b$ for any $d\in \mathcal{D}$, we have that $\widehat{\lambda}_i={\lambda}_i-a=\alpha_i-a$ for $1 \le i \le b$. Then $\widehat{\lambda}$ determines $\lambda$, from which we can read off the elements of $\mathcal{D}$.
What  remains is to determine  the configuration of dominoes of the final double-row and their labelling.

 We claim that there are at most  two  $(1^2)$-dominoes with labels $d, d' >b+1$.
 Every domino which intersects the $(2b+1)$th row must be supported by some
domino which intersects the $2b$th row (exactly as in  \cref{usefulremarkweneed}).
 Since there is precisely one more double column in the $2b$th row   than   in row $(2b+1)$th,
 and
the  rightmost $2(a-\widehat{\la}_b)$ columns of the $2b$th row
 consist solely of $(2)$-dominoes, the claim follows.  
  \color{black} (For example,  in \cref{cupbox2} the rightmost tableau has two $(1^2)$-dominoes with \color{black}
  labels $6,7>3+1=b+1$  \color{black}
  whereas the leftmost tableau has no such $(1^2)$-dominoes.) \color{black}    
We will now construct the final double-row of each of the \color{black} possible domino tableaux from the claim; we do this  \color{black} from right-to-left (as this allows us to verify the lattice permutation condition at each stage).
Given a fixed weight partition $\alpha$, we now provide a pair of algorithms.
The first (second) algorithm   determines the unique element  $\SSTT\in {\sf Dom}((a^b,a-1),\alpha)$ (if it exists) subject to the condition that  there are zero (respectively one or two)
  $(1^2)$-dominoes  of label $ d>b+1$.

 \medskip
 \noindent{\bf Algorithm 1: No $(1^2)$-dominoes of label $\bm{ d>b+1}$}.
 We now provide an algorithm for uniquely determining a tableau of a given weight subject to the condition that  there are no
  $(1^2)$-dominoes  of label $ d>b+1$.
 In what follows, we assume that such a tableau exists.  If such a tableau does not exist, then one of the deductions made during the
 running of the algorithm (for example a statement regarding the differences between labels) will be false.

Set   $W_1:= \mathcal{D}$, the multiset of labels determined by the  weight  $\alpha-\lambda$ (of the final double-row)    and we set $w_1=\max(W_1)$.
 Set $f_1$ equal to the label of  $ {F}_1=\{(2b,2a-1),(2b,2a)\}$.
Set $D_1$ equal to the  bottommost horizontal domino/leftmost vertical domino  in the region $(b,a-1)_2$
and set $d_1$ to be the label of $D_1$.
Set  $E_1 $
and $\overline{E}_1	 $ to be the (at this point empty)  dominoes in $(b+1,a-1)_2$ with
$\overline{E}_1	 $  above
  $E_1$.
Step $i \geq  1$ of the algorithm proceeds as follows:

\begin{itemize}[leftmargin=*]
\item Fill in  $E_i$  with the label   $e_i:=w_i$.
\item
   If $w_i=f_i+1$, then  $E_i$ is supported by  $F_i$; therefore $\overline{E}_i$ must be supported by $ {D}_i$
    and so we fill in   $\overline{E}_i$  with the label   $\overline{e}_i:=d_i+1$.    Now, if  $e_i>\overline{e}_i+1$, then  $E_i$ remains supported by $F_i$ (and
$\overline{E}_i$
is free to support a subsequent empty domino) and so we set $F_{i+1}:=E_i$ and we additionally set   $\delta_i=i$.
On the other hand, if $e_i=\overline{e}_i+1$, then  $E_i$ is now supported by $\overline{E}_i$
(and so $F_i$  remains  free to support a subsequent empty domino)  and we set $F_{i+1}:=F_i$
and we additionally set   $\delta_i=0$.

\item
If $w_i\neq f_i+1$, then  $E_i$ must be supported by $\overline{E}_i$.
  Therefore we fill in   $\overline{E}_i$  with the label   $\overline{e}_i:=w_i-1 \in W_i$.
    Now, if $\overline{e}_i=d_i+1$ then  the domino $\overline{E}_i$ is supported by  ${D}_i$  (and $F_i$  is free to support a subsequent empty domino)
 and so we set  ${F}_{i+1}:={F}_i$.
 On the other hand, if  $\overline{e}_i>d_i+1 $
then $\overline{E}_i$ is supported by   ${F}_i$
  (which by necessity implies that $\overline{e}_i=f_i+1 $ and that ${D}_i$ is free to support a subsequent empty domino) and so  we set ${F}_{i+1}={D}_i$.
 Set $\delta_i=0$.

\item In either case, we now set $W_{i+1} = W_i \setminus \{e_i,\overline{e}_i\}$,
\color{black}
$w_{i+1}=\max(W_{i+1})$,
\color{black} and $D_{i+1}$  equal to the  bottommost horizontal domino/leftmost vertical domino  in the region $(b,a-i-1)_2$
and set $d_{i+1}$ to be the label of $D_{i+1}$.
  If $W_{i+1}$ does not consist solely of labels $b+1$,
then we label the
  top domino   $\overline{E}_{i+1}$ and the bottom domino  $E_{i+1}$ and we commence step $i+1$.
  Otherwise, the algorithm terminates with us placing all the remaining labels in $(1^2)$-dominoes.

\end{itemize}

 The algorithm terminates with
output given by   $\SSTT $.
That the resulting tableau $\SSTT$ belongs to ${\sf Dom}((a^b,a-1) ,\alpha)$ is immediate from the definition of the $i$th step:
 we place the largest possible value in the bottom rightmost $(2)$-domino   (of course) and then place the only
  possible label in the
$(2)$-domino immediately above this (with cases prescribed precisely by the system of parentheses).

 \medskip
 \noindent{\bf Algorithm 2: At least one $(1^2)$-domino of label $\bm{ d>b+1}$}.
  We now provide an algorithm for uniquely determining a tableau of a given weight subject to the condition that  there exists at least one
  $(1^2)$-domino    of label $ d>b+1$.
 In what follows, we assume that such a tableau exists.  If such a tableau does not exist, then one of the deductions made during the
 running of the algorithm (for example a statement regarding the differences between labels) will be false.

 Set   $W_1:= \mathcal{D}$,  the multiset of labels determined by the  weight  $\alpha-\lambda$  (of the final double-row), and set $w_1=\max(W_1)$.  Set $f_1$ equal to the label of  $ {F}_1=\{(2b,2a-1),(2b,2a)\}$.
Set $D_1$ equal to the  bottommost horizontal domino/leftmost vertical domino  in the region $(b,a-1)_2$
and set $d_1$ to be the label of $D_1$.
Step $i\geq 1$ of the algorithm proceeds as follows:
 \begin{itemize}[leftmargin=*]
\item Suppose $F_i$ is in the $2b$th row.
 \begin{itemize}[leftmargin=*]
\item
If   $w_i=f_i+2$,  then
necessarily  $ f_i	+1\in W_i$.  We
 place two $(2)$-dominoes $\overline{E}_i$ and $ {E}_{i}$ in
  $(b+1,a-i)_2$ with ascending labels
   $\overline{e}_i=f_i+1$ and
   $e_i=f_i+2$.
    If $d_i=f_i$ then set $F_{i+1}:=F_i$ and if $d_i<f_i$ then set $F_{i+1}:=D_i$.

\item
If $w_i=f_i+1$, then $d_i+1\in W_i\setminus \{w_i\}$.
\begin{itemize}[leftmargin=*]
\item[$(\clubsuit)$] If
$d_i+2\not\in W_i\setminus \{f_i+1,d_i+1\}$, place a $(1^2)$-domino, $E_i$   in the rightmost position and then place a $(1^2)$-domino, $\overline{E}_i$, in the adjacent position
with labels $e_i=f_i+1$ and $\overline{e}_i= d_i+1$.    Set $F_{i+1}:=\emptyset$.

\item [$(\spadesuit)$]
If
$d_i+2 \in W_i\setminus \{f_i+1,d_i+1\}$,
 then place a $(1^2)$-domino, $V $, in the rightmost position with label $e_i= f_i+1$.
 Then place   a
  $(2)$-domino
$\overline{E}_i$
adjacent to $V $ in the $(2b+1)$th row
with label  $ \overline{e}_i=d_i+1$.
Set $F_{i+1}:=\overline{E}_i$.

\end{itemize}

 \end{itemize}

\item Suppose $F_i$ is in the $(2b+1)$th row.    In this case,  $d_i\neq f_i$ and we must have $d_i+1, f_i+1\in W_i$.
\begin{itemize}[leftmargin=*]
\item
 If $d_i+2\in W_i\setminus \{f_i+1\}$ then place a $(2)$-domino, $E_{i}$, in the rightmost position in the $(2b+2)$th row
 with label $e_i=f_i+1$.
 We then place  a $(2)$-domino, $\overline{E}_{i}$, in the rightmost available position in the $(2b+1)$th row with label $\overline{e}_i=d_i+1$.
 We set $F_{i+1}:=\overline{E}_{i}$.

 \item  If $d_i+2 \not\in W_i\setminus \{f_i+1\}$ then
 place a $(2)$-domino $E_i$ in the rightmost available position in the $(2b+2)$th row with label $e_i=f_i+1$.
Then   place a $(1^2)$-domino   $\overline{V}$  in the adjacent position to the left with label $\overline{e}_i=d_i+1$.
Then set $F_{i+1}=\emptyset$.

\end{itemize}

\item Suppose $F_i=\emptyset$.
  If $W_{i}$ does not consist solely of labels $b+1$, then $d_i+1,d_i+2\in W_i$  and
   we  place  a pair of $(2)$-dominoes  $\overline{E}_{i}$ and   $E_{i}$ with labels $d_i+1$ and $d_i+2$.
  Otherwise, the algorithm terminates with us placing all the remaining labels in $(1^2)$-dominoes.

\item
We now set $W_{i+1} = W_i \setminus \{e_i,\overline{e}_i\}$,
\color{black}
$w_{i+1}=\max(W_{i+1})$,
\color{black}
     and $D_{i+1}$  equal to the  bottommost horizontal domino/leftmost vertical domino  in the region $(b,a-i-1)_2$
and set $d_{i+1}$ to be the label of $D_{i+1}$.

\end{itemize}
  The algorithm terminates with
output given by   $\SSTT $.   That the resulting tableau $\SSTT$ belongs to ${\sf Dom}((a^b,a-1) ,\alpha)$ is immediate from the definition.
It is not immediate that this tableau is unique: in the step $(\spadesuit)$ we have apparently made a choice.
We could have placed two $(2)$-dominoes at this step and set $F_{i+1}:=\overline{E}_i$ in the $(2b+1)$th row.  However, a $(2)$-domino in the $(2b+1)$th row is unable to support a $(1^2)$-domino and so this choice is invalid.

 \smallskip
 \noindent{\bf Uniqueness of sign.  }
Given a weight $\alpha$, each algorithm produces at most one tableau of that weight.
If the second algorithm does not produce a tableau, then the result follows.
Now suppose that the second algorithm does terminate with a tableau $\SSTT$.
 We depict $\SSTT \cap \{(r,c) \mid r\geq 2b, 1\leq c \leq 2a\}$ in \cref{fig1:rowena} below.

\begin{figure}[ht!]
$$\scalefont{0.7}
\begin{tikzpicture} [scale=1.2]

\draw(-3.5,0)--++(0:12-1.5)--++(90:0.5)--++(180:12-1.5)--++(-90:.5);

 \draw(-3.5,0) rectangle (-2.5,-1) node [midway ] {$\dots  $};
       \draw(-3.5,0) rectangle (-2.5,0.5) node [midway ] {$\dots  $};

       \draw(-1.5,0) rectangle (-2.5,0.5) node [midway ] {$ d_{j+2}$};
       \draw(-1.5,0) rectangle (-0.5,0.5) node [midway ] {$\overline{v}-1$};

             \draw(0.5+-1+-2+0,0) rectangle (0.5+-1+-1,-0.5) node [midway ] {$ d_{j+2}+1$};
                    \draw(0.5+-1+-2,-1) rectangle (0.5+-1+-1,-0.5) node [midway ] {$d_{j+2}+2$};

     \draw(-1.5,-1) rectangle (-1,-0) node [midway ] {$\overline{v}$};

   \draw(.5,0) rectangle (-0.5,0.5) node [midway ] {$d_j$};

            \draw(+0,0) rectangle (-1,-0.5) node [midway ] {$d_j+1$};
                    \draw(+-1,-1) rectangle (0,-0.5) node [midway ] {$d_j+2$};

    \draw(2,0.5) rectangle (0.5,0) node [midway ] {$\dots$};

    \draw(1.5,-1) rectangle (0,0) node [midway ] {$\dots$};

           \draw(1+1+0,0) rectangle (1+1+1,0.5) node [midway ] {$d_i$};

                            \draw(-1.5+1+1+1+0,0) rectangle (-1.5+1+1+1+1,-0.5) node [midway ] {$d_{i }+1$};
 \draw(-1.5+1+1+1+0,-1) rectangle (-1.5+1+1+1+1,-0.5) node [midway ] {$d_{i }+2$};

          \draw(1+1+1+0,0) rectangle (1+1+.5,-1) node [midway  ] {$v$};

          \draw(1+1+1+0,0) rectangle (1+1+1+1,0.5) node [midway ] {$v-1$};
          \draw(1+1+1+0,0) rectangle (1+1+1+1,-0.5) node [midway ] {$f_{i-2}+1$};
 \draw(1+1+1+0,-1) rectangle (1+1+1+1,-0.5) node [midway ] {$f_{i-2}+2$};
 \draw(1+3,-1) rectangle (1+4,0) node [midway ] {$\dots$};
 \draw(1+3,0.5) rectangle (1+4,0) node [midway ] {$\dots$};

                                \draw(1+-1+2+4+0,0) rectangle (2+-1+1+4+1,0.5) node [midway ] {$f_{1}$};
                               \draw(1+-1+1+4+0,0) rectangle (1+-1+1+4+1,0.5) node [midway ] {$f_{2}$};
                               \draw(1+-1+1+4+0,0) rectangle (1+-1+1+4+1,-0.5) node [midway ] {$f_{1}+1$};
                    \draw(1+-1+1+4+0,-1) rectangle (1+-1+1+4+1,-0.5) node [midway ] {$f_1+2$};

          \end{tikzpicture}
$$
 \caption{Rows $2b, 2b+1, 2b+2$ of the domino tableau  $\SSTT$ constructed by Algorithm 2.   Note that $v-1=f_{i-1}$. 
 \color{black} Compare with the leftmost domino tableau in \cref{cupbox2}.}
 \label{fig1:rowena}
\end{figure}

If $i-j=-1$ in the above and $v=\overline{v}$, then $\SSTT$ is the unique tableau  in ${\sf Dom}(a^b,a-1,\alpha)$.
To see this, note that algorithms 1 and 2 coincide up to the point in the $(i-2)$th step at which we insert a vertical domino.
 At this point
 $d_{i-1}+1=v=w_{i-1}=\max(W_{i-1})$ and
  $\overline{v}=d_{i-1}+1$  and  so $\overline{v}=v$; thus algorithm~1 fails.

Now assume that $ i-j\geq 0$ or  $\overline{v}\neq v$.
We now describe how to obtain a semistandard tableau $\SSTT^{\rm rot}$ from $\SSTT$ with no
  $(1^2)$-dominoes of label  $d>b+1$, but such that  $\SSTT^{\rm rot}$ has   opposite sign.
  Note that    $\SSTT^{\rm rot}$ will be the output of algorithm~1.

\begin{figure}[ht!]
$$\scalefont{0.7}
\begin{tikzpicture} [scale=1.2]

 \draw(-3.5,0)--++(0:12-1.5)--++(90:0.5)--++(180:12-1.5)--++(-90:.5);

 \draw(-3.5,0) rectangle (-2.5,-1) node [midway ] {$\dots  $};
       \draw(-3.5,0) rectangle (-2.5,0.5) node [midway ] {$\dots  $};

       \draw(-1.5,0) rectangle (-2.5,0.5) node [midway ] {$ d_{j+2}$};
       \draw(-1.5,0) rectangle (-0.5,0.5) node [midway ] {$\overline{v}-1$};

             \draw(0.5+-1+-2+0,0) rectangle (0.5+-1+-1,-0.5) node [midway ] {$ d_{j+2}+1$};
                    \draw(0.5+-1+-2,-1) rectangle (0.5+-1+-1,-0.5) node [midway ] {$d_{j+2}+2$};

    \draw(.5,0) rectangle (-0.5,0.5) node [midway ] {$d_j$};

            \draw(-0.5++0,0) rectangle (-0.5+-1,-0.5) node [midway ] {$\overline{v}$};
                    \draw(-0.5++-1,-1) rectangle (-0.5+0,-0.5) node [midway ] {$d_j+2$};

     \draw(1+-0.5++0,0) rectangle (1+-0.5+-1,-0.5) node [midway ] {$d_j+1$};
                    \draw(1+-0.5++-1,-1) rectangle (1+-0.5+0,-0.5) node [midway ] {$ d_{j-1}+2$};

         \draw(1+1++0,0) rectangle (1+1+-0.5+-1,0.5) node [midway ] {$\dots$};
         \draw(1+1++0,0) rectangle (1+1+-0.5+-1,-1) node [midway ] {$\dots$};

          \draw(1+1+0,0) rectangle (1+1+1,0.5) node [midway ] {$d_i$};
                            \draw(0.5+-1.5+1+1+1+0,0) rectangle (0.5+-1.5+1+1+1+1,-0.5) node [midway ] {$d_{i }+1$};
 \draw(0.5+-1.5+1+1+1+0,-1) rectangle (0.5+-1.5+1+1+1+1,-0.5) node [midway ] {$v$};

          \draw(1+1+1+0,0) rectangle (1+1+1+1,0.5) node [midway ] {$v-1$};
          \draw(1+1+1+0,0) rectangle (1+1+1+1,-0.5) node [midway ] {$f_{i-2}+1$};
 \draw(1+1+1+0,-1) rectangle (1+1+1+1,-0.5) node [midway ] {$f_{i-2}+2$};
 \draw(1+3,-1) rectangle (1+4,0) node [midway ] {$\dots$};

                                \draw(1+-1+2+4+0,0) rectangle (2+-1+1+4+1,0.5) node [midway ] {$f_{1}$};
                               \draw(1+-1+1+4+0,0) rectangle (1+-1+1+4+1,0.5) node [midway ] {$f_{2}$};
                               \draw(1+-1+1+4+0,0) rectangle (1+-1+1+4+1,-0.5) node [midway ] {$f_{1}+1$};
                    \draw(1+-1+1+4+0,-1) rectangle (1+-1+1+4+1,-0.5) node [midway ] {$f_1+2$};

          \end{tikzpicture}
$$
\caption{Rows $2b, 2b+1, 2b+2$ of the domino tableau   $\SSTT^{\rm rot}$.
 \color{black} Compare with the rightmost domino tableau in \cref{cupbox2}.}
\label{fig2:rowena}
\end{figure}

  Given $\SSTT$  as in \cref{fig1:rowena}, we define  $\SSTT^{\rm rot}$ to the tableau obtained from $\SSTT$ as in \cref{fig2:rowena}.  We need only show that    $\SSTT^{\rm rot}$  satisfies the semistandard and lattice permutation conditions.

The lattice permutation can be checked by inspection of   \cref{fig2:rowena}.
That $\SSTT^{\rm rot}$ is weakly increasing along rows follows as each set of row labels
of  $\SSTT^{\rm rot}$ is a subset of the row labels of $\SSTT$.
That the columns increase from the entries in the $2b$th to the $(2b+1)$th row is immediate.
Finally, the column strict inequality $\overline{v}<d_j+2$ in $\SSTT^{\rm rot}$ follows from the
  row  semistandardness inequality  $\overline{v}\leq d_j+1$ of $\SSTT$.
Similarly, $d_k+1 < d_{k-1}+2$   for $i\leq k \leq j$ and  $d_i+1<v$ because
$d_k		  \leq  d_{k-1} $ and  $d_{i}+2\leq v$, both  by the row semistandardness of $\SSTT$.

Therefore the signs of the tableaux (if they both exist) produced in  Algorithms~1 and 2 are opposite and so
$s_ {(2)}\circ s_{(a^b,a-1) }\leq 1$ and $s_ {(1^2)}\circ s_{(a^b,a-1) }\leq 1$ as required.
 \end{proof}

 \begin{cor}
All the products listed in \cref{conj} are   multiplicity-free.

 \end{cor}
 \begin{proof}
 Case $(i)$ is trivial, and cases $(iii)$ and $(iv)$ have been checked by computer.
 Above, we have explicitly checked case $(ii)$ for $\mu=  (a^b)$,     $(a^b,1) $,    $(a^{b-1},a-1)$  and $\mu$ a hook.
 The case $\mu=(a+1,a^{b-1})=(a^b,1)^T $ then follows immediately by \cref{max-conjugate}.
 \end{proof}

 \section{Near maximal constituents of $s_\nu \circ s_{(2)}$} \label{sec:proof}

   For an arbitrary partition $\nu\vdash n$, we calculate
the near maximal (in the lexicographic ordering) constituents    of the product  $s_{\nu}\circ s_{(2)}$ and their multiplicities.
The answer is reminiscent of the famous rule for Kronecker products with the standard representation of the symmetric group.
We expect the results and ideas of this section to be of independent interest; these results will also be vital  in the proof of the classification.

\newcommand{\PSSYT}{{\rm PStd}}
\newcommand{\SSYT}{{\rm SStd}}
\newcommand{\ix}{i_x}

Given $\nu \vdash n$, we have already seen in \cref{pppppppppp} that
$s_{(n+\nu_1,\nu_2,\dots,\nu_\ell)}$ is the lexicographically maximal constituent
of $s_\nu \circ s_{(2)}$, and that
\begin{equation}\label{nu-2-max-constituent}
\langle s_\nu \circ s_{(2)}\mid s_{(n+\nu_1,\nu_2,\dots,\nu_\ell)}\rangle=1 .
\end{equation}
\color{black}
We set $\bar\nu = \nu +(n)$. We first note that if $\la \vdash 2n$ is any partition with $\la_1= n+\nu_1=\bar\nu_1$ then
there is a bijection
$$\PSSYT((2)^\nu, \la) \to \SStd(\bar\nu, \la) , $$
simply given by 
(1) exorcising  the first entry (equal to 1 in every case) of each tableau $\SSTT(r,c)=\gyoung(1;z) $ for $(r,c)\in [\nu]$ and then 
(2) adjoining $n$ additional boxes each containing an entry~1 into the top row. (We note that since $\SSTT$ had $n+\nu_1$ tableau entries~1, this map sends $\SSTT$ to a semistandard $\bar\nu$-tableau whose first row contains only entries~1.) 
  We now use \cref{dvir}: 
$$ \langle s_\nu \circ s_{(2)}\mid s_{\la} \rangle
= |{\rm PStd}((2)^\nu,\lambda)| -
 \sum_{\beta\rhd \lambda}
\langle s_\nu \circ s_{(2)}\mid s_{\beta} \rangle
\times |{\rm SStd}( \beta,\lambda)|$$ to argue that the multiplicity is zero.  If $\bar\nu \ntrianglerighteq \lambda$ then
$|{\rm PStd}((2)^\nu,\lambda)|=|{\rm SStd}(\bar\nu,\lambda)|=0$. Otherwise
$\bar\nu \vartriangleright \lambda$ and the $\beta=\bar\nu$ term cancels with the first term, again giving multiplicity zero.
Therefore
\begin{equation}\label{maxfirstpart}
  \langle s_\nu \circ s_{(2)}\mid s_{\la} \rangle =0 \textrm{ if } \la_1 = n+\nu_1 \textrm{ and } \la \neq {\nu}+(n).
\end{equation}
 \color{black}An example is given in \cref{Rowena_example22}.   \color{black}

 \begin{rmk}\label{aremark}\color{black}
 In the above   one can think of our construction of plethystic tableaux 
 of weight $(n+\nu_1,\dots)$ 
  via the following summary: {\em the position of 
  every single entry equal to $1$ is forced}.   
  Namely, we cannot have a 
$ {\scriptsize\young(11)}$ appearing in  the $r$th row for any $r>1$ (as $ {\scriptsize\young(11)}$ is minimal in the ordering on Young tableaux) and so 
every entry of a plethystic tableau, $\SSTT$, 
must be of the form
$\SSTT(r,c)= {\scriptsize\young(1t)}$ for $r>1$
and 
$\SSTT(1,c)= {\scriptsize\young(11)}$.  This is simply because we have $n+\nu_1$ such entries equal to 1.  
With the $1$s in place, the conditions on the integers $t$ in the entries
  $\SSTT(r,c)={\scriptsize\young(1t)}$ are the same as the conditions on 
  the entries $\stt(r,c)=t$ of some $\stt \in \SStd(\bar\nu, \la)$.  
  \end{rmk}

\begin{figure}[ht!]  $$   \begin{minipage}{5.3cm} \begin{tikzpicture} [scale=0.33]
\draw[very thick](-0.5,0.5)--++(0:15)--++(-90:3)--++(180:15)--++(-90:3)--++(0:12)--++(180:6)
(-0.5,0.5)--++(0:3)--++(-90:9)
(-0.5,0.5)--++(0:6)--++(-90:9)
(-0.5,0.5)--++(0:9)--++(-90:6)
(-0.5,0.5)--++(0:12)--++(-90:6)
(-0.5,0.5)--++(-90:9)--++(0:6)
; 

  \path(-0.2,-0.5) coordinate (origin);
 \begin{scope}
 {\draw (origin)--++(0:2*1.2)--++(-90:1*1.2)--++(180:2*1.2)--++(90:1*1.2); 

      \path(origin)--++(0.5*1.2,-0.5*1.2)  node {1};
   \path(origin)--++(1.5*1.2,-0.5*1.2) node {1};
  \clip (origin)--++(0:2*1.2)--++(-90:1*1.2)--++(180:2*1.2)--++(90:1*1.2);

  \path(origin) coordinate (origin);
   \foreach \i in {1,...,19}
  {
    \path (origin)++(0:1.2*\i cm)  coordinate (a\i);
    \path (origin)++(-90:1.2*\i cm)  coordinate (b\i);
    \path (a\i)++(-90:10cm) coordinate (ca\i);
    \path (b\i)++(0:10cm) coordinate (cb\i);
    \draw[thin] (a\i) -- (ca\i)  (b\i) -- (cb\i); }

     }

     \end{scope}

  \path(3-0.2,-0.5) coordinate (origin);
 \begin{scope}
 {\draw (origin)--++(0:2*1.2)--++(-90:1*1.2)--++(180:2*1.2)--++(90:1*1.2); 

      \path(origin)--++(0.5*1.2,-0.5*1.2)  node {1};
   \path(origin)--++(1.5*1.2,-0.5*1.2) node {1};
  \clip (origin)--++(0:2*1.2)--++(-90:1*1.2)--++(180:2*1.2)--++(90:1*1.2);

  \path(origin) coordinate (origin);
   \foreach \i in {1,...,19}
  {
    \path (origin)++(0:1.2*\i cm)  coordinate (a\i);
    \path (origin)++(-90:1.2*\i cm)  coordinate (b\i);
    \path (a\i)++(-90:10cm) coordinate (ca\i);
    \path (b\i)++(0:10cm) coordinate (cb\i);
    \draw[thin] (a\i) -- (ca\i)  (b\i) -- (cb\i); }

     }

     \end{scope}

    \path(6-0.2,-0.5) coordinate (origin);
  \begin{scope}
 {\draw (origin)--++(0:2*1.2)--++(-90:1*1.2)--++(180:2*1.2)--++(90:1*1.2); 

      \path(origin)--++(0.5*1.2,-0.5*1.2)  node {1};
   \path(origin)--++(1.5*1.2,-0.5*1.2) node {1};
  \clip (origin)--++(0:2*1.2)--++(-90:1*1.2)--++(180:2*1.2)--++(90:1*1.2);

  \path(origin) coordinate (origin);
   \foreach \i in {1,...,19}
  {
    \path (origin)++(0:1.2*\i cm)  coordinate (a\i);
    \path (origin)++(-90:1.2*\i cm)  coordinate (b\i);
    \path (a\i)++(-90:10cm) coordinate (ca\i);
    \path (b\i)++(0:10cm) coordinate (cb\i);
    \draw[thin] (a\i) -- (ca\i)  (b\i) -- (cb\i); }

     }

     \end{scope}

\path(9-0.2,-0.5) coordinate (origin);

 \begin{scope}
 {\draw (origin)--++(0:2*1.2)--++(-90:1*1.2)--++(180:2*1.2)--++(90:1*1.2); 

      \path(origin)--++(0.5*1.2,-0.5*1.2)  node {1};
   \path(origin)--++(1.5*1.2,-0.5*1.2) node {1};
  \clip (origin)--++(0:2*1.2)--++(-90:1*1.2)--++(180:2*1.2)--++(90:1*1.2);

  \path(origin) coordinate (origin);
   \foreach \i in {1,...,19}
  {
    \path (origin)++(0:1.2*\i cm)  coordinate (a\i);
    \path (origin)++(-90:1.2*\i cm)  coordinate (b\i);
    \path (a\i)++(-90:10cm) coordinate (ca\i);
    \path (b\i)++(0:10cm) coordinate (cb\i);
    \draw[thin] (a\i) -- (ca\i)  (b\i) -- (cb\i); }

     }

     \end{scope}

\path(12-0.2,-0.5) coordinate (origin);

 \begin{scope}
 {\draw (origin)--++(0:2*1.2)--++(-90:1*1.2)--++(180:2*1.2)--++(90:1*1.2); 

      \path(origin)--++(0.5*1.2,-0.5*1.2)  node {1};
   \path(origin)--++(1.5*1.2,-0.5*1.2) node {1};
  \clip (origin)--++(0:2*1.2)--++(-90:1*1.2)--++(180:2*1.2)--++(90:1*1.2);

  \path(origin) coordinate (origin);
   \foreach \i in {1,...,19}
  {
    \path (origin)++(0:1.2*\i cm)  coordinate (a\i);
    \path (origin)++(-90:1.2*\i cm)  coordinate (b\i);
    \path (a\i)++(-90:10cm) coordinate (ca\i);
    \path (b\i)++(0:10cm) coordinate (cb\i);
    \draw[thin] (a\i) -- (ca\i)  (b\i) -- (cb\i); }

     }

     \end{scope}

  \path(-0.2,-3.5) coordinate (origin);

 \begin{scope}
 {\draw (origin)--++(0:2*1.2)--++(-90:1*1.2)--++(180:2*1.2)--++(90:1*1.2); 

      \path(origin)--++(0.5*1.2,-0.5*1.2)  node {1};
   \path(origin)--++(1.5*1.2,-0.5*1.2) node {2};   
  \clip (origin)--++(0:2*1.2)--++(-90:1*1.2)--++(180:2*1.2)--++(90:1*1.2);

  \path(origin) coordinate (origin);
   \foreach \i in {1,...,19}
  {
    \path (origin)++(0:1.2*\i cm)  coordinate (a\i);
    \path (origin)++(-90:1.2*\i cm)  coordinate (b\i);
    \path (a\i)++(-90:10cm) coordinate (ca\i);
    \path (b\i)++(0:10cm) coordinate (cb\i);
    \draw[thin] (a\i) -- (ca\i)  (b\i) -- (cb\i); }

     }

     \end{scope}

  \path(2.8,-3.5) coordinate (origin);

 \begin{scope}
 {\draw (origin)--++(0:2*1.2)--++(-90:1*1.2)--++(180:2*1.2)--++(90:1*1.2); 

      \path(origin)--++(0.5*1.2,-0.5*1.2)  node {1};
   \path(origin)--++(1.5*1.2,-0.5*1.2) node {2};
  \clip (origin)--++(0:2*1.2)--++(-90:1*1.2)--++(180:2*1.2)--++(90:1*1.2);

  \path(origin) coordinate (origin);
   \foreach \i in {1,...,19}
  {
    \path (origin)++(0:1.2*\i cm)  coordinate (a\i);
    \path (origin)++(-90:1.2*\i cm)  coordinate (b\i);
    \path (a\i)++(-90:10cm) coordinate (ca\i);
    \path (b\i)++(0:10cm) coordinate (cb\i);
    \draw[thin] (a\i) -- (ca\i)  (b\i) -- (cb\i); }

     }

     \end{scope}

  \path(5.8,-3.5) coordinate (origin);

 \begin{scope}
 {\draw (origin)--++(0:2*1.2)--++(-90:1*1.2)--++(180:2*1.2)--++(90:1*1.2); 

      \path(origin)--++(0.5*1.2,-0.5*1.2)  node {1};
   \path(origin)--++(1.5*1.2,-0.5*1.2) node {2};
  \clip (origin)--++(0:2*1.2)--++(-90:1*1.2)--++(180:2*1.2)--++(90:1*1.2);

  \path(origin) coordinate (origin);
   \foreach \i in {1,...,19}
  {
    \path (origin)++(0:1.2*\i cm)  coordinate (a\i);
    \path (origin)++(-90:1.2*\i cm)  coordinate (b\i);
    \path (a\i)++(-90:10cm) coordinate (ca\i);
    \path (b\i)++(0:10cm) coordinate (cb\i);
    \draw[thin] (a\i) -- (ca\i)  (b\i) -- (cb\i); }

     }

     \end{scope}

  \path(8.8,-3.5) coordinate (origin);

 \begin{scope}
 {\draw (origin)--++(0:2*1.2)--++(-90:1*1.2)--++(180:2*1.2)--++(90:1*1.2); 

      \path(origin)--++(0.5*1.2,-0.5*1.2)  node {1};
   \path(origin)--++(1.5*1.2,-0.5*1.2) node {2};
  \clip (origin)--++(0:2*1.2)--++(-90:1*1.2)--++(180:2*1.2)--++(90:1*1.2);

  \path(origin) coordinate (origin);
   \foreach \i in {1,...,19}
  {
    \path (origin)++(0:1.2*\i cm)  coordinate (a\i);
    \path (origin)++(-90:1.2*\i cm)  coordinate (b\i);
    \path (a\i)++(-90:10cm) coordinate (ca\i);
    \path (b\i)++(0:10cm) coordinate (cb\i);
    \draw[thin] (a\i) -- (ca\i)  (b\i) -- (cb\i); }

     }

     \end{scope}

 \path(-0.2,-6.5) coordinate (origin);

 \begin{scope}
 {\draw (origin)--++(0:2*1.2)--++(-90:1*1.2)--++(180:2*1.2)--++(90:1*1.2); 

      \path(origin)--++(0.5*1.2,-0.5*1.2)  node {1};
   \path(origin)--++(1.5*1.2,-0.5*1.2) node {3};
  \clip (origin)--++(0:2*1.2)--++(-90:1*1.2)--++(180:2*1.2)--++(90:1*1.2);

  \path(origin) coordinate (origin);
   \foreach \i in {1,...,19}
  {
    \path (origin)++(0:1.2*\i cm)  coordinate (a\i);
    \path (origin)++(-90:1.2*\i cm)  coordinate (b\i);
    \path (a\i)++(-90:10cm) coordinate (ca\i);
    \path (b\i)++(0:10cm) coordinate (cb\i);
    \draw[thin] (a\i) -- (ca\i)  (b\i) -- (cb\i); }

     }

     \end{scope}

 \path(3-0.2,-6.5) coordinate (origin);

 \begin{scope}
 {\draw (origin)--++(0:2*1.2)--++(-90:1*1.2)--++(180:2*1.2)--++(90:1*1.2); 

      \path(origin)--++(0.5*1.2,-0.5*1.2)  node {1};
   \path(origin)--++(1.5*1.2,-0.5*1.2) node {3};
  \clip (origin)--++(0:2*1.2)--++(-90:1*1.2)--++(180:2*1.2)--++(90:1*1.2);

  \path(origin) coordinate (origin);
   \foreach \i in {1,...,19}
  {
    \path (origin)++(0:1.2*\i cm)  coordinate (a\i);
    \path (origin)++(-90:1.2*\i cm)  coordinate (b\i);
    \path (a\i)++(-90:10cm) coordinate (ca\i);
    \path (b\i)++(0:10cm) coordinate (cb\i);
    \draw[thin] (a\i) -- (ca\i)  (b\i) -- (cb\i); }

     }

     \end{scope}

 \end{tikzpicture}\end{minipage}\mapsto \;\;
 \begin{minipage}{5.3cm}
 \begin{tikzpicture}
 [scale=0.18]
\draw[very thick](-0.5,0.5)--++(0:15)--++(-90:3)--++(180:15)--++(-90:3)--++(0:12)--++(180:6)
(-0.5,0.5)--++(0:3)--++(-90:9)
(-0.5,0.5)--++(0:6)--++(-90:9)
(-0.5,0.5)--++(0:9)--++(-90:6)
(-0.5,0.5)--++(0:12)--++(-90:6)
(-0.5,0.5)--++(-90:9)--++(0:6)
; 
\draw(1,-1) node {1};
\draw(1+3,-1) node {1};
\draw(1+3+3,-1) node {1};
\draw(1+3+3+3,-1) node {1};
\draw(1+3+3+3+3,-1) node {1};
\draw(1+3+3+3+3+3,-1) node {1};
\draw(1+3+3+3+3+3+3+3+1.5,-1) node {1};

\draw[very thick] (-0.5+15,0.5)--++(0:3)--++(-90:3)--++(180:3); 

\draw[very thick] (-0.5+15+6+1.5,0.5)--++(0:3)--++(-90:3)--++(180:3)--++(90:3);

\draw[very thick,densely dotted] (-0.5+15+1.5,0.5)--++(0:6);
\draw[very thick,densely dotted] (-0.5+15+1.5,0.5-3)--++(0:6);

\draw(1 ,-1-3) node {2};
\draw(1+3 ,-1-3) node {2}; 
\draw(1+3+3 ,-1-3) node {2}; 
\draw(1+3+3+3,-1-3) node {2}; 

\draw(1 ,-1-6) node {3};
\draw(4 ,-1-6) node {3};

\end{tikzpicture}\end{minipage}
$$
\caption{\color{black}A plethystic tableaux $  \SSTT$ of shape   $(2)^{(5,4,2)}$ and maximal weight $(16,4,2)$ is depicted on the left.  
The corresponding semistandard tableau of weight $(16,4,2)$  is depicted on the right.}
\label{Rowena_example22}
\end{figure}
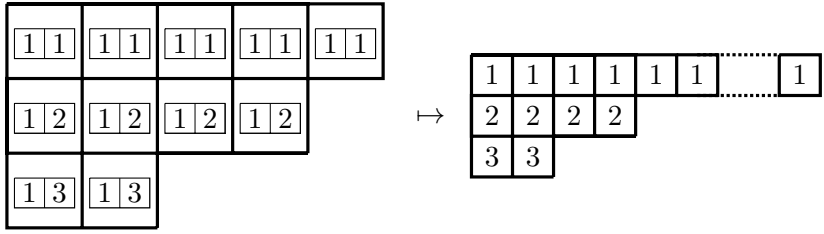

 We will now consider the next layer in the lexicographic ordering,  namely the constituents  labelled by partitions
  $\la \vdash 2n$ with $\la_1=n+\nu_1-1$.
Recall $\bar\nu = \nu +(n)$.

\newcommand{\zedone}{t_1}
\newcommand{\zedtwo}{t_2}
\newcommand{\stwo}{s_1}
\newcommand{\sone}{s_2}
\newcommand{\zedx}{t_X}
\newcommand{\anode}{\sts'(X)}

We already know that $s_{(n)}\circ s_{(2)}$ is multiplicity-free,
so we will now assume that $\nu \ne (n)$.
For the remainder of this section, we will assume that    $\la \vdash 2n$ with $\la_1=n+\nu_1-1$.
      We begin by defining  a map
$$\Phi: \PSSYT((2)^\nu, \la) \to \bigsqcup_{ \begin{subarray}c \beta=\bar\nu-\varepsilon_1-\varepsilon_x+\varepsilon_a+\varepsilon_b
\\x,a,b\ge 2
\end{subarray}} \SSYT(\beta, \la)
\sqcup  \SSYT(\bar\nu , \la),$$
by first breaking  $\PSSYT((2)^\nu, \la) $ into two disjoint subsets as follows. 
 We observe that
any $\SSTT\in \PSSYT((2)^\nu, \la)$ is of one of the following forms:
 \begin{itemize}[leftmargin=*]
\item[$(i)$] we have that $\SSTT(X)= {\scriptsize\young(1\zedx)}$ with $t_X \ge 1$ for  all  $X\in [\nu]$;
 in row 1 there is a unique entry not of the form
 ${\scriptsize\young(11)}\;\!$, namely    $ \SSTT(1,\nu_1)={\scriptsize\young(1t)} $ for  some $t:=t_{(1,\nu_1)}>1$;
 \item[$(ii)$]
 the tableau $\SSTT$  has a unique entry  of the form  $\SSTT(x, \nu_x)={\scriptsize\young(\zedone\zedtwo)}$ for
some $2\leq t_1\leq t_2$ and  $x \ge 2$;
 all other entries
 of $\SSTT$   are of  the form   $\SSTT(X)={\scriptsize\young(1\zedx)}$ with $t_X \ge 1$ for   $X\in [\nu] \setminus (x, \nu_x)$;
  and in particular  $\SSTT(X)={\scriptsize\young(11)}$ for all $X=(1,c)$ for $c\leq \nu_1$.
 \end{itemize}
For an example of the two cases, see \cref{casesiandii}

        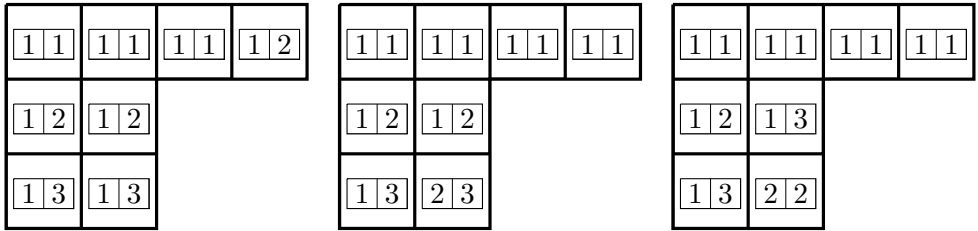
\begin{figure}[ht!]
        \color{black}
  $$    \begin{tikzpicture} [scale=0.33]

 \draw[very thick](-0.5,0.5)--++(0:12)--++(-90:3)--++(180:6)
 --++(-90:6)--++(180:6)
    --++(90:9);

 \clip(-0.5,0.5)--++(0:12)--++(-90:3)--++(180:6)
 --++(-90:6)--++(180:6)
    --++(90:6);

  \path(-0.2,-0.5) coordinate (origin);
 \begin{scope}
 {\draw (origin)--++(0:2*1.2)--++(-90:1*1.2)--++(180:2*1.2)--++(90:1*1.2);

      \path(origin)--++(0.5*1.2,-0.5*1.2)  node {1};
   \path(origin)--++(1.5*1.2,-0.5*1.2) node {1};
  \clip (origin)--++(0:2*1.2)--++(-90:1*1.2)--++(180:2*1.2)--++(90:1*1.2);

  \path(origin) coordinate (origin);
   \foreach \i in {1,...,19}
  {
    \path (origin)++(0:1.2*\i cm)  coordinate (a\i);
    \path (origin)++(-90:1.2*\i cm)  coordinate (b\i);
    \path (a\i)++(-90:10cm) coordinate (ca\i);
    \path (b\i)++(0:10cm) coordinate (cb\i);
    \draw[thin] (a\i) -- (ca\i)  (b\i) -- (cb\i); }

     }

     \end{scope}

  \path(3-0.2,-0.5) coordinate (origin);
 \begin{scope}
 {\draw (origin)--++(0:2*1.2)--++(-90:1*1.2)--++(180:2*1.2)--++(90:1*1.2); 

      \path(origin)--++(0.5*1.2,-0.5*1.2)  node {1};
   \path(origin)--++(1.5*1.2,-0.5*1.2) node {1};
  \clip (origin)--++(0:2*1.2)--++(-90:1*1.2)--++(180:2*1.2)--++(90:1*1.2);

  \path(origin) coordinate (origin);
   \foreach \i in {1,...,19}
  {
    \path (origin)++(0:1.2*\i cm)  coordinate (a\i);
    \path (origin)++(-90:1.2*\i cm)  coordinate (b\i);
    \path (a\i)++(-90:10cm) coordinate (ca\i);
    \path (b\i)++(0:10cm) coordinate (cb\i);
    \draw[thin] (a\i) -- (ca\i)  (b\i) -- (cb\i); }

     }

     \end{scope}

    \path(6-0.2,-0.5) coordinate (origin);

  \begin{scope}
 {\draw (origin)--++(0:2*1.2)--++(-90:1*1.2)--++(180:2*1.2)--++(90:1*1.2); 

      \path(origin)--++(0.5*1.2,-0.5*1.2)  node {1};
   \path(origin)--++(1.5*1.2,-0.5*1.2) node {1};
  \clip (origin)--++(0:2*1.2)--++(-90:1*1.2)--++(180:2*1.2)--++(90:1*1.2);

  \path(origin) coordinate (origin);
   \foreach \i in {1,...,19}
  {
    \path (origin)++(0:1.2*\i cm)  coordinate (a\i);
    \path (origin)++(-90:1.2*\i cm)  coordinate (b\i);
    \path (a\i)++(-90:10cm) coordinate (ca\i);
    \path (b\i)++(0:10cm) coordinate (cb\i);
    \draw[thin] (a\i) -- (ca\i)  (b\i) -- (cb\i); }

     }

     \end{scope}

\path(9-0.2,-0.5) coordinate (origin);

 \begin{scope}
 {\draw (origin)--++(0:2*1.2)--++(-90:1*1.2)--++(180:2*1.2)--++(90:1*1.2); 

      \path(origin)--++(0.5*1.2,-0.5*1.2)  node {1};
   \path(origin)--++(1.5*1.2,-0.5*1.2) node {2};
  \clip (origin)--++(0:2*1.2)--++(-90:1*1.2)--++(180:2*1.2)--++(90:1*1.2);

  \path(origin) coordinate (origin);
   \foreach \i in {1,...,19}
  {
    \path (origin)++(0:1.2*\i cm)  coordinate (a\i);
    \path (origin)++(-90:1.2*\i cm)  coordinate (b\i);
    \path (a\i)++(-90:10cm) coordinate (ca\i);
    \path (b\i)++(0:10cm) coordinate (cb\i);
    \draw[thin] (a\i) -- (ca\i)  (b\i) -- (cb\i); }

     }

     \end{scope}


  \path(-0.2,-3.5) coordinate (origin);

 \begin{scope}
 {\draw (origin)--++(0:2*1.2)--++(-90:1*1.2)--++(180:2*1.2)--++(90:1*1.2); 

      \path(origin)--++(0.5*1.2,-0.5*1.2)  node {1};
   \path(origin)--++(1.5*1.2,-0.5*1.2) node {2};   
  \clip (origin)--++(0:2*1.2)--++(-90:1*1.2)--++(180:2*1.2)--++(90:1*1.2);

  \path(origin) coordinate (origin);
   \foreach \i in {1,...,19}
  {
    \path (origin)++(0:1.2*\i cm)  coordinate (a\i);
    \path (origin)++(-90:1.2*\i cm)  coordinate (b\i);
    \path (a\i)++(-90:10cm) coordinate (ca\i);
    \path (b\i)++(0:10cm) coordinate (cb\i);
    \draw[thin] (a\i) -- (ca\i)  (b\i) -- (cb\i); }

     }

     \end{scope}

  \path(2.8,-3.5) coordinate (origin);

 \begin{scope}
 {\draw (origin)--++(0:2*1.2)--++(-90:1*1.2)--++(180:2*1.2)--++(90:1*1.2); 

      \path(origin)--++(0.5*1.2,-0.5*1.2)  node {1};
   \path(origin)--++(1.5*1.2,-0.5*1.2) node {2};
  \clip (origin)--++(0:2*1.2)--++(-90:1*1.2)--++(180:2*1.2)--++(90:1*1.2);

  \path(origin) coordinate (origin);
   \foreach \i in {1,...,19}
  {
    \path (origin)++(0:1.2*\i cm)  coordinate (a\i);
    \path (origin)++(-90:1.2*\i cm)  coordinate (b\i);
    \path (a\i)++(-90:10cm) coordinate (ca\i);
    \path (b\i)++(0:10cm) coordinate (cb\i);
    \draw[thin] (a\i) -- (ca\i)  (b\i) -- (cb\i); }

     }

     \end{scope}


 \path(-0.2,-6.5) coordinate (origin);

 \begin{scope}
 {\draw (origin)--++(0:2*1.2)--++(-90:1*1.2)--++(180:2*1.2)--++(90:1*1.2); 

      \path(origin)--++(0.5*1.2,-0.5*1.2)  node {1};
   \path(origin)--++(1.5*1.2,-0.5*1.2) node {3};
  \clip (origin)--++(0:2*1.2)--++(-90:1*1.2)--++(180:2*1.2)--++(90:1*1.2);

  \path(origin) coordinate (origin);
   \foreach \i in {1,...,19}
  {
    \path (origin)++(0:1.2*\i cm)  coordinate (a\i);
    \path (origin)++(-90:1.2*\i cm)  coordinate (b\i);
    \path (a\i)++(-90:10cm) coordinate (ca\i);
    \path (b\i)++(0:10cm) coordinate (cb\i);
    \draw[thin] (a\i) -- (ca\i)  (b\i) -- (cb\i); }

     }

     \end{scope}

 \path(3-0.2,-6.5) coordinate (origin);

 \begin{scope}
 {\draw (origin)--++(0:2*1.2)--++(-90:1*1.2)--++(180:2*1.2)--++(90:1*1.2); 

      \path(origin)--++(0.5*1.2,-0.5*1.2)  node {1};
   \path(origin)--++(1.5*1.2,-0.5*1.2) node {3};
  \clip (origin)--++(0:2*1.2)--++(-90:1*1.2)--++(180:2*1.2)--++(90:1*1.2);

  \path(origin) coordinate (origin);
   \foreach \i in {1,...,19}
  {
    \path (origin)++(0:1.2*\i cm)  coordinate (a\i);
    \path (origin)++(-90:1.2*\i cm)  coordinate (b\i);
    \path (a\i)++(-90:10cm) coordinate (ca\i);
    \path (b\i)++(0:10cm) coordinate (cb\i);
    \draw[thin] (a\i) -- (ca\i)  (b\i) -- (cb\i); }

     }

     \end{scope}
  \draw[very thick]
  (-0.5,0.5)--++(-90:9)--++(0:3)--++(90:3) ;

 \draw[very thick]
  (-0.5,0.5)--(8.5,0.5)--(8.5,-2.5)--(5.5,-2.5)--(5.5,-5.5)--(-0.5,-5.5)--(-0.5,0.5);
 \draw[very thick] (5.5,-2.5)--(-0.5,-2.5);
  \draw[very thick] (5.5,-2.5)--(5.5,0.5);
    \draw[very thick] (2.5,-5.5)--(2.5,0.5);
   \end{tikzpicture}  \quad  \begin{tikzpicture} [scale=0.33]

 \draw[very thick](-0.5,0.5)--++(0:12)--++(-90:3)--++(180:6)
 --++(-90:6)--++(180:6)
    --++(90:9);

 \clip(-0.5,0.5)--++(0:12)--++(-90:3)--++(180:6)
 --++(-90:6)--++(180:6)
    --++(90:6);

  \path(-0.2,-0.5) coordinate (origin);
 \begin{scope}
 {\draw (origin)--++(0:2*1.2)--++(-90:1*1.2)--++(180:2*1.2)--++(90:1*1.2);

      \path(origin)--++(0.5*1.2,-0.5*1.2)  node {1};
   \path(origin)--++(1.5*1.2,-0.5*1.2) node {1};
  \clip (origin)--++(0:2*1.2)--++(-90:1*1.2)--++(180:2*1.2)--++(90:1*1.2);

  \path(origin) coordinate (origin);
   \foreach \i in {1,...,19}
  {
    \path (origin)++(0:1.2*\i cm)  coordinate (a\i);
    \path (origin)++(-90:1.2*\i cm)  coordinate (b\i);
    \path (a\i)++(-90:10cm) coordinate (ca\i);
    \path (b\i)++(0:10cm) coordinate (cb\i);
    \draw[thin] (a\i) -- (ca\i)  (b\i) -- (cb\i); }

     }

     \end{scope}

  \path(3-0.2,-0.5) coordinate (origin);
 \begin{scope}
 {\draw (origin)--++(0:2*1.2)--++(-90:1*1.2)--++(180:2*1.2)--++(90:1*1.2); 

      \path(origin)--++(0.5*1.2,-0.5*1.2)  node {1};
   \path(origin)--++(1.5*1.2,-0.5*1.2) node {1};
  \clip (origin)--++(0:2*1.2)--++(-90:1*1.2)--++(180:2*1.2)--++(90:1*1.2);

  \path(origin) coordinate (origin);
   \foreach \i in {1,...,19}
  {
    \path (origin)++(0:1.2*\i cm)  coordinate (a\i);
    \path (origin)++(-90:1.2*\i cm)  coordinate (b\i);
    \path (a\i)++(-90:10cm) coordinate (ca\i);
    \path (b\i)++(0:10cm) coordinate (cb\i);
    \draw[thin] (a\i) -- (ca\i)  (b\i) -- (cb\i); }

     }

     \end{scope}

    \path(6-0.2,-0.5) coordinate (origin);

  \begin{scope}
 {\draw (origin)--++(0:2*1.2)--++(-90:1*1.2)--++(180:2*1.2)--++(90:1*1.2); 

      \path(origin)--++(0.5*1.2,-0.5*1.2)  node {1};
   \path(origin)--++(1.5*1.2,-0.5*1.2) node {1};
  \clip (origin)--++(0:2*1.2)--++(-90:1*1.2)--++(180:2*1.2)--++(90:1*1.2);

  \path(origin) coordinate (origin);
   \foreach \i in {1,...,19}
  {
    \path (origin)++(0:1.2*\i cm)  coordinate (a\i);
    \path (origin)++(-90:1.2*\i cm)  coordinate (b\i);
    \path (a\i)++(-90:10cm) coordinate (ca\i);
    \path (b\i)++(0:10cm) coordinate (cb\i);
    \draw[thin] (a\i) -- (ca\i)  (b\i) -- (cb\i); }

     }

     \end{scope}

\path(9-0.2,-0.5) coordinate (origin);

 \begin{scope}
 {\draw (origin)--++(0:2*1.2)--++(-90:1*1.2)--++(180:2*1.2)--++(90:1*1.2); 

      \path(origin)--++(0.5*1.2,-0.5*1.2)  node {1};
   \path(origin)--++(1.5*1.2,-0.5*1.2) node {1};
  \clip (origin)--++(0:2*1.2)--++(-90:1*1.2)--++(180:2*1.2)--++(90:1*1.2);

  \path(origin) coordinate (origin);
   \foreach \i in {1,...,19}
  {
    \path (origin)++(0:1.2*\i cm)  coordinate (a\i);
    \path (origin)++(-90:1.2*\i cm)  coordinate (b\i);
    \path (a\i)++(-90:10cm) coordinate (ca\i);
    \path (b\i)++(0:10cm) coordinate (cb\i);
    \draw[thin] (a\i) -- (ca\i)  (b\i) -- (cb\i); }

     }

     \end{scope}


  \path(-0.2,-3.5) coordinate (origin);

 \begin{scope}
 {\draw (origin)--++(0:2*1.2)--++(-90:1*1.2)--++(180:2*1.2)--++(90:1*1.2); 

      \path(origin)--++(0.5*1.2,-0.5*1.2)  node {1};
   \path(origin)--++(1.5*1.2,-0.5*1.2) node {2};   
  \clip (origin)--++(0:2*1.2)--++(-90:1*1.2)--++(180:2*1.2)--++(90:1*1.2);

  \path(origin) coordinate (origin);
   \foreach \i in {1,...,19}
  {
    \path (origin)++(0:1.2*\i cm)  coordinate (a\i);
    \path (origin)++(-90:1.2*\i cm)  coordinate (b\i);
    \path (a\i)++(-90:10cm) coordinate (ca\i);
    \path (b\i)++(0:10cm) coordinate (cb\i);
    \draw[thin] (a\i) -- (ca\i)  (b\i) -- (cb\i); }

     }

     \end{scope}

  \path(2.8,-3.5) coordinate (origin);

 \begin{scope}
 {\draw (origin)--++(0:2*1.2)--++(-90:1*1.2)--++(180:2*1.2)--++(90:1*1.2); 

      \path(origin)--++(0.5*1.2,-0.5*1.2)  node {1};
   \path(origin)--++(1.5*1.2,-0.5*1.2) node {2};
  \clip (origin)--++(0:2*1.2)--++(-90:1*1.2)--++(180:2*1.2)--++(90:1*1.2);

  \path(origin) coordinate (origin);
   \foreach \i in {1,...,19}
  {
    \path (origin)++(0:1.2*\i cm)  coordinate (a\i);
    \path (origin)++(-90:1.2*\i cm)  coordinate (b\i);
    \path (a\i)++(-90:10cm) coordinate (ca\i);
    \path (b\i)++(0:10cm) coordinate (cb\i);
    \draw[thin] (a\i) -- (ca\i)  (b\i) -- (cb\i); }

     }

     \end{scope}


 \path(-0.2,-6.5) coordinate (origin);

 \begin{scope}
 {\draw (origin)--++(0:2*1.2)--++(-90:1*1.2)--++(180:2*1.2)--++(90:1*1.2); 

      \path(origin)--++(0.5*1.2,-0.5*1.2)  node {1};
   \path(origin)--++(1.5*1.2,-0.5*1.2) node {3};
  \clip (origin)--++(0:2*1.2)--++(-90:1*1.2)--++(180:2*1.2)--++(90:1*1.2);

  \path(origin) coordinate (origin);
   \foreach \i in {1,...,19}
  {
    \path (origin)++(0:1.2*\i cm)  coordinate (a\i);
    \path (origin)++(-90:1.2*\i cm)  coordinate (b\i);
    \path (a\i)++(-90:10cm) coordinate (ca\i);
    \path (b\i)++(0:10cm) coordinate (cb\i);
    \draw[thin] (a\i) -- (ca\i)  (b\i) -- (cb\i); }

     }

     \end{scope}

 \path(3-0.2,-6.5) coordinate (origin);

 \begin{scope}
 {\draw (origin)--++(0:2*1.2)--++(-90:1*1.2)--++(180:2*1.2)--++(90:1*1.2); 

      \path(origin)--++(0.5*1.2,-0.5*1.2)  node {2};
   \path(origin)--++(1.5*1.2,-0.5*1.2) node {3};
  \clip (origin)--++(0:2*1.2)--++(-90:1*1.2)--++(180:2*1.2)--++(90:1*1.2);

  \path(origin) coordinate (origin);
   \foreach \i in {1,...,19}
  {
    \path (origin)++(0:1.2*\i cm)  coordinate (a\i);
    \path (origin)++(-90:1.2*\i cm)  coordinate (b\i);
    \path (a\i)++(-90:10cm) coordinate (ca\i);
    \path (b\i)++(0:10cm) coordinate (cb\i);
    \draw[thin] (a\i) -- (ca\i)  (b\i) -- (cb\i); }

     }

     \end{scope}
  \draw[very thick]
  (-0.5,0.5)--++(-90:9)--++(0:3)--++(90:3) ;

 \draw[very thick]
  (-0.5,0.5)--(8.5,0.5)--(8.5,-2.5)--(5.5,-2.5)--(5.5,-5.5)--(-0.5,-5.5)--(-0.5,0.5);
 \draw[very thick] (5.5,-2.5)--(-0.5,-2.5);
  \draw[very thick] (5.5,-2.5)--(5.5,0.5);
    \draw[very thick] (2.5,-5.5)--(2.5,0.5);
   \end{tikzpicture}
\quad
\begin{tikzpicture} [scale=0.33]

 \draw[very thick](-0.5,0.5)--++(0:12)--++(-90:3)--++(180:6)
 --++(-90:6)--++(180:6)
    --++(90:9);

 \clip(-0.5,0.5)--++(0:12)--++(-90:3)--++(180:6)
 --++(-90:6)--++(180:6)
    --++(90:6);

  \path(-0.2,-0.5) coordinate (origin);
 \begin{scope}
 {\draw (origin)--++(0:2*1.2)--++(-90:1*1.2)--++(180:2*1.2)--++(90:1*1.2);

      \path(origin)--++(0.5*1.2,-0.5*1.2)  node {1};
   \path(origin)--++(1.5*1.2,-0.5*1.2) node {1};
  \clip (origin)--++(0:2*1.2)--++(-90:1*1.2)--++(180:2*1.2)--++(90:1*1.2);

  \path(origin) coordinate (origin);
   \foreach \i in {1,...,19}
  {
    \path (origin)++(0:1.2*\i cm)  coordinate (a\i);
    \path (origin)++(-90:1.2*\i cm)  coordinate (b\i);
    \path (a\i)++(-90:10cm) coordinate (ca\i);
    \path (b\i)++(0:10cm) coordinate (cb\i);
    \draw[thin] (a\i) -- (ca\i)  (b\i) -- (cb\i); }

     }

     \end{scope}

  \path(3-0.2,-0.5) coordinate (origin);
 \begin{scope}
 {\draw (origin)--++(0:2*1.2)--++(-90:1*1.2)--++(180:2*1.2)--++(90:1*1.2); 

      \path(origin)--++(0.5*1.2,-0.5*1.2)  node {1};
   \path(origin)--++(1.5*1.2,-0.5*1.2) node {1};
  \clip (origin)--++(0:2*1.2)--++(-90:1*1.2)--++(180:2*1.2)--++(90:1*1.2);

  \path(origin) coordinate (origin);
   \foreach \i in {1,...,19}
  {
    \path (origin)++(0:1.2*\i cm)  coordinate (a\i);
    \path (origin)++(-90:1.2*\i cm)  coordinate (b\i);
    \path (a\i)++(-90:10cm) coordinate (ca\i);
    \path (b\i)++(0:10cm) coordinate (cb\i);
    \draw[thin] (a\i) -- (ca\i)  (b\i) -- (cb\i); }

     }

     \end{scope}

    \path(6-0.2,-0.5) coordinate (origin);

  \begin{scope}
 {\draw (origin)--++(0:2*1.2)--++(-90:1*1.2)--++(180:2*1.2)--++(90:1*1.2); 

      \path(origin)--++(0.5*1.2,-0.5*1.2)  node {1};
   \path(origin)--++(1.5*1.2,-0.5*1.2) node {1};
  \clip (origin)--++(0:2*1.2)--++(-90:1*1.2)--++(180:2*1.2)--++(90:1*1.2);

  \path(origin) coordinate (origin);
   \foreach \i in {1,...,19}
  {
    \path (origin)++(0:1.2*\i cm)  coordinate (a\i);
    \path (origin)++(-90:1.2*\i cm)  coordinate (b\i);
    \path (a\i)++(-90:10cm) coordinate (ca\i);
    \path (b\i)++(0:10cm) coordinate (cb\i);
    \draw[thin] (a\i) -- (ca\i)  (b\i) -- (cb\i); }

     }

     \end{scope}

\path(9-0.2,-0.5) coordinate (origin);

 \begin{scope}
 {\draw (origin)--++(0:2*1.2)--++(-90:1*1.2)--++(180:2*1.2)--++(90:1*1.2); 

      \path(origin)--++(0.5*1.2,-0.5*1.2)  node {1};
   \path(origin)--++(1.5*1.2,-0.5*1.2) node {1};
  \clip (origin)--++(0:2*1.2)--++(-90:1*1.2)--++(180:2*1.2)--++(90:1*1.2);

  \path(origin) coordinate (origin);
   \foreach \i in {1,...,19}
  {
    \path (origin)++(0:1.2*\i cm)  coordinate (a\i);
    \path (origin)++(-90:1.2*\i cm)  coordinate (b\i);
    \path (a\i)++(-90:10cm) coordinate (ca\i);
    \path (b\i)++(0:10cm) coordinate (cb\i);
    \draw[thin] (a\i) -- (ca\i)  (b\i) -- (cb\i); }

     }

     \end{scope}


  \path(-0.2,-3.5) coordinate (origin);

 \begin{scope}
 {\draw (origin)--++(0:2*1.2)--++(-90:1*1.2)--++(180:2*1.2)--++(90:1*1.2); 

      \path(origin)--++(0.5*1.2,-0.5*1.2)  node {1};
   \path(origin)--++(1.5*1.2,-0.5*1.2) node {2};   
  \clip (origin)--++(0:2*1.2)--++(-90:1*1.2)--++(180:2*1.2)--++(90:1*1.2);

  \path(origin) coordinate (origin);
   \foreach \i in {1,...,19}
  {
    \path (origin)++(0:1.2*\i cm)  coordinate (a\i);
    \path (origin)++(-90:1.2*\i cm)  coordinate (b\i);
    \path (a\i)++(-90:10cm) coordinate (ca\i);
    \path (b\i)++(0:10cm) coordinate (cb\i);
    \draw[thin] (a\i) -- (ca\i)  (b\i) -- (cb\i); }

     }

     \end{scope}

  \path(2.8,-3.5) coordinate (origin);

 \begin{scope}
 {\draw (origin)--++(0:2*1.2)--++(-90:1*1.2)--++(180:2*1.2)--++(90:1*1.2); 

      \path(origin)--++(0.5*1.2,-0.5*1.2)  node {1};
   \path(origin)--++(1.5*1.2,-0.5*1.2) node {3};
  \clip (origin)--++(0:2*1.2)--++(-90:1*1.2)--++(180:2*1.2)--++(90:1*1.2);

  \path(origin) coordinate (origin);
   \foreach \i in {1,...,19}
  {
    \path (origin)++(0:1.2*\i cm)  coordinate (a\i);
    \path (origin)++(-90:1.2*\i cm)  coordinate (b\i);
    \path (a\i)++(-90:10cm) coordinate (ca\i);
    \path (b\i)++(0:10cm) coordinate (cb\i);
    \draw[thin] (a\i) -- (ca\i)  (b\i) -- (cb\i); }

     }

     \end{scope}


 \path(-0.2,-6.5) coordinate (origin);

 \begin{scope}
 {\draw (origin)--++(0:2*1.2)--++(-90:1*1.2)--++(180:2*1.2)--++(90:1*1.2); 

      \path(origin)--++(0.5*1.2,-0.5*1.2)  node {1};
   \path(origin)--++(1.5*1.2,-0.5*1.2) node {3};
  \clip (origin)--++(0:2*1.2)--++(-90:1*1.2)--++(180:2*1.2)--++(90:1*1.2);

  \path(origin) coordinate (origin);
   \foreach \i in {1,...,19}
  {
    \path (origin)++(0:1.2*\i cm)  coordinate (a\i);
    \path (origin)++(-90:1.2*\i cm)  coordinate (b\i);
    \path (a\i)++(-90:10cm) coordinate (ca\i);
    \path (b\i)++(0:10cm) coordinate (cb\i);
    \draw[thin] (a\i) -- (ca\i)  (b\i) -- (cb\i); }

     }

     \end{scope}

 \path(3-0.2,-6.5) coordinate (origin);

 \begin{scope}
 {\draw (origin)--++(0:2*1.2)--++(-90:1*1.2)--++(180:2*1.2)--++(90:1*1.2); 

      \path(origin)--++(0.5*1.2,-0.5*1.2)  node {2};
   \path(origin)--++(1.5*1.2,-0.5*1.2) node {2};
  \clip (origin)--++(0:2*1.2)--++(-90:1*1.2)--++(180:2*1.2)--++(90:1*1.2);

  \path(origin) coordinate (origin);
   \foreach \i in {1,...,19}
  {
    \path (origin)++(0:1.2*\i cm)  coordinate (a\i);
    \path (origin)++(-90:1.2*\i cm)  coordinate (b\i);
    \path (a\i)++(-90:10cm) coordinate (ca\i);
    \path (b\i)++(0:10cm) coordinate (cb\i);
    \draw[thin] (a\i) -- (ca\i)  (b\i) -- (cb\i); }

     }

     \end{scope}
  \draw[very thick]
  (-0.5,0.5)--++(-90:9)--++(0:3)--++(90:3) ;

 \draw[very thick]
  (-0.5,0.5)--(8.5,0.5)--(8.5,-2.5)--(5.5,-2.5)--(5.5,-5.5)--(-0.5,-5.5)--(-0.5,0.5);
 \draw[very thick] (5.5,-2.5)--(-0.5,-2.5);
  \draw[very thick] (5.5,-2.5)--(5.5,0.5);
    \draw[very thick] (2.5,-5.5)--(2.5,0.5);
   \end{tikzpicture}
   $$       
   \caption{\color{black}The plethystic tableaux $\SSTT_1,\SSTT_2,\SSTT_3\in  \PStd((2)^{(4,2^2)} , (11,3,2))$.
   The plethystic tableau $\SSTT_1$ is of the form $(i)$ and   
 $  \SSTT_2, \SSTT_3$  are of the form $(ii)$.  
     }
   \label{casesiandii}

   \end{figure}

\begin{rmk}\color{black}
In \cref{aremark} we saw that we had zero choices for where to place the
 $n+\nu_1$ integers equal to  $1$ within a plethystic tableau of shape $\nu$ and weight $(n+\nu_1,\dots)$.  
 The two cases $(i)$ and $(ii)$ above arise as there is precisely one choice 
 as to where {\em not} to put an entry $1$ in a plethystic tableau
 of shape $\nu$ and weight $(n+\nu_1-1,\dots)$.  
 
\end{rmk}

We  now define a tableau $\stt$ in each of the  cases $(i)$ and $(ii)$ above as follows, and then set  $\Phi(\SSTT)=\stt$.
\color{black}

 \smallskip
\noindent
{\bf Case $(i)$. }
We 
define a tableau of shape $\bar\nu$.    Set 
\color{black} $\stt(1,c) = 1$ for all $1\leq c <  n+\nu_1$ and $\stt(1,n+\nu_1) =t_{(1,\nu_1)}$.
For the remaining nodes, $X\in [\nu_{>1}]$, we set $\stt(X)= t_X$ (where $\SSTT(X)= {\scriptsize\young(1\zedx)}$).

 \smallskip
\noindent
{\bf Case $(ii)$. }
  We first define a semistandard $(\bar\nu-\epsilon_x)$-tableau $\bar{\stt}$ by setting $\bar{\stt}(1,c)=1$ for all $1\leq c \leq n+\nu_1-1$ and $\bar{\stt}(X)=t_X$ if $\SSTT(X)= {\scriptsize\young(1\zedx)}$).
 \color{black}
   We then let $\stt$ be the tableau obtained from $\bar{\stt}$ by applying the RSK bumping algorithm to insert $t_1$ into row 2 (resulting in the addition of a box in the $a$th row for some  $a\geq 2$) followed by $t_2$ into row 2 (resulting in a box added into the $b$th row  for some $2\leq b \leq a$).

For an example of the two cases, see \cref{casesiandii2}.

\begin{rmk} \label{maprmk}
We note that in case $(i)$, $\Phi(\SSTT) \in \SStd(\bar{\nu},\la)$  and in case (ii) $\Phi(\SSTT) \in \SStd(\beta,\la)$ for $ \beta=\bar\nu-\varepsilon_1-\varepsilon_x+\varepsilon_a+\varepsilon_b$ where  the shape $\beta$ is determined by
the numbers $a,b$ with  $2\leq b \leq a$  produced via the RSK bumping.
We emphasise that since the two RSK applications will never add two boxes in the same column, we must have that $\nu_a\neq \nu_b$ whenever $a\neq b$.
\end{rmk}
  
\begin{eg}\color{black} 
We now illustrate the effect of the map $\Phi$ on the  plethystic tableaux $\SSTT_1, \SSTT_2,\SSTT_3 $ of \cref{casesiandii}.  The semistandard tableaux $\Phi(\SSTT_1), \Phi(\SSTT_2)$ and $ \Phi(\SSTT_3)$ are depicted in \cref{casesiandii2}.  
The tableau $\Phi(\SSTT_1)$ is easily calculated; we simply remove the   initial entries 1 from 
 each $\SSTT(r,c)= {\scriptsize\young(1t)}$ for each $r\geq 1$ and adjoin these to row 1.    
 
To compute $\Phi(\SSTT_2)$, we first move the 1 entries as above.  
Next, we observe the unique entry of $\SSTT_2$  not containing 1 occurs in the removable box
$\SSTT(3,2)= {\scriptsize\young(23)}$ in row $x = 3$; we remove this box and its entries.  
Then add the removed numbers 2, 3 to the second row   (shown in blue). The result is the Young tableau $\Phi(\SSTT_2)$.

To compute $\Phi(\SSTT_3)$, we first move the 1 entries as above.  
Next, we observe the unique entry of $\SSTT_3$  not containing 1 occurs in the removable box
$\SSTT(3,2)= {\scriptsize\young(22)}$ in row $x = 3$; we remove this box and its entries.  
Then add the removed numbers 2, 2 to the second row   (shown in cyan) and, using the RSK bumping algorithm, displace the entry 3 to the third row (shown in pink). 
In more detail: in the first addition, the 2 bumps the entry 3 from the second row   into the third row; 
as the third row   consists only of entries 3 there are no further bumps. 
The second insertion simply adds the entry 2 to the right of the second row.  
The result is the Young tableau $\Phi(\SSTT_3)$.

\begin{figure}[ht!] \color{black} 
  $$\begin{minipage}{4.9cm}
 \begin{tikzpicture}
 [scale=0.18]
 
  \draw[very thick](-0.5,0.5)--++(0:9)--++(-90:3)--++(180:3)--++(-90:6)--++(180:6)--++(90:9);
  
  \draw[very thick](-0.5,0.5)--++(0:6) --++(-90:3);
    \draw[very thick](-0.5,0.5)--++(0:3) --++(-90:9);
    
      \draw[very thick](-0.5,0.5)--++(-90:3) --++(0:6);
      \draw[very thick](-0.5,0.5)--++(-90:6) --++(0:6);
  
 \draw(1,-1) node {1};
\draw(1+3,-1) node {1};
\draw(1+3+3,-1) node {1};
\draw(1+3+3+3,-1) node {1};

\draw(1+3+3+3+3+4.5,-1) node {1}; 
\draw(1+3+3+3+3+3+4.5,-1) node {2};

\draw[very thick] (-0.5+15-3-3,0.5)--++(0:3)--++(-90:3)--++(180:3); 

\draw[very thick] (-0.5+15+6+1.5-3-3,0.5)--++(0:3)--++(-90:3)--++(180:3)--++(90:3); 
\draw[very thick] (-0.5+15+6+1.5-3,0.5)--++(0:3)--++(-90:3)--++(180:3)--++(90:3);

\draw[very thick,densely dotted] (-0.5+15+1.5-6,0.5)--++(0:6);
\draw[very thick,densely dotted] (-0.5+15+1.5-6,0.5-3)--++(0:6);

\draw(1 ,-1-3) node {2};
\draw(1+3 ,-1-3) node {2}; 

\draw(1 ,-1-6) node {3};
\draw(4 ,-1-6) node {3};

\end{tikzpicture}\end{minipage}  
\quad
  \begin{minipage}{4.9cm}
 \begin{tikzpicture}
 [scale=0.18]
 
  \draw[very thick](-0.5,0.5)--++(0:9)--++(-90:3)--++(180:3)--++(-90:3)--++(180:3)--++(-90:3)--++(180:3)--++(90:9);
  
  \draw[very thick](-0.5,0.5)--++(0:6) --++(-90:3);
    \draw[very thick](-0.5,0.5)--++(0:3) --++(-90:9);
    
      \draw[very thick](-0.5,0.5)--++(-90:3) --++(0:6);
      \draw[very thick](-0.5,0.5)--++(-90:6) --++(0:6);
  
 \draw(1,-1) node {1};
\draw(1+3,-1) node {1};
\draw(1+3+3,-1) node {1};
\draw(1+3+3+3,-1) node {1};

\draw(1+3+3+3+3+4.5,-1) node {1}; 
 \draw(1+3+3+3+3+3+4.5,-1) node {$\times$};

\draw[very thick] (-0.5+15-3-3,0.5)--++(0:3)--++(-90:3)--++(180:3); 

\draw[very thick] (-0.5+15+6+1.5-3-3,0.5)--++(0:3)--++(-90:3)--++(180:3)--++(90:3); 

\draw[very thick,densely dotted] (-0.5+15+1.5-6,0.5)--++(0:6);
\draw[very thick,densely dotted] (-0.5+15+1.5-6,0.5-3)--++(0:6);

\draw(1 ,-1-3) node {2};
\draw(1+3 ,-1-3) node {2}; 
\draw(1+3+3 ,-1-3) node {\color{cyan}2}; 
\draw(1+3+3+3,-1-3) node {\color{cyan}3}; 
\draw[very thick] (-0.5+6+6-6,0.5-3)--++(0:3)--++(-90:3)--++(180:3)--++(90:3); 
\draw[very thick] (-0.5+6+6-3,0.5-3)--++(0:3)--++(-90:3)--++(180:3)--++(90:3);

\draw(1 ,-1-6) node {3};
 
\end{tikzpicture}\end{minipage}  
\quad
  \begin{minipage}{4.9cm}
 \begin{tikzpicture}
 [scale=0.18]
 
  \draw[very thick](-0.5,0.5)--++(0:9)--++(-90:3)--++(180:3)--++(-90:3)--++(180:3)--++(-90:3)--++(180:3)--++(90:9);
  
  \draw[very thick](-0.5,0.5)--++(0:6) --++(-90:3);
    \draw[very thick](-0.5,0.5)--++(0:3) --++(-90:9);
    
      \draw[very thick](-0.5,0.5)--++(-90:3) --++(0:6);
      \draw[very thick](-0.5,0.5)--++(-90:6) --++(0:6);
  
 \draw(1,-1) node {1};
\draw(1+3,-1) node {1};
\draw(1+3+3,-1) node {1};
\draw(1+3+3+3,-1) node {1};

\draw(1+3+3+3+3+4.5,-1) node {1}; 
 \draw(1+3+3+3+3+3+4.5,-1) node {$\times$};

\draw[very thick] (-0.5+15-3-3,0.5)--++(0:3)--++(-90:3)--++(180:3); 

\draw[very thick] (-0.5+15+6+1.5-3-3,0.5)--++(0:3)--++(-90:3)--++(180:3)--++(90:3); 

\draw[very thick,densely dotted] (-0.5+15+1.5-6,0.5)--++(0:6);
\draw[very thick,densely dotted] (-0.5+15+1.5-6,0.5-3)--++(0:6);

\draw(1 ,-1-3) node {2};
\draw(1+3 ,-1-3) node {\color{cyan}2}; 
\draw(1+3+3 ,-1-3) node {\color{cyan}2}; 
 \draw(1+3+3+3-3-3,-1-3-3) node {\color{magenta}3}; 
\draw[very thick] (-0.5+6+6-6,0.5-3)--++(0:3)--++(-90:3)--++(180:3)--++(90:3); 
\draw[very thick] (-0.5+6+6-3-3-3,0.5-3-3)--++(0:3)--++(-90:3)--++(180:3)--++(90:3);

\draw(1 ,-1-6) node {3};
 
\end{tikzpicture}\end{minipage}  
   $$
\caption{\color{black} The images under $\Phi$ of the plethystic tableaux $\SSTT_1, \SSTT_2, \SSTT_3$ of \cref{casesiandii} respectively.  
The shapes of these tableaux are $(12,2^2)$, $(11,4,1)$ and $(11,3,2)$ respectively.  We have used   $\times$ to illustrate that the latter two tableaux have shorter first rows (by one box).   }
\label{casesiandii2}
\end{figure}
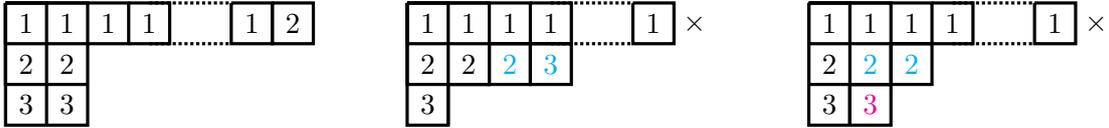
\end{eg}

We let
$M(\nu)$ be the set of all partitions $\beta \vdash 2n$ such that   $\beta$ can be obtained from $\bar\nu-\varepsilon_1$ by first removing a node from $\bar\nu-\varepsilon_1$
in row $x>1$ and then adding two nodes in rows $a \geq b \geq 2$
where $\beta_a\neq \beta_b$ if $a\neq b$.
 In particular,  $\beta$ can be written in the form
$\beta= \bar\nu-\varepsilon_1-\varepsilon_x+\varepsilon_a+\varepsilon_b$ for some
$2\leq a,b,x$ with conditions as above.
A partition $\beta\in M(\nu)$ may be obtained in the form
$\beta=\bar\nu-\varepsilon_1-\varepsilon_x+\varepsilon_a+\varepsilon_b$
for different choices of $a\geq b$ satisfying the conditions above ($x$ is then uniquely determined);
we note that $\beta$ has only one such form if $x\not\in \{a,b\}$.
We let $I(\beta)$ be the set of possible pairs $(a,b)$ for $\beta$ as above.

\subsection{The case $\nu_1\ne \nu_2$. }

 \begin{prop}\label{4.2}
Let $\nu\vdash n $ with $\nu_1\neq \nu_2$.  Let  $\la \vdash 2n$ with $\la_1=n+\nu_1-1$.
 The following map is a bijection:
\begin{equation}\label{ourmap}
\widehat{\Phi}:  \PSSYT((2)^\nu, \la) \to
\SStd(\bar\nu,\lambda) \sqcup\Bigg(
 \bigsqcup_{
 \begin{subarray}c
\beta \in M(\nu) \\
\beta \trianglerighteq \la
\end{subarray}}
 ( \SSYT(\beta, \la) \times I(\beta))\Bigg)
\end{equation}
given by $\widehat{\Phi}(\SSTT)=\Phi(\SSTT)$ in case $(i)$, and in case $(ii)$  $\widehat{\Phi}(\SSTT)$ is equal to $ (\Phi(\SSTT), (a,b))$, with
$(a,b)$ obtained in the RSK bumping.
\end{prop}
\begin{proof}
The fact that $\widehat{\Phi}$ is a well-defined map follows from the definition of $\Phi$ and
$(*)$ above.
We shall now prove that $\widehat{\Phi}$ is bijective.  Finding the preimage in case $(i)$ is trivial.
We now consider case $(ii)$.  Suppose
 that  $\beta=\bar\nu-\varepsilon_1-\varepsilon_x+\varepsilon_a+\varepsilon_b$
 with $(a,b)\in I(\beta)$.
We can apply reverse RSK to $\sts \in \SStd(\beta,\la)$ to remove nodes from the $b$th and then $a$th rows and hence obtain a unique tableau $\sts'$ and a pair of integers $s_1\leq s_2$ removed from the tableau.  We set $\SSTS$ to be the plethystic tableau obtained by letting
$$\SSTS(X)=
\Yboxdim{29pt}\gyoung(1;\anode)\qquad
\SSTS(x,\nu_x)=\Yboxdim{29pt}\gyoung(\stwo;\sone) $$ for $X \in [\nu-\varepsilon_x] $.
This provides the required inverse map.
 \end{proof}

\begin{cor}\label{yaaaas}
Let $\nu \vdash n$ with $\nu_1\neq \nu_2$ and  $\la \vdash 2n$ with $\la_1=n+\nu_1-1$.    We have that
$$\langle s_\nu \circ s_{(2)}, s_\la \rangle = \begin{cases}
 \color{black}  |I(\la) |= \color{black}
1 & \textrm{if }\la=\bar\nu-\varepsilon_1-\varepsilon_x+\varepsilon_a+\varepsilon_b \textrm{ for }  x \neq a,b, \nu_a\neq \nu_b \text{ if } a\neq b\\
 |I(\la) |
& \textrm{if }\la=\bar\nu-\varepsilon_1+\varepsilon_c  \textrm{ for some } c>1\\
0 & \textrm{otherwise}.
\end{cases}$$
\end{cor}
\begin{proof}
 For partitions $\pi$ with $\pi_1 =n+\nu_1$, we have already seen that
$
 \langle s_\nu \circ s_{(2)}, s_\pi \rangle =1
$  or 0 if $\pi$ is or is not equal to $\bar\nu$, respectively.
\color{black}
For partitions $\lambda$ with $\lambda_1=n+\nu_1-1$, we have already observed that $|I(\lambda)|=1$ in the first case listed in the corollary. We proceed inductively. By  \cref{dvir} together with the bijection of \cref{4.2},
$$\langle s_\nu \circ s_{(2)}, s_\la \rangle = |\SStd(\bar\nu, \la)|+\sum_{ \begin{subarray}c
\beta \in M(\nu) \\
\beta \trianglerighteq \la
\end{subarray}}|I(\lambda)|  \times |\SStd(\beta, \la)| - \sum_{\beta \rhd \la}\langle s_\nu \circ s_{(2)}, s_\beta \rangle  \times |\SStd(\beta, \la)|.
$$
The inductive hypothesis allows the cancellation of each term of the first sum corresponding to $\beta \rhd \la$ with the corresponding term in the second sum. In the first sum, we are left with the term for $\beta=\la$ if $\la \in M(\nu)$ and nothing otherwise. In the second sum, only the term for $\beta=\bar\nu$ remains and only provided $\bar\nu \rhd \la$. This term equals $1 \times |\SStd(\bar\nu, \la)|$ and cancels with the initial term (and, in the case where $\bar\nu \ntrianglerighteq \la$, the initial term is zero). Thus all terms cancel except $|I(\la)|  \times |\SStd(\la, \la)|=|I(\la)|$ in the two cases where $\la \in M(\nu)$, as claimed.
\color{black}
\end{proof}

\subsection{The case $\nu_1=\nu_2$. }
In the previous section, we made the assumption that $\nu_1\neq \nu_2$ in order to guarantee
that    \cref{ourmap} was  a bijection.  If $\nu_1=\nu_2$ then this map is not surjective.  In fact, we have the following.

 \begin{prop}
Let $\nu\vdash n $ with $\nu_1=\nu_2$.  Let  $\la \vdash 2n$ with $\la_1=n+\nu_1-1$.
 The following map is a bijection:
\begin{equation}\label{ourmap2}
\tilde{\Phi}:  \PSSYT((2)^\nu, \la) \to
\SStd( \nu,\lambda-(n)) \sqcup\Bigg(
 \bigsqcup_{
 \begin{subarray}c
\beta \in M(\nu) \\
\beta \trianglerighteq \la
\end{subarray}}
 ( \SSYT(\beta, \la) \times I(\beta))\Bigg)
\end{equation}
given, in case $(i)$, by   $\tilde{\Phi}(\SSTT )$  obtained by deleting all initial 1s in all tableaux entries of $\SSTT$ and,
 in case $(ii)$,   $\tilde{\Phi}(\SSTT )=(\Phi(\SSTT), (a,b))$ with
$(a,b)$ obtained in the RSK bumping.
\end{prop}

 The proof is identical to that of \cref{4.2}.

\begin{cor}\label{yaaaas2}
Let $\nu \vdash n$ with $\nu_1= \nu_2$ and  $\la \vdash 2n$ with $\la_1=n+\nu_1-1$.    We have that
$$\langle s_\nu \circ s_{(2)}, s_\la \rangle =
 \begin{cases}
1 & \textrm{if }\la=\bar\nu-\varepsilon_1-\varepsilon_x+\varepsilon_a+\varepsilon_b \textrm{ for }  x \neq a,b, \nu_a\neq \nu_b \text{ if } a\neq b\\
 |	I(\la) | -1
& \textrm{if }\la=\bar\nu-\varepsilon_1+\varepsilon_2   \\
 |I(\la) |
& \textrm{if }\la=\bar\nu-\varepsilon_1+\varepsilon_c  \textrm{ for some } c>2\\
0 & \textrm{otherwise}.
\end{cases}$$
\end{cor}

\begin{proof}
One proceeds as in \cref{yaaaas} and reduces the problem to constructing the following equality
$$
|\SStd(\bar\nu,\la)|
=
|\SStd( \nu , \la-(n))|
+
|\SStd(\bar\nu-\varepsilon_1+ \varepsilon_2, \la)|
.$$
The bijection $\tilde{\phi}$  behind this equality is given as follows. If
$\stt \in \SStd(\bar\nu,\la)$  is such that $\stt({1,\nu_1+n})<\stt(2,\nu_2)$ then
$\tilde{\phi}(\stt)$ is obtained by deleting a total of $n$ entries equal to $1$ from the first row of $\stt$ (so $\tilde{\phi}$ is semistandard as  $\stt({1,\nu_1+n})<\stt(2,\nu_2)$).
If $\stt \in \SStd(\bar\nu,\la)$  is such that $\stt({1,\nu_1+n})\geq \stt(2,\nu_2)$,
then move the final box in row 1 containing entry $\stt({1,\nu_1+n})$ and add this box to the end of row 2.
\end{proof}

 \section{Proof of the classification} \label{sec:endproof}

We are now  ready to prove the converse of the main theorem, namely that any product not on the list  of \cref{conj} does indeed contain multiplicities.   The idea of the proof is as follows: we first  calculate ``seeds of multiplicity" using
plethystic tableaux
and then we ``grow" these seeds to infinite families of products  $s_\nu \circ s_\mu$ containing coefficients which are strictly greater than 1.  We shall provide an example of this procedure below and then afterwards explain the idea of the proof in detail.  We organise this section according to the outer partition   ---  in more detail, each result
 of this section proves \cref{conj} under some restriction on $\nu$ (that $\nu$ has 3 removable nodes, is a proper fat hook, rectangle, 2-line, linear partition) until we have exhausted all possibilities.

  \cref{yaaaas} provided our first ``seed", which we will now ``grow" as follows.

 \begin{prop}\label{coolcor}\label{3removable}
Let  $\nu$ be a partition with $\rem(\nu)\geq 3$.
Then $p(\nu,\mu)>1$
for any  partition
$\mu  $ such that   $|\mu|> 1$.
 \end{prop}

 \begin{proof}
Let $\mathcal N$ be the set of all partitions $\nu$ with
$\rem(\nu)\geq 3$.
Let $\nu \in \mathcal N$.
By \cref{yaaaas} and \cref{yaaaas2}
we have
$$2 \leq \langle s_\nu\circ s_{(2)} \mid s_{\bar\nu -\varepsilon_1+\varepsilon_2}\rangle ,
$$
and thus $p(\nu,(2))>1$.
As $\mathcal N$ is closed under conjugation, the result now follows by
\cref{cor:2-important}.
 \end{proof}

    It now only remains to consider all products of the form $s_\nu \circ s_\mu$  such that $\nu$ has at most 2 removable nodes.  As these products are ``closer to being on our list" we have to delve deeper into the dominance order if we are to find the desired multiplicities.
 \begin{prop}\label{the latter}
  Let $ \nu= (a^b)  \supseteq (2^3)$ be a rectangle.
Then
\begin{equation}\label{equality1}
\langle s_\nu \circ s_{(2)} \mid s_{\bar\nu-2\varepsilon_1+2\varepsilon_2} \rangle =2 .
\end{equation}
 \end{prop}
 \begin{proof}
  The partitions $\lambda$ satisfying
 $$
\bar\nu \succeq
\lambda
\rhd
\bar\nu-2\varepsilon_1+2\varepsilon_2 \quad \text{ and }\quad \PStd((2)^{(a^b)},\la)\neq \emptyset
  $$
are  obtained from $\bar\nu $ by \\
(1) removing $i\leq 2$ nodes   from the first row of $\bar\nu$,\\
(2) removing at most $i$  nodes   from the final ($b$th and $(b-1)$th) rows of  $\bar\nu$, \\
(3)  adding these nodes in rows  with indices strictly greater than 1 and strictly less than $b$.
  The partitions satisfying these criteria are
  $$
\bar\nu,
 \quad
 \alpha=\bar\nu-\varepsilon_1-\varepsilon_b+2\varepsilon_2,
 \quad
 \beta_{(4)}=\bar\nu-2\varepsilon_1-2\varepsilon_b+4\varepsilon_2,
 \quad
  \beta_{(3,1)}=\bar\nu-2\varepsilon_1-2\varepsilon_b+3\varepsilon_2+\varepsilon_3,
$$
$$
  \beta_{(2,2)}=\bar\nu-2\varepsilon_1-2\varepsilon_b+2\varepsilon_2+2\varepsilon_3,
  \quad
    \beta_{(2,1,1)}=\bar\nu-2\varepsilon_1-2\varepsilon_b+2\varepsilon_2+ \varepsilon_3+ \varepsilon_4,
    $$
        $$
        \gamma_{(4)}=\bar\nu-2\varepsilon_1- \varepsilon_b- \varepsilon_{b-1}+4\varepsilon_2,
        \quad
          \gamma_{(3,1)}=\bar\nu-2\varepsilon_1-\varepsilon_b-\varepsilon_{b-1}+3\varepsilon_2+\varepsilon_3,$$
          $$
 \gamma_{(2,2)}=\bar\nu-2\varepsilon_1-\varepsilon_b-\varepsilon_{b-1}+2\varepsilon_2+2\varepsilon_3,
  \quad
    \gamma_{(2,1,1)}=\bar\nu-2\varepsilon_1-\varepsilon_b-\varepsilon_{b-1}+2\varepsilon_2+ \varepsilon_3+ \varepsilon_4,
$$
 $$
    \zeta_{(3)}= \bar\nu-2\varepsilon_1-\varepsilon_b+3\varepsilon_2,
    \quad
    \zeta_{(2,1)}=  \bar\nu-2\varepsilon_1-\varepsilon_b+2\varepsilon_2+ \varepsilon_3,
  $$
$$
    \delta= \bar\nu-\varepsilon_1+\varepsilon_2,
    \quad
    \omega=  \bar\nu-2\varepsilon_1+2\varepsilon_2.
  $$
The   Hasse diagram of these partitions, under the dominance ordering, is depicted in \cref{hasse33}, below.

 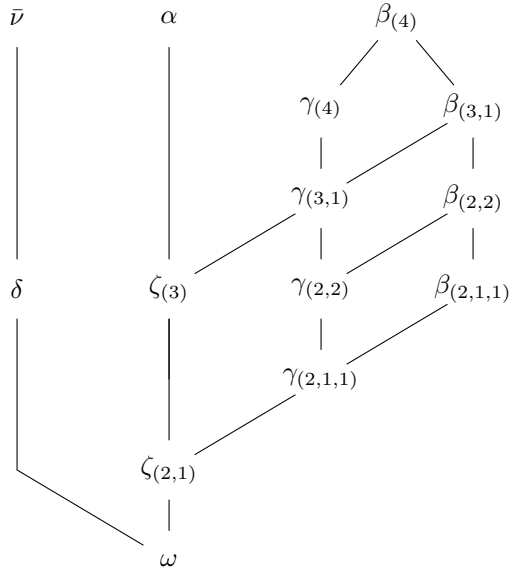
\begin{figure}[ht!]  \!\!\!\! $$ \scalefont{0.9}  \begin{tikzpicture} [scale=0.4]

\draw(-2.5,-3)--(-7.5,-6)--(-7.5,-9)--(-7.5,-15);
\draw(-2.5,-9)--(-7.5,-12);
\draw(-7.5,3)--(-7.5,-9);
\draw(-12.5,3)--(-12.5,-12)--(-7.5,-15);

\draw(0,3)--(2.5,0)--(2.5,-3)--(2.5,-6)--(-2.5,-9);
\draw(2.5,-3)--(-2.5,-6);
\draw(2.5,-0)--(-2.5,-3);
\draw(0,3)--(-2.5,0)--(-2.5,-3)--(-2.5,-6)--(-2.5,-9);

\fill[white](-12.5,3)   circle (28pt) ;
\fill[white](-7.5,3)   circle (28pt) ;
\fill[white](-7.5,-6)   circle (28pt) ;
\fill[white](-7.5,-12)   circle (28pt) ;
\fill[white](-7.5,-15)   circle (28pt) ;

\fill[white](2.5,0)   circle (28pt) ; 
\fill[white](-2.5,0)   circle (28pt) ; 
\fill[white](0,3)   circle (28pt) ; 

\fill[white](2.5,-3)   circle (28pt) ; 
\fill[white](-2.5,-3)   circle (28pt) ; 

\fill[white](2.5,-6)   circle (28pt) ; 
\fill[white](-2.5,-6)   circle (28pt) ; 

\fill[white](-2.5,-9)   circle (28pt) ; 

\path(0,3) node {$\beta_{(4)}$} ;

\path(2.5,0) node {$\beta_{(3,1)}$} ;
\path(-2.5,0) node {$\gamma_{(4)} $} ;

\path(2.5,-3) node {$\beta_{(2,2)}$} ;
\path(-2.5,-3) node {$\gamma_{(3,1)}$} ;

\path(2.5,-6) node {$\beta_{(2,1,1)}$} ;
\path(-2.5,-6) node {$\gamma_{(2,2)}$} ;

\path(-2.5,-9) node {$\gamma_{(2,1,1)}$} ;

\path(0-7.5,-12) node {$\zeta_{(2,1)}$} ;

\path(0-7.5,-6) node {$\zeta_{(3)}$} ;

\path(0-7.5,-15) node {$\omega$} ;

\path(0-7.5,3) node {$\alpha$} ;

\fill[white](-17.5,-6)   circle (28pt) ;
\fill[white](-12.5,-6)   circle (28pt) ;
\path(-12.5,-6) node {$\delta  $} ;

\path(-12.5,3) node {$\bar\nu $} ;

\end{tikzpicture}$$
\caption{Hasse diagram of the dominance ordering on the relevant partitions $\lambda$ such that
$ \lambda\trianglerighteq \omega:=\bar\nu - 2\varepsilon_1+2\varepsilon_2 $.      }
\label{hasse33}
\end{figure}

\smallskip\noindent{\bf The partitions $\bm\bar\nu$, $\bm\alpha$, $\bm\delta $. }
By \cref{yaaaas2} and \cref{pppppppppp}, we know that
$$
\langle s_\nu \circ s_{(2)} \mid s_{\bar\nu} \rangle =
\langle s_\nu \circ s_{(2)} \mid s_{\alpha} \rangle =1  $$
and
$$\langle s_\nu \circ s_{(2)} \mid s_{\delta }
\rangle =0.$$
\smallskip
\noindent{\bf The partitions $\bm\beta_{(4)}$ and $\bm\gamma_{(4)}$. }
There is a single plethystic tableau $\SSTT^{\beta_{(4)}}\in \PStd((2)^{(a^b)}, \beta_{(4)})$ as follows:
$$
\SSTT^{\beta_{(4)}}(b,a)=\SSTT^{\beta_{(4)}}(b,a-1) = \gyoung(2;2)
\qquad
\SSTT^{\beta_{(4)}}\node = \gyoung(1;x)
$$
for $\node$ otherwise.  This weight is maximal in the dominance order and so
$\langle s_\nu \circ s_{(2)} \mid s_{\beta_{(4)}}\rangle =1$.
Similarly,   there is a single plethystic tableau $\SSTT^{{\gamma}_{(4)}}\in \PStd((2)^{(a^b)}, \gamma_{(4)})$ as follows:
$$
\SSTT^{{\gamma}_{(4)}}(b,a)=\SSTT^{{\gamma}_{(4)}}(b,a-1) = \gyoung(2;2)
\qquad
\SSTT^{{\gamma}_{(4)}}(b-1,a) = \gyoung(1;b)
\qquad
\SSTT^{{\gamma}_{(4)}}\node = \gyoung(1;x)
$$
for $\node$ otherwise.    Since ${\beta_{(4)}}\rhd  {\gamma_{(4)}}$ and 
 $|\SStd({\beta_{(4)}}, {\gamma_{(4)}})|=1$,   
it follows that
 $\langle s_\nu \circ s_{(2)} \mid  s_{\gamma_{(4)}}\rangle =1-1=0$.

\smallskip
\noindent{\bf The partitions $\bm\beta_{(3,1)}$ and $\bm\gamma_{(3,1)}$. }
There is a unique plethystic tableau $\SSTT^\beta_{(3,1)}  \in \PStd
((2)^{(a^b)}, \beta_{(3,1)})$ as follows:
$$\begin{array}{rrrrrccc}
 &\SSTT^{{\beta}_{(3,1)}}(b,a-1) = \gyoung(2;2)
 \quad
& \SSTT^{{\beta}_{(3,1)}}(b,a ) = \gyoung(2;3)
 \quad
&\SSTT^{{\beta}_{(3,1)}}\node = \gyoung(1;x)
\end{array}$$
for $\node$ otherwise.  We  find that  
 $|\SStd(\beta_{{(4)}}, \beta_{(3,1)})|=1$ 
  and so
$\langle s_\nu \circ s_{(2)} \mid s_{\beta_{(3,1)}}\rangle =1-1=0$.
There are two plethystic tableaux $\SSTT^{\gamma_{(3,1)}}_1,\SSTT^{\gamma_{(3,1)}}_2 \in \PStd
((2)^{(a^b)}, \gamma_{(3,1)})$ as follows:
$$\begin{array}{rrrrrrrr}
  &\SSTT^{{\gamma_{(3,1)}}}_1(b,a-1) = \gyoung(2;2)
 \quad
 &\SSTT^{{\gamma_{(3,1)}}}_1(b,a) = \gyoung(2;3)
  \quad
 &\SSTT^{{\gamma_{(3,1)}}}_1(b-1,a) = \gyoung(1;b)
\\
 &\SSTT^{{\gamma_{(3,1)}}}_2(b-1,a) = \gyoung(2;2)
 \quad
&\SSTT^{{\gamma_{(3,1)}}}_2(b,a) = \gyoung(2;3)
 \quad
  &\SSTT^{{\gamma_{(3,1)}}}_i\node = \gyoung(1;x)
\end{array}$$
for $i=1,2$ and $\node$ otherwise.
 Since $|\SStd(\beta_{(4)},{\gamma_{(3,1)}})|=1$, it follows that
$\langle s_\nu \circ s_{(2)} \mid s_{\gamma_{(3,1)}}\rangle =2-1=1$.

 \smallskip
\noindent{\bf The partitions $\bm\beta_{(2,2)}$ and  $\bm\gamma_{(2,2)}$. }
  We define
$$
\newcommand{\ibar}{\bar{i}}
\SSTS^{\beta_{(2,2)}} (b,a-1)= \Yboxdim{14pt}\gyoung(2;2)\quad
\SSTS^{\beta_{(2,2)}} (b,a)= \gyoung(3;3)\qquad
 \SSTS^{\beta_{(2,2)}} \node= \gyoung(1;x)
$$
$$\newcommand{\ibar}{\bar{i}}\Yboxdim{14pt}
\SSTT^{\beta_{(2,2)}}  (b,a-1)= \gyoung(2;3)\quad
\SSTT^{\beta_{(2,2)}}  (b,a)= \gyoung(2;3)\quad
 \SSTT^{\beta_{(2,2)}}  \node= \gyoung(1;x)
$$
and similarly, we define
 $$
\newcommand{\ibar}{\bar{i}}
\SSTS^{\gamma_{(2,2)}} (b,a-1)= \Yboxdim{14pt}\gyoung(2;2)\quad
\SSTS^{\gamma_{(2,2)}} (b,a)= \gyoung(3;3)\qquad
\SSTS^{\gamma_{(2,2)}} (b-1,a)= \gyoung(1;b)\qquad
$$
$$\newcommand{\ibar}{\bar{i}}\Yboxdim{14pt}
\SSTT^{\gamma_{(2,2)}}  (b,a-1)= \gyoung(2;3)\quad
\SSTT^{\gamma_{(2,2)}} (b,a)= \gyoung(2;3)\quad
\SSTT^{\gamma_{(2,2)}} (b-1,a)= \gyoung(1;b)\qquad
$$
$$\newcommand{\ibar}{\bar{i}}\Yboxdim{14pt}
\SSTU^{\gamma_{(2,2)}}  (b-1,a)= \gyoung(2;2)\quad
\SSTU^{\gamma_{(2,2)}}  (b,a)= \gyoung(3;3)\quad
$$
and
$$
\SSTS^{\gamma_{(2,2)}} \node= \gyoung(1;x)
\quad
\SSTT^{\gamma_{(2,2)}} \node= \gyoung(1;x)
\quad\SSTU^{\gamma_{(2,2)}} \node= \gyoung(1;x)
$$
for $\node$ otherwise.  We calculate $|\SStd(\beta_{(4)},\beta_{(2,2)})|=1$, and hence
  $$\langle s_\nu \circ s_{(2)} \mid s_{\beta_{(2,2)}}\rangle =1.$$
Similarly,
$$
 |\SStd(\beta_{(4)},\gamma_{(2,2)})|=1, \quad |\SStd(\gamma_{(3,1)}, \gamma_{(2,2)})|=1   \text{ and }
|\SStd(\beta_{(2,2)},\gamma_{(2,2)} )|=1,
 $$
  and so
 $$\langle s_\nu \circ s_{(2)} \mid s_{\gamma_{(2,2)}}\rangle =0.$$

 \smallskip
\noindent{\bf The partitions $\bm\beta_{(2,1,1)}$ and  $\bm\gamma_{(2,1,1)}$. }   We claim that
$$
\langle s_\nu \circ s_{(2)} \mid s_{\beta_{(2,1,1)}}\rangle =0
=
\langle s_\nu \circ s_{(2)} \mid s_{\gamma_{(2,1,1)}}\rangle.
$$ The calculation is similar to that for $ \beta_{(2,2)}$ and  $ \gamma_{(2,2)}$ and so we leave this as an exercise for the reader.

\smallskip
\noindent{\bf The partition $\bm\zeta_{(3)}$. }   Given $2\leq i\leq b$ we let
$$
\SSTT^{\zeta_{(3)} }_i(b,a-1)=\gyoung(2;2)\quad
\SSTT^{\zeta_{(3)}}_i(b,a)=\gyoung(2;i)\quad
\SSTT^{\zeta_{(3)}}_i(j-1,a)=\gyoung(1;j)
\quad
\SSTT^{\zeta_{(3)}}_i\node=\gyoung(1;x)
$$
for $i<j<b$ and $\node$ otherwise.
Given $2 <  i < b$ we let
$$
\SSTS^{\zeta_{(3)}}_i(b-1,a)=\gyoung(2;2)\quad
\SSTS^{\zeta_{(3)}}_i(b,a)=\gyoung(2;i)\quad
\SSTS^{\zeta_{(3)}}_i(j-1,a)=\gyoung(1;j)
\quad
\SSTS^{\zeta_{(3)}}_i\node=\gyoung(1;x)
$$
for $i<j<b$ and $\node$ otherwise.
We compute
 $|\SStd(\alpha, \zeta_{(3)})|=1$,
 $|\SStd(\beta_{(4)}, \zeta_{(3)})|=b-2$,  and finally $|\SStd(\gamma_{(3,1)}, \zeta_{(3)})|=b-4$
provided $b \ne 3$. (When $b=3$ this last  multiplicity is zero.) We therefore obtain that, provided $b \ne 3$,
$$\langle s_\nu \circ s_{(2)} \mid s_{\zeta_{(3)}} \rangle =  (b-3)+(b-1)-(b-4)-(b-2)-1=1,$$
but this multiplicity is zero in the case $b=3$.

\smallskip
 \noindent{\bf The partition $\bm\zeta_{(2,1)}$. }     For $3\leq i\leq b$, we define
$$\begin{array}{rrrrrrrr}
&\SSTS_i^{\zeta_{(2,1)}} (b,a-1)=\gyoung(2;2)
\quad
&\SSTS_i^{\zeta_{(2,1)}} (b,a)=\gyoung(3;i)
\quad
&\SSTS_i^{\zeta_{(2,1)}} (j-1,a)=\gyoung(1;j)
\\
&\SSTT_i^{\zeta_{(2,1)}} (b,a-1)=\gyoung(2;3)
\quad
&\SSTT_i^{\zeta_{(2,1)}} (b,a)=\gyoung(2;i)
\quad
&\SSTT_i^{\zeta_{(2,1)}} (j-1,a)=\gyoung(1;j)
\end{array}$$
and
$
 \SSTS^{\zeta_{(2,1)}}_i \node=\gyoung(1;x),  \SSTT^{\zeta_{(2,1)}}_i \node=\gyoung(1;x)$
for $i<j \leq b$ and $\node$ otherwise.  Now, for $3\leq i \leq b-1$, we define
$$\begin{array}{rrrrrrrr}
&\SSTU_i^{\zeta_{(2,1)}} (b-1,a)=\gyoung(2;2)
\quad
&\SSTU_i^{\zeta_{(2,1)}} (b,a)=\gyoung(3;i)
\quad
&\SSTU_i^{\zeta_{(2,1)}} (j-1,a)=\gyoung(1;j)
\end{array}$$
and $\SSTU^{\zeta_{(2,1)}}_i \node=\gyoung(1;x) $ for $i<j \leq b$ and $\node$ otherwise.
For $4\leq i \leq b-1$, we define
$$\begin{array}{rrrrrrrr}
&\SSTV_i^{\zeta_{(2,1)}} (b-1,a)=\gyoung(2;3)
\quad
&\SSTV_i^{\zeta_{(2,1)}} (b,a)=\gyoung(2;i)
\quad
&\SSTV_i^{\zeta_{(2,1)}} (j-1,a)=\gyoung(1;j)
\end{array}$$
and $\SSTV^{\zeta_{(2,1)}}_i \node=\gyoung(1;x) $ for $i<j \leq b$ and $\node$ otherwise.  We have two final plethystic tableaux of weight $\zeta_{(2,1)}$ to consider, namely
$$\begin{array}{rrrrrrrr}
&
\SSTW_1^{\zeta_{(2,1)}}(i-1,a)=\gyoung(1;i)\quad
&\SSTW_1^{\zeta_{(2,1)}}(b,a-1)=\gyoung(2;2)\quad
&\SSTW_1^{\zeta_{(2,1)}}(b,a)=\gyoung(2;3)
\\
&
\SSTW_2^{\zeta_{(2,1)}}(j-1,a)=\gyoung(1;j)\quad
 &\SSTW_2^{\zeta_{(2,1)}}(b,a)=\gyoung(2;3)\quad
  &\SSTW_2^{\zeta_{(2,1)}}(b-1,a)=\gyoung(2;2)
\end{array}
$$
and $\SSTW_k^{\zeta_{(2,1)}}\node=\gyoung(1;x) $ for $ 2\leq i < b$, $ 2\leq j < b-1$, $k=1,2$ and $\node$ otherwise.
We have that
 $$|\SStd(\beta_{(4)},\zeta_{(2,1)})|=b-2
 \quad
 |\SStd(\gamma_{(3,1)},\zeta_{(2,1)})|= 2(b-4)
$$
$$
 |\SStd(\zeta_{(3)},\zeta_{(2,1)})|=1
 \quad
 |\SStd(  \beta_{(2,2)} ,\zeta_{(2,1)})|=b-3
 \quad
 |\SStd(  \alpha ,\zeta_{(2,1)})|=2
 $$
and putting this altogether we deduce that
$\langle s_\nu \circ s_{(2)} \mid s_			
{\zeta_{(2,1)}}
\rangle =1 $.

\smallskip
\noindent{\bf The partition $\bm\omega$. }
The plethystic tableaux of weight $\omega$ are as follows.  For $2\leq i \leq j\leq b$ we have
$$\begin{array}{rrrrrrrr}
&
\SSTS_{i,j}^\omega(b,a-1)=\gyoung(2;i)\quad
&\SSTS_{i,j}^\omega(b,a)=\gyoung(2;j)\quad
&\SSTS_{i,j}^\omega( k-1,a)=\gyoung(1;k)\quad
\\
&\SSTS_{i,j}^\omega( \ell-1,a-1)=\gyoung(1;\ell)
\quad
&\SSTS_{i,j}^\omega\node=\gyoung(1;x) \\
\end{array}$$
for all $i<k\leq b $ and $j<\ell \leq b $ and $\node$ otherwise.
 For $3\leq i \leq j\leq b$ we have
$$\begin{array}{rrrrrrrr}&
\SSTT_{i,j}^\omega(b,a-1)=\gyoung(2;2)\quad
&\SSTT_{i,j}^\omega(b,a)=\gyoung(i;j)\quad
&\SSTT_{i,j}^\omega( k-1,a)=\gyoung(1;k)\quad
\\
&\SSTT_{i,j}^\omega( \ell-1,a-1)=\gyoung(1;\ell)
\quad
&\SSTT_{i,j}^\omega\node=\gyoung(1;x)
\end{array}$$
for all $i<k\leq b $ and $j<\ell \leq b $ and $\node$ otherwise.
For $2\leq i<j \leq b $ we define
$$\begin{array}{rrrrrrrr}
&
\SSTU_{i,j}^\omega(b-1,a)=\gyoung(2;i)\quad
&\SSTU_{i,j}^\omega(b,a)=\gyoung(2;j)\quad
&\SSTU_{i,j}^\omega( k-1,a)=\gyoung(1;k)\quad
\\
&\SSTU_{i,j}^\omega( \ell-2,a-1)=\gyoung(1;\ell)
\quad
&\SSTU_{i,j}^\omega\node=\gyoung(1;x)
\end{array}$$
 for all $i+1\leq k\leq j-1 $ and $j+1 \leq \ell \leq b $ and $\node$ otherwise.
\color{black}
For $3 \leq i< j \leq b $ we define
$$\begin{array}{rrrrrrrr}
&
\SSTV_{i,j}^\omega(b-1,a)=\gyoung(2;2)\quad
&\SSTV_{i,j}^\omega(b,a)=\gyoung(i;j)\quad
&\SSTV_{i,j}^\omega( k-1,a)=\gyoung(1;k)\quad
\\
&\SSTV_{i,j}^\omega( \ell-2,a-1)=\gyoung(1;\ell)
\quad
&\SSTV_{i,j}^\omega\node=\gyoung(1;x)
\end{array}$$
 for all $i+1\leq k\leq j-1 $ and $j+1 \leq \ell \leq b $ and $\node$ otherwise.
\color{black}
Finally, we define
$$\begin{array}{rrrrrrrr}
&
\SSTW^\omega(b ,a)=\gyoung(2;2)\quad
&\SSTW^\omega(i-1,a)=\gyoung(1;i)\quad
 &\SSTW^\omega\node=\gyoung(1;x)
\end{array}$$
for $2\leq i\leq b$ and $\node$ otherwise.  We have that
 $$|\SStd(\bar\nu,\omega)|=1
 \quad
 |\SStd(\alpha,\omega)|=
2 (b-2)
 \quad
 |\SStd(\beta_{(4)},\omega)|=\textstyle{{b-1} \choose 2}
 $$
$$ |\SStd(\gamma_{(3,1)},\omega)|= (b-2)(b-4)
 \quad  |\SStd(\beta_{(2,2)},\omega)|=\textstyle{{b-2} \choose 2}$$
$$
|\SStd(\zeta_{(3)},\omega)|=b-2 \quad |\SStd(\zeta_{(2,1)},\omega)|= b-3.
 $$
Taking the usual summation as in \cref{dvir}, we obtain the required equality
$\langle s_\nu \circ s_{(2)} \mid s_\omega\rangle =2 $.

In the cases $b=3,4,5$, not all the partitions listed at the start of the proof are defined. Nonetheless the calculation proceeds in exactly the same way and the only difference is
that  $\langle s_{(a^3)} \circ s_{(2)} \mid s_{\zeta_{(3)}}\rangle =0$, but we still find that  $\langle s_{(a^3)} \circ s_{(2)} \mid s_{\omega}\rangle =2$.
\end{proof}

\begin{cor}
Let $\nu=(a^b)$ be a rectangle with $a,b\geq 3$.
Then $p(\nu,\mu)>1$ for any partition $\mu$ such that $|\mu|>1$.
\end{cor}
\begin{proof}
Notice that  our extra restriction  on the width  being at least 3 ensures  that our set $\mathcal N$ of rectangles is conjugation-invariant.
  We have that
    $$
2\le \langle s_{(a^b)} \circ s_{(2)}\mid
 s_{ (ab+a-2,a+2,a^{b-2})}   \rangle
 $$
and so the result holds by \cref{cor:2-important}.
  \end{proof}

   \begin{prop}\label{2rec2}
For $a >3$ we have
$$
\langle s_{(a^2)} \circ s_{(2)} \mid s_{( 3a-2, a, 2)} \rangle =2
=
\langle s_{(2^a)} \circ s_{(2)} \mid s_{(2a, 4, 2^{a-2})} \rangle .
 $$
\end{prop}
 \begin{proof}
The latter equality follows from \cref{the latter} and is only recorded here for convenience.
We note that
$ \langle s_{(a^2)} \circ s_{(2)}  \mid s_{ ( 3a, a)}\rangle =1$ 
 by \cref{pppppppppp}.
\color{black}
  By \cref{dvir}, it is enough to     calculate the plethystic and semistandard tableaux for each of the partitions $\alpha $ such that $(3a,a)\succeq \alpha\trianglerighteq (3a-2,a,2)$
  in order to deduce the result.
 We record the Hasse diagram (under the dominance ordering) for this set of partitions  in \cref{lotstocheck2}.   We claim that
  $$
  \langle s_{(a^2)} \circ s_{(2)}  \mid s_\alpha \rangle
  =
    \begin{cases}
    0			&\text{for }\alpha=( 3a-1, a+1),{(3a,a-1,1), (3a,a-2,2)} \\
        2			&\text{for }\alpha =  (3a-2,a,2) \\
    1			&\text{for all other }(3a,a)\succeq \alpha\vartriangleright  (3a-2,a,2)\\
   \end{cases}
    $$

 \begin{figure}[ht!]  $$ \scalefont{0.9}  \begin{tikzpicture} [scale=0.45]

\draw(11,3)--(11,0)--(11,-3)--(4.5,-6)--(4.5,-3)--(11,-0);
\draw(11,3)--(4.5,0)--(4.5,-3)--(4.5,-6)--(-3,-9);
\draw(4.5,-3)--(-3,-6);
\draw(4.5,0)--(-3,-3)--(-3,-6)--(-3,-9);

\fill[white](4.5,0)   circle (32pt) ;
\fill[white](11,0.2)   circle (32pt) ;
\fill[white](11,-3.2)   circle (32pt) ;

\fill[white](11,3)   circle (32pt) ; 
\fill[white](3,3)   circle (32pt) ; 

\fill[white](4.5,-3)   circle (32pt) ; 
\fill[white](-3,-3)   circle (32pt) ; 
\fill[white](4.5,-3)   circle (32pt) ; 
\fill[white](-3,-3)   circle (32pt) ; 

\fill[white](4.5,-6)   circle (32pt) ; 
\fill[white](-3,-6)   circle (32pt) ; 

\fill[white](4.5,-6)   circle (32pt) ; 
\fill[white](-3,-6)   circle (32pt) ; 

\fill[white](-3,-9)   circle (32pt) ; 

\path(11,3) node {$( 3a, a)$} ;

\path(4.5,0) node {$( 3a-1, a+1)$} ;

\path(12 ,0) node {$( 3a, a-1, 1)$} ;

\path(12 ,-3) node {$( 3a, a-2, 2)$} ;

\path(4.5,-3) node {$( 3a-1, a,1)$} ;
\path(-3,-3) node {$( 3a-2, a+2)$} ;

\path(4.5,-6) node {$( 3a-1, a-1,2)$} ;
\path(-3,-6) node {$( 3a-2, a+1,1)$} ;

\path(-3,-9) node {$( 3a-2, a,2)$} ;

\end{tikzpicture}$$
\caption{Hasse diagram of the partial ordering on the partitions $\alpha$ such that
$(3a,a)\succeq \alpha\trianglerighteq (3a-2,a,2)$.      }
\label{lotstocheck2}
\end{figure}
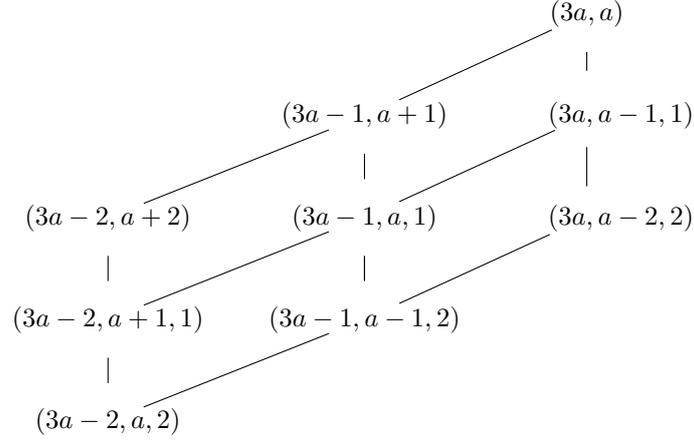

\noindent We have  that
$ \langle s_{(a^2)} \circ s_{(2)}  \mid s_{ ( 3a, a)}\rangle =1$ by \cref{PW}
  and 
 $$
 \langle s_{(a^2)} \circ s_{(2)}  \mid s_{ (3a, a-1,1)}\rangle=
  \langle s_{(a^2)} \circ s_{(2)}  \mid s_{( 3a, a-2,2)}\rangle
  =0
$$
by \cref{maxfirstpart}.
The partitions $(3a-1,a+1)=\bar\nu+\epsilon_1+\epsilon_2$, 
$(3a-1,a,1)=\bar\nu+\epsilon_1+\epsilon_3$ and $(3a-1,a-1,2)=\bar\nu+\epsilon_1-\epsilon_2+2\epsilon_3$ are dealt with by \cref{yaaaas2}, having multiplicities 0,1,1 respectively.
\color{black}

Now, there are two elements of ${\rm PStd}((2)^{(a^2)},   ( 3a-2, a+2))$ given by
\begin{align*}\SSTT_1(1,a)=\gyoung(1;2)
\quad
\SSTT_1(2,a)=\gyoung(2;2)&
\\
 \SSTT_2(2,a-1)=\gyoung(2;2)
\quad
\SSTT_2(2,a)=\gyoung(2;2)&
\end{align*}
and $\SSTT_i(r,c)=\gyoung(1;r)$ otherwise for $i=1,2$.  There is a single  element of ${\rm SStd}(\color{black}(3a,a),( 3a-2, a+2))$ and so
$$
 \langle s_{(a^2)} \circ s_{(2)}  \mid s_{ ( 3a-2, a+2)}\rangle = 2 - 1= 1.
$$
by \cref{dvir}.
The  five  elements of ${\rm PStd}((2)^{(a^2)},   ( 3a-2, a+1,1))$  are given by
$$\begin{array}{rrrrrrrrrrrrrrrr}
\SSTT_1(1,a)=\gyoung(1;2)
&
\SSTT_1(2,a)=\gyoung(2;3)&
\\
\SSTT_2(2,a-1)=\gyoung(2;2)
&
\SSTT_2(2,a)=\gyoung(2;3)&
 \\
\SSTT_3(2,a-2)=\gyoung(1;3)
&
\SSTT_3(2,a-1)=\gyoung(2;2)
&
\SSTT_3(2,a)=\gyoung(2;2)&
\\
 \color{black} \SSTT_4(1,a)=\gyoung(1;3)
 & \color{black}
\SSTT_4(2,a)=\gyoung(2;2)&
\\
\SSTT_5(1,a)=\gyoung(1;2)
&  \quad
\SSTT_5(2,a-1)=\gyoung(1;3)
&\quad
\SSTT_5(2,a)=\gyoung(2;2)&
\\  \end{array}$$
and $\SSTT_i(r,c)=\gyoung(1;r)$ otherwise for $i=1,2,3,4,5$.
 We have that
\begin{align*}
  |{\rm SStd}((3a,a),( 3a-2, a+1,1))| 
 \color{black}
 &= 2
\\
  |{\rm SStd}(( 3a-2, a+1,1),(3a-2,a+2))| &= 1
\\
   |{\rm SStd}((3a-1,a,1),( 3a-2, a+1,1))|
  \color{black}
   &= 1.
\end{align*}
Therefore
 $$
 \langle s_{(a^2)} \circ s_{(2)}  \mid s_{ ( 3a-2, a+1,1)}\rangle = 5-2-1-1=1
$$by \cref{dvir}.
Finally, we are now ready to show that the last constituent     of interest, $( 3a-2, a,2)$, appears with multiplicity   2.
The ten elements of ${\rm PStd}((2)^{(a^2)},   ( 3a-2, a,2))$  are given by
$$\begin{array}{rrrrrrrrrrrrrrrr}
\SSTT_1(1,a)=\gyoung(1;2)
&
\SSTT_1(2,a)=\gyoung(3;3)&
\\
\SSTT_2(2,a-1)=\gyoung(2;2)
&
 \SSTT_2(2,a)=\gyoung(3;3)
\color{black}
&
\\
\SSTT_3(1,a)=\gyoung(1;3)
&
\SSTT_3(2,a)=\gyoung(2;3)&
\\
\SSTT_4(2,a-1)=\gyoung(2;3)
&
\SSTT_4(2,a)=\gyoung(2;3)&
\\
\SSTT_5(1,a)=\gyoung(1;2)
&
\SSTT_5(2,a-1)=\gyoung(1;3)&
&
\SSTT_5(2,a)=\gyoung(2;3)&
\\
\SSTT_6(2,a-2)=\gyoung(1;3)
&
\SSTT_6(2,a-1)=\gyoung(2;2)&
&
\SSTT_6(2,a)=\gyoung(2;3)&
\\
\SSTT_7(1,a)=\gyoung(1;3)
&
\SSTT_7(2,a-1)=\gyoung(1;3)&
&
\SSTT_7(2,a)=\gyoung(2;2)&
\end{array}$$
along with the following
$$\begin{array}{rrrrrrrrrr}
\SSTT_8(1,a-1)=\gyoung(1;2)\qquad
&
\SSTT_8(1,a)=\gyoung(1;2)
\\
\SSTT_8(2,a-1)=\gyoung(1;3)
&
\SSTT_8(2,a)=\gyoung(1;3)
\\
\SSTT_9(1,a)=\gyoung(1;2)
&
\SSTT_9(2,a-2)=\gyoung(1;3)
&  \\
\SSTT_9(2,a-1)=\gyoung(1;3)
&
\SSTT_9(2,a)=\gyoung(2;2)&
\\
\SSTT_{10}(2,a-3)=\gyoung(1;3)
&
\SSTT_{10}(2,a-2)=\gyoung(1;3)
&  \\
\SSTT_{10}(2,a-1)=\gyoung(2;2)
&
\SSTT_{10}(2,a)=\gyoung(2;2)&
\end{array}$$
where $\SSTT_i(r,c)=\gyoung(1;r)$ for all $1\leq i \leq 10$ and $(r,c)$ other than the boxes detailed above.
We have that
$$
 |{\rm SStd}(	\alpha , ( 3a-2, a,2)	)| 
\color{black}
 =
\begin{cases}
3 &\text{if } \alpha=(3a,a)   \\
2&\text{if } \alpha=(3a-1,a,1)   \\
1&\text{if } \alpha=(3a-2,a+2), (3a-2,a+1,1), \text{ or }(3a-1,a-1,2).
\end{cases}
$$
 Therefore
  $
 \langle s_{(a^2)} \circ s_{(2)}  \mid s_{ ( 3a-2, a,2)}\rangle = 10-3- 2-1-1-1=2
 $ by \cref{dvir}, as required.
  \end{proof}

 \begin{prop}\label{usefulrestulwhichstopproblem}
Given $\nu=(2^a,1^b) $ with $a,b>1$, we have that
$$
\langle s_\nu \circ s_{(2)}\mid
s_{(	a+b+1   , a+2,2,		1^{2a+b-5   })} \rangle
 =
\begin{cases}
2 &b=2 \\
3 &b>2
\end{cases}
 $$
  \end{prop}

\begin{proof}
 Denote $(	a+b+1   , a+2,2,		1^{2a+b-5   })$ by $\lambda$. 
\color{black}
 We have that $ s_{(2^a,1^{b})}=e_{(a+b,a)}-e_{(a+b+1,a-1)}$
  by \cite[page 115]{MR3443860}.  Therefore, by  \cref{maxminall} we have that
\begin{align*}
 s_{(2^a,1^{b})}\circ s_{(2)}		&=
 e_{(a+b,a)} \circ s_{(2)} -e_{(a+b+1,a-1)}  \circ s_{(2)}
\\
 &=
 \left(  e_{(a+b)}\circ s_{(2)} \right)
  \boxtimes (e_{(a)}\circ s_{(2)})
-\left( e_{(a+b+1)}  \circ s_{(2)}\right)
  \boxtimes (e_{(a-1)}\circ s_{(2)})
\\
&=
\bigg(  \sum_{\rho\vdash a+b}  s_{ss[\rho]}
\bigg)\boxtimes
\bigg(  \sum_{\pi \vdash a }  s_{ss[\pi ]}
\bigg)
-
\bigg(  \sum_{\rho'\vdash a+b+1}  s_{ss[\rho']}
\bigg)\boxtimes
\bigg(  \sum_{\pi '\vdash a-1 }  s_{ss[\pi ']}
\bigg)\end{align*}
where here the sum is taken over all partitions $\rho, \pi , \rho', \pi ' $ with no repeated parts.  
We now use the Littlewood--Richardson Rule.
\color{black}

 \begin{figure}[ht!]$$
 \scalefont{0.7}
  \begin{minipage}{5.75cm} \begin{tikzpicture}  [xscale=0.5,yscale=0.4]

 \draw[thick](0,0)--++(90:1)--++(0:1)--++(-90:1)--++(180:1);
 \draw(0.5,0.5) node {1};
 \draw[thick](1,0)--++(90:1)--++(0:1)--++(-90:1)--++(180:1);
 \draw(1.5,0.5) node {1};
  \draw[thick](2,0)--++(90:1)--++(0:1)--++(-90:1)--++(180:1);
 \draw(2.5,0.5) node {1};
  \draw[thick](5,0)--++(90:1)--++(0:1)--++(-90:1)--++(180:1);
 \draw(5.5,0.5) node {1};

\draw[thick, densely dotted] (3,0)--++(90:1)--++(0:2)--++(-90:1)--++(180:2);
   \draw[thick](6,0)--++(90:1)--++(0:1)--++(-90:1)--++(180:1);
 \draw(6.5,0.5) node {1};

  \draw[thick] (0,-1)--++(90:1)--++(0:1)--++(-90:1)--++(180:1);
 \draw(0.5,-0.5) node {2};

   \draw[thick,densely dotted,fill=gray!30](7,1)--++(0:4)--++(90:1)--++(180:1)--++(180:11)--++(-90:7)--++(0:1)--
   ++(90:6);

   \draw[thick] (-1,-2-4)--++(90:1)--++(0:1)--++(-90:1)--++(180:1);
 \draw(-0.5,-2-3.5) node {3};

  \draw[thick] (-1,-2-5)--++(90:1)--++(0:1)--++(-90:1)--++(180:1);
 \draw(-0.5,-2-4.5) node {4};
    \draw[thick] (-1,-2-6)--++(90:1)--++(0:1)--++(-90:1)--++(180:1);
 \draw(-0.5,-2-5.5) node {5};

  \draw[thick,densely dotted] (-1,-2-6)--++(90:-2);
    \draw[thick,densely dotted] (0,-2-6)--++(90:-2);

   \draw[thick] (-1,-2-9)--++(90:1)--++(0:1)--++(-90:1)--++(180:1);
 \draw(-0.5,-2-8.5) node {$a$};

 \end{tikzpicture}
\end{minipage}
 \qquad
   \begin{minipage}{5.75cm} \begin{tikzpicture}  [xscale=0.5,yscale=0.4]

 \draw[thick](0,0)--++(90:1)--++(0:1)--++(-90:1)--++(180:1);
 \draw(0.5,0.5) node {1};
 \draw[thick](1,0)--++(90:1)--++(0:1)--++(-90:1)--++(180:1);
 \draw(1.5,0.5) node {1};
  \draw[thick](2,0)--++(90:1)--++(0:1)--++(-90:1)--++(180:1);
 \draw(2.5,0.5) node {1};
  \draw[thick](5,0)--++(90:1)--++(0:1)--++(-90:1)--++(180:1);
 \draw(5.5,0.5) node {1};

\draw[thick, densely dotted] (3,0)--++(90:1)--++(0:2)--++(-90:1)--++(180:2);
   \draw[thick](6,0)--++(90:1)--++(0:1)--++(-90:1)--++(180:1);
 \draw(6.5,0.5) node {1};

  \draw[thick] (0,-1)--++(90:1)--++(0:1)--++(-90:1)--++(180:1);
 \draw(0.5,-0.5) node {2};

   \draw[thick,densely dotted,fill=gray!30](7,1)--++(0:3)--++(90:1)--++(180:1)--++(180:10)--++(-90:7)--++(0:1)--
   ++(90:6);

    \fill[ fill=gray!30](-0.1,1)--++(0:2.1)--++(-90:1)--++(180:2.1)--++(90:1);

    \draw[thick](0,0)--++(0:2)--++(90:1);

  \draw[thick] (10,1)--++(90:1)--++(0:1)--++(-90:1)--++(180:1);
   \draw(10.5,1.5) node {1};

 \draw[thick,fill=white] (-1,-2-3)--++(90:1)--++(0:1)--++(-90:1)--++(180:1);
 \draw(-0.5,-2-2.5) node {1};

   \draw[thick] (-1,-2-4)--++(90:1)--++(0:1)--++(-90:1)--++(180:1);
 \draw(-0.5,-2-3.5) node {3};

  \draw[thick] (-1,-2-5)--++(90:1)--++(0:1)--++(-90:1)--++(180:1);
 \draw(-0.5,-2-4.5) node {4};
    \draw[thick] (-1,-2-6)--++(90:1)--++(0:1)--++(-90:1)--++(180:1);
 \draw(-0.5,-2-5.5) node {5};

  \draw[thick,densely dotted] (-1,-2-6)--++(90:-2);
    \draw[thick,densely dotted] (0,-2-6)--++(90:-2);

   \draw[thick] (-1,-2-9)--++(90:1)--++(0:1)--++(-90:1)--++(180:1);
 \draw(-0.5,-2-8.5) node {$a$};

 \end{tikzpicture}
\end{minipage} $$
\caption{
Let $a,b\geq 2$.  The tableau on the left is the unique 
Littlewood--Richardson
\color{black}
tableau of shape $
 \lambda\setminus ss[(a+b)]$ and weight $ss[(a)]$.
 The tableau on the right is  the first of three  of shape $
 \lambda\setminus ss[(a+b-1,1)]$ and weight $ss[(a)]$.
}
\label{gg1}
\end{figure}
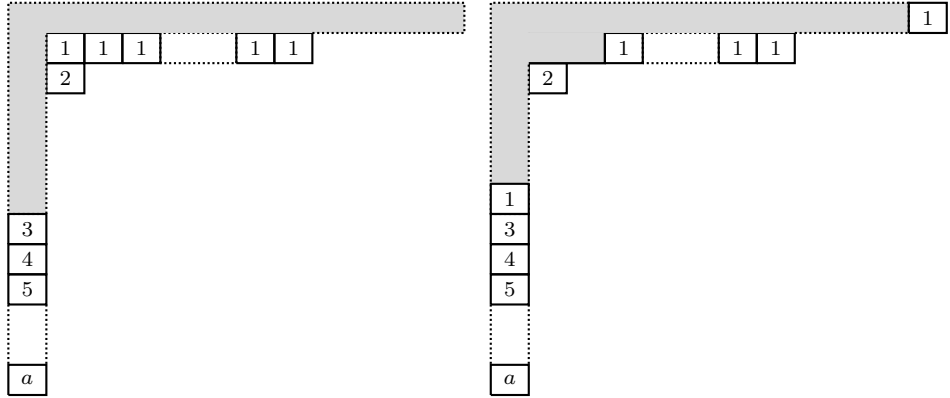

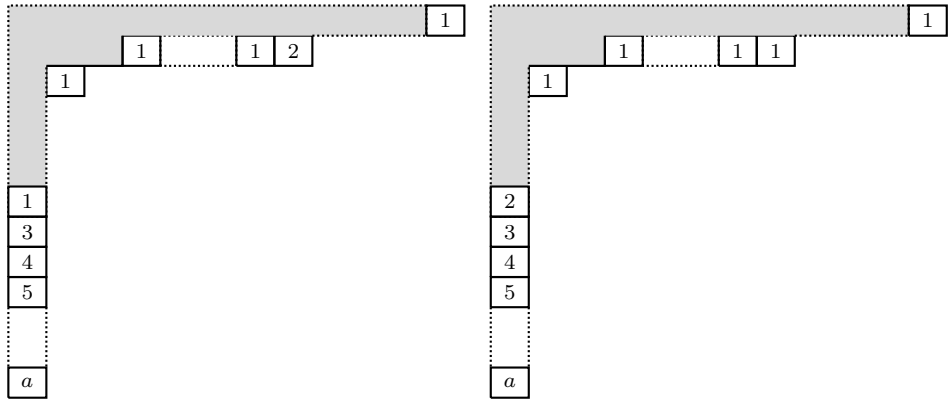
\begin{figure}[ht!]$$
 \scalefont{0.7}
  \begin{minipage}{5.75cm} \begin{tikzpicture}  [xscale=0.5,yscale=0.4]

 \draw[thick](0,0)--++(90:1)--++(0:1)--++(-90:1)--++(180:1);
 \draw(0.5,0.5) node {1};
 \draw[thick](1,0)--++(90:1)--++(0:1)--++(-90:1)--++(180:1);
 \draw(1.5,0.5) node {1};
  \draw[thick](2,0)--++(90:1)--++(0:1)--++(-90:1)--++(180:1);
 \draw(2.5,0.5) node {1};
  \draw[thick](5,0)--++(90:1)--++(0:1)--++(-90:1)--++(180:1);
 \draw(5.5,0.5) node {1};

\draw[thick, densely dotted] (3,0)--++(90:1)--++(0:2)--++(-90:1)--++(180:2);
   \draw[thick](6,0)--++(90:1)--++(0:1)--++(-90:1)--++(180:1);
 \draw(6.5,0.5) node {2};

  \draw[thick] (0,-1)--++(90:1)--++(0:1)--++(-90:1)--++(180:1);
 \draw(0.5,-0.5) node {1};

   \draw[thick,densely dotted,fill=gray!30](7,1)--++(0:3)--++(90:1)--++(180:1)--++(180:10)--++(-90:7)--++(0:1)--
   ++(90:6);

    \fill[ fill=gray!30](-0.1,1.1)--++(0:2.1)--++(-90:1.1)--++(180:2.1)--++(90:1);

    \draw[thick](0,0)--++(0:2)--++(90:1);

  \draw[thick] (10,1)--++(90:1)--++(0:1)--++(-90:1)--++(180:1);
   \draw(10.5,1.5) node {1};

 \draw[thick,fill=white] (-1,-2-3)--++(90:1)--++(0:1)--++(-90:1)--++(180:1);
 \draw(-0.5,-2-2.5) node {1};

   \draw[thick] (-1,-2-4)--++(90:1)--++(0:1)--++(-90:1)--++(180:1);
 \draw(-0.5,-2-3.5) node {3};

  \draw[thick] (-1,-2-5)--++(90:1)--++(0:1)--++(-90:1)--++(180:1);
 \draw(-0.5,-2-4.5) node {4};
    \draw[thick] (-1,-2-6)--++(90:1)--++(0:1)--++(-90:1)--++(180:1);
 \draw(-0.5,-2-5.5) node {5};

  \draw[thick,densely dotted] (-1,-2-6)--++(90:-2);
    \draw[thick,densely dotted] (0,-2-6)--++(90:-2);

   \draw[thick] (-1,-2-9)--++(90:1)--++(0:1)--++(-90:1)--++(180:1);
 \draw(-0.5,-2-8.5) node {$a$};

 \end{tikzpicture}
\end{minipage}
\qquad
 \begin{minipage}{5.75cm} \begin{tikzpicture}  [xscale=0.5,yscale=0.4]

 \draw[thick](0,0)--++(90:1)--++(0:1)--++(-90:1)--++(180:1);
 \draw(0.5,0.5) node {1};
 \draw[thick](1,0)--++(90:1)--++(0:1)--++(-90:1)--++(180:1);
 \draw(1.5,0.5) node {1};
  \draw[thick](2,0)--++(90:1)--++(0:1)--++(-90:1)--++(180:1);
 \draw(2.5,0.5) node {1};
  \draw[thick](5,0)--++(90:1)--++(0:1)--++(-90:1)--++(180:1);
 \draw(5.5,0.5) node {1};

\draw[thick, densely dotted] (3,0)--++(90:1)--++(0:2)--++(-90:1)--++(180:2);
   \draw[thick](6,0)--++(90:1)--++(0:1)--++(-90:1)--++(180:1);
 \draw(6.5,0.5) node {1};

  \draw[thick] (0,-1)--++(90:1)--++(0:1)--++(-90:1)--++(180:1);
 \draw(0.5,-0.5) node {1};

   \draw[thick,densely dotted,fill=gray!30](7,1)--++(0:3)--++(90:1)--++(180:1)--++(180:10)--++(-90:7)--++(0:1)--
   ++(90:6);

    \fill[ fill=gray!30](-0.1,1)--++(0:2.1)--++(-90:1)--++(180:2.1)--++(90:1);

    \draw[thick](0,0)--++(0:2)--++(90:1);

  \draw[thick] (10,1)--++(90:1)--++(0:1)--++(-90:1)--++(180:1);
   \draw(10.5,1.5) node {1};

 \draw[thick,fill=white] (-1,-2-3)--++(90:1)--++(0:1)--++(-90:1)--++(180:1);
 \draw(-0.5,-2-2.5) node {2};

   \draw[thick] (-1,-2-4)--++(90:1)--++(0:1)--++(-90:1)--++(180:1);
 \draw(-0.5,-2-3.5) node {3};

  \draw[thick] (-1,-2-5)--++(90:1)--++(0:1)--++(-90:1)--++(180:1);
 \draw(-0.5,-2-4.5) node {4};
    \draw[thick] (-1,-2-6)--++(90:1)--++(0:1)--++(-90:1)--++(180:1);
 \draw(-0.5,-2-5.5) node {5};

  \draw[thick,densely dotted] (-1,-2-6)--++(90:-2);
    \draw[thick,densely dotted] (0,-2-6)--++(90:-2);

   \draw[thick] (-1,-2-9)--++(90:1)--++(0:1)--++(-90:1)--++(180:1);
 \draw(-0.5,-2-8.5) node {$a$};

 \end{tikzpicture}
\end{minipage}
 $$
\caption{
Two of the three  
Littlewood--Richardson 
\color{black}
tableaux  of shape $\lambda\setminus ss[(a+b-1,1)]$ and weight equal to $ss[(a)]$ for $a,b\geq 2$.  	  (\cref{gg1} contains the final tableau.)	 }
\label{gg2}
\end{figure}

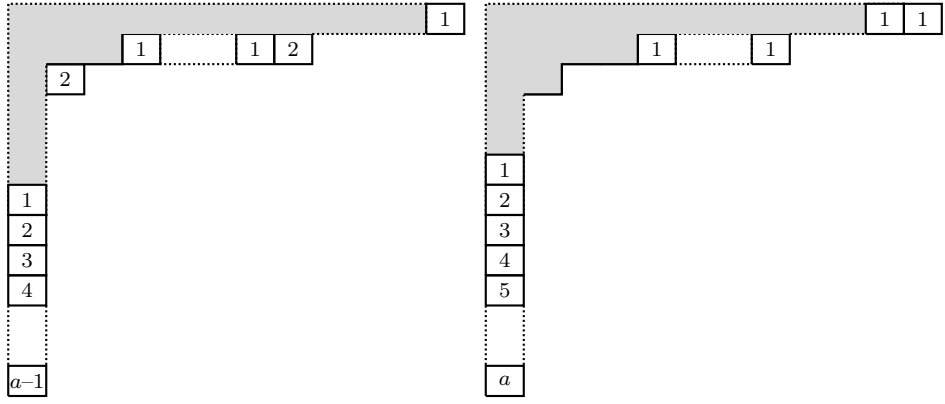
\begin{figure}[ht!]$$
 \scalefont{0.7}
  \begin{minipage}{5.75cm} \begin{tikzpicture}  [xscale=0.5,yscale=0.4]

 \draw[thick](0,0)--++(90:1)--++(0:1)--++(-90:1)--++(180:1);
 \draw(0.5,0.5) node {1};
 \draw[thick](1,0)--++(90:1)--++(0:1)--++(-90:1)--++(180:1);
 \draw(1.5,0.5) node {1};
  \draw[thick](2,0)--++(90:1)--++(0:1)--++(-90:1)--++(180:1);
 \draw(2.5,0.5) node {1};
  \draw[thick](5,0)--++(90:1)--++(0:1)--++(-90:1)--++(180:1);
 \draw(5.5,0.5) node {1};

\draw[thick, densely dotted] (3,0)--++(90:1)--++(0:2)--++(-90:1)--++(180:2);
   \draw[thick](6,0)--++(90:1)--++(0:1)--++(-90:1)--++(180:1);
 \draw(6.5,0.5) node {2};

  \draw[thick] (0,-1)--++(90:1)--++(0:1)--++(-90:1)--++(180:1);
 \draw(0.5,-0.5) node {2};

   \draw[thick,densely dotted,fill=gray!30](7,1)--++(0:3)--++(90:1)--++(180:1)--++(180:10)--++(-90:7)--++(0:1)--
   ++(90:6);

    \fill[ fill=gray!30](-0.1,1.1)--++(0:2.1)--++(-90:1.1)--++(180:2.1)--++(90:1);

    \draw[thick](0,0)--++(0:2)--++(90:1);

  \draw[thick] (10,1)--++(90:1)--++(0:1)--++(-90:1)--++(180:1);
   \draw(10.5,1.5) node {1};

 \draw[thick,fill=white] (-1,-2-3)--++(90:1)--++(0:1)--++(-90:1)--++(180:1);
 \draw(-0.5,-2-2.5) node {1};

   \draw[thick] (-1,-2-4)--++(90:1)--++(0:1)--++(-90:1)--++(180:1);
 \draw(-0.5,-2-3.5) node {2};

  \draw[thick] (-1,-2-5)--++(90:1)--++(0:1)--++(-90:1)--++(180:1);
 \draw(-0.5,-2-4.5) node {3};
    \draw[thick] (-1,-2-6)--++(90:1)--++(0:1)--++(-90:1)--++(180:1);
 \draw(-0.5,-2-5.5) node {4};

  \draw[thick,densely dotted] (-1,-2-6)--++(90:-2);
    \draw[thick,densely dotted] (0,-2-6)--++(90:-2);

   \draw[thick] (-1,-2-9)--++(90:1)--++(0:1)--++(-90:1)--++(180:1);
 \draw(-0.5,-2-8.5) node {$a\text{--}1$};

 \end{tikzpicture}
\end{minipage} \qquad
 \begin{minipage}{5.75cm} \begin{tikzpicture}  [xscale=0.5,yscale=0.4]

 \draw[thick](0,0)--++(90:1)--++(0:1)--++(-90:1)--++(180:1);
 \draw(0.5,0.5) node {1};
 \draw[thick](1,0)--++(90:1)--++(0:1)--++(-90:1)--++(180:1);
 \draw(1.5,0.5) node {1};
  \draw[thick](2,0)--++(90:1)--++(0:1)--++(-90:1)--++(180:1);
 \draw(2.5,0.5) node {1};
  \draw[thick](3,0)--++(90:1)--++(0:1)--++(-90:1)--++(180:1);
 \draw(3.5,0.5) node {1};

\draw[thick, densely dotted] (4,0)--++(90:1)--++(0:2)--++(-90:1)--++(180:2);
   \draw[thick](6,0)--++(90:1)--++(0:1)--++(-90:1)--++(180:1);
 \draw(6.5,0.5) node {1};

  \draw[thick] (0,-1)--++(90:1)--++(0:1)--++(-90:1)--++(180:1);
 \draw(0.5,-0.5) node {2};

   \draw[thick,densely dotted,fill=gray!30](7,1)--++(0:3)--++(90:1)--++(180:1)--++(180:10)--++(-90:7)--++(0:1)--
   ++(90:6);

    \fill[ fill=gray!30](-0.1,1.1)--++(0:3.1)--++(-90:1.1)--++(180:3.1)--++(90:1);
    \fill[ fill=gray!30](-0.1,-1)--++(00:1.1)--++(90:1)--++(180:1.1);

    \draw[thick](0,-1)--++(0:1)--++(90:1)--++(0:2)--++(90:1);

  \draw[thick] (10,1)--++(90:1)--++(0:1)--++(-90:1)--++(180:1);
   \draw(10.5,1.5) node {1};
  \draw[thick,fill=white] (9,1)--++(90:1)--++(0:1)--++(-90:1)--++(180:1);
   \draw(9.5,1.5) node {1};

 \draw[thick,fill=white] (-1,-2-2)--++(90:1)--++(0:1)--++(-90:1)--++(180:1);
 \draw(-0.5,-2-1.5) node {1};

 \draw[thick,fill=white] (-1,-2-3)--++(90:1)--++(0:1)--++(-90:1)--++(180:1);
 \draw(-0.5,-2-2.5) node {2};

   \draw[thick] (-1,-2-4)--++(90:1)--++(0:1)--++(-90:1)--++(180:1);
 \draw(-0.5,-2-3.5) node {3};

  \draw[thick] (-1,-2-5)--++(90:1)--++(0:1)--++(-90:1)--++(180:1);
 \draw(-0.5,-2-4.5) node {4};
    \draw[thick] (-1,-2-6)--++(90:1)--++(0:1)--++(-90:1)--++(180:1);
 \draw(-0.5,-2-5.5) node {5};

  \draw[thick,densely dotted] (-1,-2-6)--++(90:-2);
    \draw[thick,densely dotted] (0,-2-6)--++(90:-2);

   \draw[thick] (-1,-2-9)--++(90:1)--++(0:1)--++(-90:1)--++(180:1);
 \draw(-0.5,-2-8.5) node {$a $};

 \end{tikzpicture}
\end{minipage}
 $$
\caption{
The  tableau on the left is the unique 
  Littlewood--Richardson
  tableau   of shape $\lambda\setminus ss[(a+b-1,1)]$ and weight equal to $ss[(a-1,1)]$ for $a\neq 2$.    	
The tableau on the right is the unique 
  Littlewood--Richardson
 tableau   of shape $\lambda\setminus ss[(a+b-2,2)]$ and weight equal to $ss[(a )]$ for   $b\geq 3$.	
	 }
\label{gg3}
\end{figure}

To compute the multiplicity  $\langle s_\nu \circ s_{(2)}\mid
s_{(	a+b+1   , a+2,2,		1^{2a+b-5   })} \rangle $ it is enough to consider, 
in the sums above,
\color{black} 
 $\rho, \pi , \rho', \pi ' $  with at most 2 rows and with second part at most 2.
We have  that
$$
\langle e_{(a+b,a)}\circ s_{(2)}\mid s_{(a+b+1,a+2,2,1^{2a+b-5})}\rangle =
\begin{cases}
4	&a=2, b=2  \\
5 &b=2, a> 2 \text{ or } a=2, b>2 \\
6 &a,b>2
\end{cases}
$$
 The complete list of 
 Littlewood--Richardson 
  tableaux are listed in \cref{gg1,gg2,gg3} (we depict the generic case  and list the tableaux which disappear for small values of $a$ and $b$);   
in all other relevant cases the Littlewood--Richardson coefficient is zero.

Similarly we have  that
$$
\langle e_{(a+b+1,a-1)}\circ s_{(2)}\mid s_{(a+b+1,a+2,2,1^{2a+b-5})}\rangle =
\begin{cases}
 2	&a=2   \\
3 &a>2
\end{cases} .
$$
The complete list of  
Littlewood--Richardson 
 \color{black} tableaux are listed in \cref{gg4,gg5} (we depict the generic case, one can easily delete the tableaux which disappear for small values of $a$ and $b$).
The result follows.

 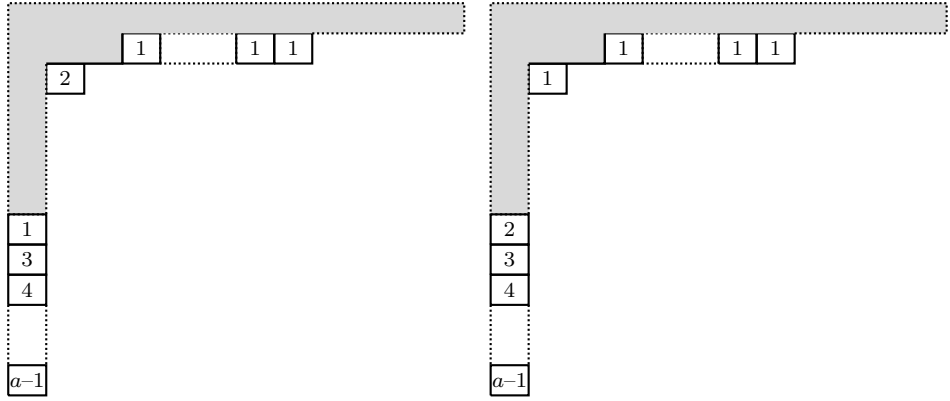
\begin{figure}[ht!]
 $$ \scalefont{0.7}\begin{minipage}{5.75cm} \begin{tikzpicture}  [xscale=0.5,yscale=0.4]

 \draw[thick](0,0)--++(90:1)--++(0:1)--++(-90:1)--++(180:1);
 \draw(0.5,0.5) node {1};
 \draw[thick](1,0)--++(90:1)--++(0:1)--++(-90:1)--++(180:1);
 \draw(1.5,0.5) node {1};
  \draw[thick](2,0)--++(90:1)--++(0:1)--++(-90:1)--++(180:1);
 \draw(2.5,0.5) node {1};
  \draw[thick](5,0)--++(90:1)--++(0:1)--++(-90:1)--++(180:1);
 \draw(5.5,0.5) node {1};

\draw[thick, densely dotted] (3,0)--++(90:1)--++(0:2)--++(-90:1)--++(180:2);
   \draw[thick](6,0)--++(90:1)--++(0:1)--++(-90:1)--++(180:1);
 \draw(6.5,0.5) node {1};

  \draw[thick] (0,-1)--++(90:1)--++(0:1)--++(-90:1)--++(180:1);
 \draw(0.5,-0.5) node {2};

   \draw[thick,densely dotted,fill=gray!30](7,1)--++(0:4)--++(90:1)--++(180:1)--++(180:11)--++(-90:7)--++(0:1)--
   ++(90:6);

    \fill[ fill=gray!30](-0.1,1)--++(0:2.1)--++(-90:1)--++(180:2.1)--++(90:1);

    \draw[thick](0,0)--++(0:2)--++(90:1);

   \draw[thick] (-1,-2-4)--++(90:1)--++(0:1)--++(-90:1)--++(180:1);
 \draw(-0.5,-2-3.5) node {1};

  \draw[thick] (-1,-2-5)--++(90:1)--++(0:1)--++(-90:1)--++(180:1);
 \draw(-0.5,-2-4.5) node {3};
    \draw[thick] (-1,-2-6)--++(90:1)--++(0:1)--++(-90:1)--++(180:1);
 \draw(-0.5,-2-5.5) node {4};

  \draw[thick,densely dotted] (-1,-2-6)--++(90:-2);
    \draw[thick,densely dotted] (0,-2-6)--++(90:-2);

   \draw[thick] (-1,-2-9)--++(90:1)--++(0:1)--++(-90:1)--++(180:1);
 \draw(-0.5,-2-8.5) node {$a\text{--}1$};

 \end{tikzpicture}
\end{minipage} \qquad
\begin{minipage}{5.75cm} \begin{tikzpicture}  [xscale=0.5,yscale=0.4]

 \draw[thick](0,0)--++(90:1)--++(0:1)--++(-90:1)--++(180:1);
 \draw(0.5,0.5) node {1};
 \draw[thick](1,0)--++(90:1)--++(0:1)--++(-90:1)--++(180:1);
 \draw(1.5,0.5) node {1};
  \draw[thick](2,0)--++(90:1)--++(0:1)--++(-90:1)--++(180:1);
 \draw(2.5,0.5) node {1};
  \draw[thick](5,0)--++(90:1)--++(0:1)--++(-90:1)--++(180:1);
 \draw(5.5,0.5) node {1};

\draw[thick, densely dotted] (3,0)--++(90:1)--++(0:2)--++(-90:1)--++(180:2);
   \draw[thick](6,0)--++(90:1)--++(0:1)--++(-90:1)--++(180:1);
 \draw(6.5,0.5) node {1};

  \draw[thick] (0,-1)--++(90:1)--++(0:1)--++(-90:1)--++(180:1);
 \draw(0.5,-0.5) node {1};

   \draw[thick,densely dotted,fill=gray!30](7,1)--++(0:4)--++(90:1)--++(180:1)--++(180:11)--++(-90:7)--++(0:1)--
   ++(90:6);

    \fill[ fill=gray!30](-0.1,1)--++(0:2.1)--++(-90:1)--++(180:2.1)--++(90:1);

    \draw[thick](0,0)--++(0:2)--++(90:1);

   \draw[thick] (-1,-2-4)--++(90:1)--++(0:1)--++(-90:1)--++(180:1);
 \draw(-0.5,-2-3.5) node {2};

  \draw[thick] (-1,-2-5)--++(90:1)--++(0:1)--++(-90:1)--++(180:1);
 \draw(-0.5,-2-4.5) node {3};
    \draw[thick] (-1,-2-6)--++(90:1)--++(0:1)--++(-90:1)--++(180:1);
 \draw(-0.5,-2-5.5) node {4};

  \draw[thick,densely dotted] (-1,-2-6)--++(90:-2);
    \draw[thick,densely dotted] (0,-2-6)--++(90:-2);

   \draw[thick] (-1,-2-9)--++(90:1)--++(0:1)--++(-90:1)--++(180:1);
 \draw(-0.5,-2-8.5) node {$a\text{--}1$};

 \end{tikzpicture}
\end{minipage} $$
\caption{The two  
Littlewood--Richardson 
 \color{black}
 tableaux of shape $\la\setminus ss[(a+b,1)]$ and weight $ss[(a-1)]$.  If $a=2$ only the tableau
on the right exists.  }
\label{gg4}
\end{figure}

 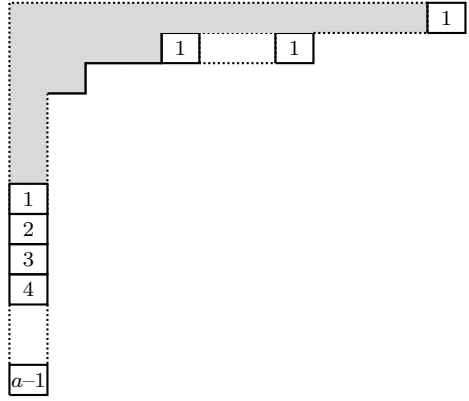
\begin{figure}[ht!]
 $$ \scalefont{0.7} \begin{minipage}{5.75cm} \begin{tikzpicture}  [xscale=0.5,yscale=0.4]

 \draw[thick](0,0)--++(90:1)--++(0:1)--++(-90:1)--++(180:1);
 \draw(0.5,0.5) node {1};
 \draw[thick](1,0)--++(90:1)--++(0:1)--++(-90:1)--++(180:1);
 \draw(1.5,0.5) node {1};
  \draw[thick](2,0)--++(90:1)--++(0:1)--++(-90:1)--++(180:1);
 \draw(2.5,0.5) node {1};
  \draw[thick](3,0)--++(90:1)--++(0:1)--++(-90:1)--++(180:1);
 \draw(3.5,0.5) node {1};

\draw[thick, densely dotted] (4,0)--++(90:1)--++(0:2)--++(-90:1)--++(180:2);
   \draw[thick](6,0)--++(90:1)--++(0:1)--++(-90:1)--++(180:1);
 \draw(6.5,0.5) node {1};

  \draw[thick] (0,-1)--++(90:1)--++(0:1)--++(-90:1)--++(180:1);
 \draw(0.5,-0.5) node {2};

   \draw[thick,densely dotted,fill=gray!30](7,1)--++(0:3)--++(90:1)--++(180:1)--++(180:10)--++(-90:7)--++(0:1)--
   ++(90:6);

    \fill[ fill=gray!30](-0.1,1.1)--++(0:3.1)--++(-90:1.1)--++(180:3.1)--++(90:1);
    \fill[ fill=gray!30](-0.1,-1)--++(00:1.1)--++(90:1)--++(180:1.1);

    \draw[thick](0,-1)--++(0:1)--++(90:1)--++(0:2)--++(90:1);

  \draw[thick] (10,1)--++(90:1)--++(0:1)--++(-90:1)--++(180:1);
   \draw(10.5,1.5) node {1};

 \draw[thick,fill=white] (-1,-2-3)--++(90:1)--++(0:1)--++(-90:1)--++(180:1);
 \draw(-0.5,-2-2.5) node {1};

   \draw[thick] (-1,-2-4)--++(90:1)--++(0:1)--++(-90:1)--++(180:1);
 \draw(-0.5,-2-3.5) node {2};

  \draw[thick] (-1,-2-5)--++(90:1)--++(0:1)--++(-90:1)--++(180:1);
 \draw(-0.5,-2-4.5) node {3};
    \draw[thick] (-1,-2-6)--++(90:1)--++(0:1)--++(-90:1)--++(180:1);
 \draw(-0.5,-2-5.5) node {4};

  \draw[thick,densely dotted] (-1,-2-6)--++(90:-2);
    \draw[thick,densely dotted] (0,-2-6)--++(90:-2);

   \draw[thick] (-1,-2-9)--++(90:1)--++(0:1)--++(-90:1)--++(180:1);
 \draw(-0.5,-2-8.5) node {$a\text{--}1$};

 \end{tikzpicture}
\end{minipage}
$$
\caption{The unique  
Littlewood--Richardson 
 \color{black}
 tableau of shape $\la\setminus ss[(a+b-1,2)]$ and weight $ss[(a-1)]$ for any $a\geq 2$.  }
\label{gg5}
\qedhere
\end{figure}
\end{proof}

 \begin{prop}\label{2liners}
If $\nu \vdash n $ is a
2-line partition
and the pair $(\nu,\mu)$ does not belong to the list of exceptions in \cref{conj}, then $p(\nu,\mu)>1$.
\end{prop}
\begin{proof}
     If $\nu=(b,a)\vdash n> 8$ then, using \cref{Brion},
  we can grow multiplicities for the products
   $s_{(b,a)} \circ s_{(2)}$    from the seeds  $(5,1)$, $(4,2)$, $(4,3)$ for $  a=1, 2, 3$ or the seed $(a^2)$ if $a>3$.
  By direct calculation,  we have that
    $$
  p(\nu,(2))
  =
  \begin{cases}
  2= p({	(5,1), (2),	(6, 4, 2)}) &\text{for } \nu=(5,1)		 \\
  3= p({	(4,2), (2),	(6, 4, 2)}) &\text{for } \nu=(4,2)  	 \\
  3= p({	(4,3), (2),	(8, 4, 2)}) &\text{for } \nu=(4,3)
  \end{cases}
    $$
    and for the final seed
      $\langle s_{(a^2)} \circ s_{(2)}  \mid s_{(3a-2,a,2)} \rangle  =
        2		$
     by  \cref{2rec2}.
  Hence   $p(\nu,(2))>1$     for any  $\nu$ a 2-row partition of $n>8$.

 \begin{figure}[ht!]$$
 \scalefont{0.7}
  \begin{minipage}{5.75cm} \begin{tikzpicture}  [xscale=0.5,yscale=0.4]

 \draw[thick](0,0)--++(90:1)--++(0:1)--++(-90:1)--++(180:1);
 \draw(0.5,0.5) node {1};
 \draw[thick](1,0)--++(90:1)--++(0:1)--++(-90:1)--++(180:1);
 \draw(1.5,0.5) node {1};
  \draw[thick](2,0)--++(90:1)--++(0:1)--++(-90:1)--++(180:1);
 \draw(2.5,0.5) node {1};
  \draw[thick](5,0)--++(90:1)--++(0:1)--++(-90:1)--++(180:1);
 \draw(5.5,0.5) node {1};

\draw[thick, densely dotted] (3,0)--++(90:1)--++(0:2)--++(-90:1)--++(180:2);
   \draw[thick](6,0)--++(90:1)--++(0:1)--++(-90:1)--++(180:1);
 \draw(6.5,0.5) node {1};

  \draw[thick] (0,-1)--++(90:1)--++(0:1)--++(-90:1)--++(180:1);
 \draw(0.5,-0.5) node {2};
  \draw[thick] (1,-1)--++(90:1)--++(0:1)--++(-90:1)--++(180:1);
 \draw(1.5,-0.5) node {2};

   \draw[thick,densely dotted,fill=gray!30](7,1)--++(0:1)--++(90:1)--++(180:1)--++(180:8)--++(-90:7)--++(0:1)--
   ++(90:6);

   \draw[thick] (-1,-2-4)--++(90:1)--++(0:1)--++(-90:1)--++(180:1);
 \draw(-0.5,-2-3.5) node {2};

  \draw[thick] (-1,-2-5)--++(90:1)--++(0:1)--++(-90:1)--++(180:1);
 \draw(-0.5,-2-4.5) node {3};
    \draw[thick] (-1,-2-6)--++(90:1)--++(0:1)--++(-90:1)--++(180:1);
 \draw(-0.5,-2-5.5) node {4};

  \draw[thick,densely dotted] (-1,-2-6)--++(90:-2);
    \draw[thick,densely dotted] (0,-2-6)--++(90:-2);

   \draw[thick] (-1,-2-9)--++(90:1)--++(0:1)--++(-90:1)--++(180:1);
 \draw(-0.5,-2-8.5) node {$a\text{--}1$};

 \end{tikzpicture}
\end{minipage}
\qquad\qquad
\begin{minipage}{5.75cm} \begin{tikzpicture}  [xscale=0.5,yscale=0.4]

 \draw[thick](0,0)--++(90:1)--++(0:1)--++(-90:1)--++(180:1);
 \draw(0.5,0.5) node {1};
 \draw[thick](1,0)--++(90:1)--++(0:1)--++(-90:1)--++(180:1);
 \draw(1.5,0.5) node {1};
  \draw[thick](2,0)--++(90:1)--++(0:1)--++(-90:1)--++(180:1);
 \draw(2.5,0.5) node {1};
  \draw[thick](5,0)--++(90:1)--++(0:1)--++(-90:1)--++(180:1);
 \draw(5.5,0.5) node {1};

\draw[thick, densely dotted] (3,0)--++(90:1)--++(0:2)--++(-90:1)--++(180:2);
   \draw[thick](6,0)--++(90:1)--++(0:1)--++(-90:1)--++(180:1);
 \draw(6.5,0.5) node {1};

  \draw[thick] (0,-1)--++(90:1)--++(0:1)--++(-90:1)--++(180:1);
 \draw(0.5,-0.5) node {1};

   \draw[thick,densely dotted,fill=gray!30](7,1)--++(0:0)--++(90:1)--++(180:1)--++(180:7)--++(-90:7)--++(0:1)--
   ++(90:6);

    \fill[ fill=gray!30](-0.1,1)--++(0:2.1)--++(-90:1)--++(180:2.1)--++(90:1);

    \draw[thick](0,0)--++(0:2)--++(90:1);

     \draw[thick,fill=white] (7,1)--++(90:1)--++(0:1)--++(-90:1)--++(180:1);
  \draw(7.5,1.5) node {1};

\draw[thick] (1,-1)--++(90:1)--++(0:1)--++(-90:1)--++(180:1);
 \draw(1.5,-0.5) node {1};

   \draw[thick,fill=white] (-1,-2-3)--++(90:1)--++(0:1)--++(-90:1)--++(180:1);
 \draw(-0.5,-2-2.5) node {2};

   \draw[thick] (-1,-2-4)--++(90:1)--++(0:1)--++(-90:1)--++(180:1);
 \draw(-0.5,-2-3.5) node {3};

  \draw[thick] (-1,-2-5)--++(90:1)--++(0:1)--++(-90:1)--++(180:1);
 \draw(-0.5,-2-4.5) node {4};
    \draw[thick] (-1,-2-6)--++(90:1)--++(0:1)--++(-90:1)--++(180:1);
 \draw(-0.5,-2-5.5) node {5};

  \draw[thick,densely dotted] (-1,-2-6)--++(90:-2);
    \draw[thick,densely dotted] (0,-2-6)--++(90:-2);

   \draw[thick] (-1,-2-9)--++(90:1)--++(0:1)--++(-90:1)--++(180:1);
 \draw(-0.5,-2-8.5) node {$a $};

 \end{tikzpicture}
\end{minipage} $$
\caption{The tableau on the left is the unique Littlewood--Richardson tableau of shape
 $(a+2,a+1,3,1^{2a-4}) \setminus ss[(a+1)]$
 and weight \color{blue}
  $ss[(a-1,1)]$.
   \color{black}
The tableau on the right is the one of three Littlewood--Richardson tableaux of shape
 $(a+2,a+1,3,1^{2a-4}) \setminus ss[(a,1)]$
\color{black} 
 and weight $ss[(a)]$.
  }
  \label{641}
\end{figure}
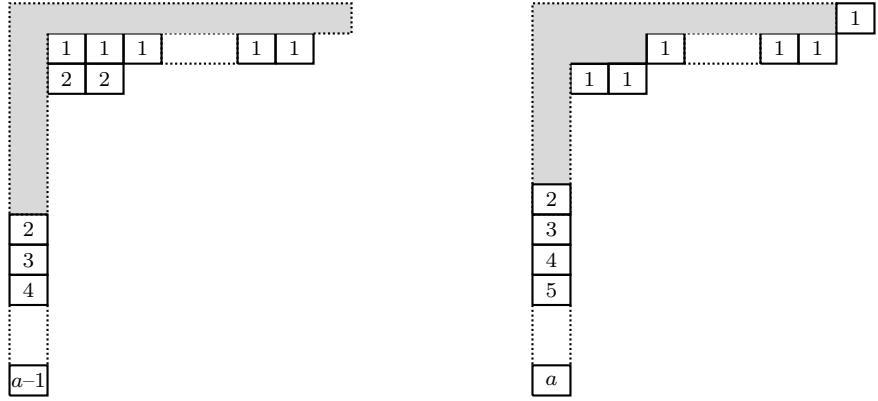

\begin{figure}[ht!]
$$ \scalefont{0.7}
\begin{minipage}{5.75cm} \begin{tikzpicture}  [xscale=0.5,yscale=0.4]

 \draw[thick](0,0)--++(90:1)--++(0:1)--++(-90:1)--++(180:1);
 \draw(0.5,0.5) node {1};
 \draw[thick](1,0)--++(90:1)--++(0:1)--++(-90:1)--++(180:1);
 \draw(1.5,0.5) node {1};
  \draw[thick](2,0)--++(90:1)--++(0:1)--++(-90:1)--++(180:1);
 \draw(2.5,0.5) node {1};
  \draw[thick](5,0)--++(90:1)--++(0:1)--++(-90:1)--++(180:1);
 \draw(5.5,0.5) node {1};

\draw[thick, densely dotted] (3,0)--++(90:1)--++(0:2)--++(-90:1)--++(180:2);
   \draw[thick](6,0)--++(90:1)--++(0:1)--++(-90:1)--++(180:1);
 \draw(6.5,0.5) node {1};

  \draw[thick] (0,-1)--++(90:1)--++(0:1)--++(-90:1)--++(180:1);
 \draw(0.5,-0.5) node {1};

    \draw[thick,densely dotted,fill=gray!30](7,1)--++(0:0)--++(90:1)--++(180:1)--++(180:7)--++(-90:7)--++(0:1)--
   ++(90:6);

    \fill[ fill=gray!30](-0.1,1)--++(0:2.1)--++(-90:1)--++(180:2.1)--++(90:1);

    \draw[thick](0,0)--++(0:2)--++(90:1);

      \draw[thick,fill=white] (7,1)--++(90:1)--++(0:1)--++(-90:1)--++(180:1);
   \draw(7.5,1.5) node {1};

\draw[thick] (1,-1)--++(90:1)--++(0:1)--++(-90:1)--++(180:1);
 \draw(1.5,-0.5) node {2};

   \draw[thick,fill=white] (-1,-2-3)--++(90:1)--++(0:1)--++(-90:1)--++(180:1);
 \draw(-0.5,-2-2.5) node {1};

   \draw[thick] (-1,-2-4)--++(90:1)--++(0:1)--++(-90:1)--++(180:1);
 \draw(-0.5,-2-3.5) node {3};

  \draw[thick] (-1,-2-5)--++(90:1)--++(0:1)--++(-90:1)--++(180:1);
 \draw(-0.5,-2-4.5) node {4};
    \draw[thick] (-1,-2-6)--++(90:1)--++(0:1)--++(-90:1)--++(180:1);
 \draw(-0.5,-2-5.5) node {5};

  \draw[thick,densely dotted] (-1,-2-6)--++(90:-2);
    \draw[thick,densely dotted] (0,-2-6)--++(90:-2);

   \draw[thick] (-1,-2-9)--++(90:1)--++(0:1)--++(-90:1)--++(180:1);
 \draw(-0.5,-2-8.5) node {$a $};

 \end{tikzpicture}
\end{minipage}
\qquad
\qquad
\begin{minipage}{5.75cm} \begin{tikzpicture}  [xscale=0.5,yscale=0.4]

 \draw[thick](0,0)--++(90:1)--++(0:1)--++(-90:1)--++(180:1);
 \draw(0.5,0.5) node {1};
 \draw[thick](1,0)--++(90:1)--++(0:1)--++(-90:1)--++(180:1);
 \draw(1.5,0.5) node {1};
  \draw[thick](2,0)--++(90:1)--++(0:1)--++(-90:1)--++(180:1);
 \draw(2.5,0.5) node {1};
  \draw[thick](5,0)--++(90:1)--++(0:1)--++(-90:1)--++(180:1);
 \draw(5.5,0.5) node {1};

\draw[thick, densely dotted] (3,0)--++(90:1)--++(0:2)--++(-90:1)--++(180:2);
   \draw[thick](6,0)--++(90:1)--++(0:1)--++(-90:1)--++(180:1);
 \draw(6.5,0.5) node {2};

  \draw[thick] (0,-1)--++(90:1)--++(0:1)--++(-90:1)--++(180:1);
 \draw(0.5,-0.5) node {1};

    \draw[thick,densely dotted,fill=gray!30](7,1)--++(0:0)--++(90:1)--++(180:1)--++(180:7)--++(-90:7)--++(0:1)--
   ++(90:6);

    \fill[ fill=gray!30](-0.1,1)--++(0:2.1)--++(-90:1)--++(180:2.1)--++(90:1);

    \draw[thick](0,0)--++(0:2)--++(90:1);

     \draw[thick,fill=white] (7,1)--++(90:1)--++(0:1)--++(-90:1)--++(180:1);
  \draw(7.5,1.5) node {1};

\draw[thick] (1,-1)--++(90:1)--++(0:1)--++(-90:1)--++(180:1);
 \draw(1.5,-0.5) node {1};

   \draw[thick,fill=white] (-1,-2-3)--++(90:1)--++(0:1)--++(-90:1)--++(180:1);
 \draw(-0.5,-2-2.5) node {1};

   \draw[thick] (-1,-2-4)--++(90:1)--++(0:1)--++(-90:1)--++(180:1);
 \draw(-0.5,-2-3.5) node {3};

  \draw[thick] (-1,-2-5)--++(90:1)--++(0:1)--++(-90:1)--++(180:1);
 \draw(-0.5,-2-4.5) node {4};
    \draw[thick] (-1,-2-6)--++(90:1)--++(0:1)--++(-90:1)--++(180:1);
 \draw(-0.5,-2-5.5) node {5};

  \draw[thick,densely dotted] (-1,-2-6)--++(90:-2);
    \draw[thick,densely dotted] (0,-2-6)--++(90:-2);

   \draw[thick] (-1,-2-9)--++(90:1)--++(0:1)--++(-90:1)--++(180:1);
 \draw(-0.5,-2-8.5) node {$a $};

 \end{tikzpicture}
\end{minipage}
$$
\caption{The two remaining Littlewood--Richardson tableaux of shape
 $(a+2,a+1,3,1^{2a-4}) \setminus ss[(a,1)]$
 and weight $ss[(a)]$. }
 \label{642}
\end{figure}
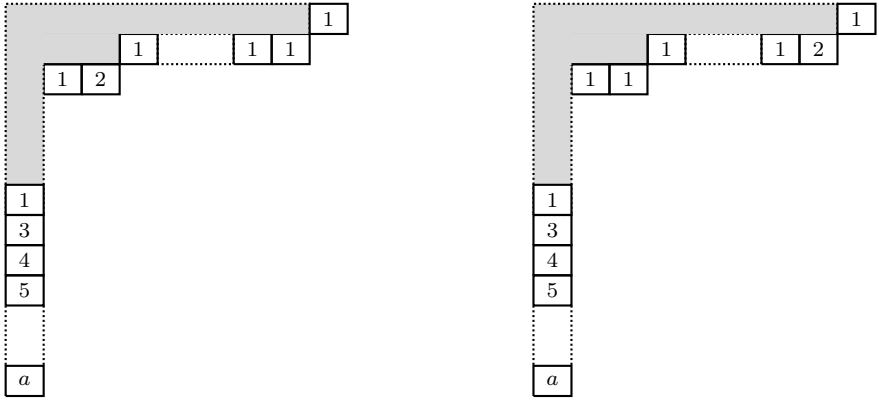

 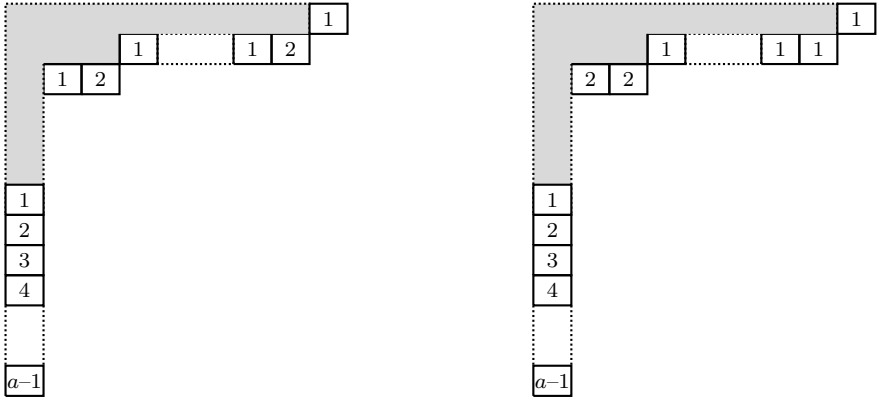
\begin{figure}[ht!]
 $$ \scalefont{0.7}
 \begin{minipage}{5.75cm} \begin{tikzpicture}  [xscale=0.5,yscale=0.4]

 \draw[thick](0,0)--++(90:1)--++(0:1)--++(-90:1)--++(180:1);
 \draw(0.5,0.5) node {1};
 \draw[thick](1,0)--++(90:1)--++(0:1)--++(-90:1)--++(180:1);
 \draw(1.5,0.5) node {1};
  \draw[thick](2,0)--++(90:1)--++(0:1)--++(-90:1)--++(180:1);
 \draw(2.5,0.5) node {1};
  \draw[thick](5,0)--++(90:1)--++(0:1)--++(-90:1)--++(180:1);
 \draw(5.5,0.5) node {1};

\draw[thick, densely dotted] (3,0)--++(90:1)--++(0:2)--++(-90:1)--++(180:2);
   \draw[thick](6,0)--++(90:1)--++(0:1)--++(-90:1)--++(180:1);
 \draw(6.5,0.5) node {2};

  \draw[thick] (0,-1)--++(90:1)--++(0:1)--++(-90:1)--++(180:1);
 \draw(0.5,-0.5) node {1};

    \draw[thick,densely dotted,fill=gray!30](7,1)--++(0:0)--++(90:1)--++(180:1)--++(180:7)--++(-90:7)--++(0:1)--
   ++(90:6);

    \fill[ fill=gray!30](-0.1,1)--++(0:2.1)--++(-90:1)--++(180:2.1)--++(90:1);

    \draw[thick](0,0)--++(0:2)--++(90:1);

     \draw[thick,fill=white] (7,1)--++(90:1)--++(0:1)--++(-90:1)--++(180:1);
  \draw(7.5,1.5) node {1};

\draw[thick] (1,-1)--++(90:1)--++(0:1)--++(-90:1)--++(180:1);
 \draw(1.5,-0.5) node {2};

   \draw[thick,fill=white] (-1,-2-3)--++(90:1)--++(0:1)--++(-90:1)--++(180:1);
 \draw(-0.5,-2-2.5) node {1};

   \draw[thick] (-1,-2-4)--++(90:1)--++(0:1)--++(-90:1)--++(180:1);
 \draw(-0.5,-2-3.5) node {2};

  \draw[thick] (-1,-2-5)--++(90:1)--++(0:1)--++(-90:1)--++(180:1);
 \draw(-0.5,-2-4.5) node {3};
    \draw[thick] (-1,-2-6)--++(90:1)--++(0:1)--++(-90:1)--++(180:1);
 \draw(-0.5,-2-5.5) node {4};

  \draw[thick,densely dotted] (-1,-2-6)--++(90:-2);
    \draw[thick,densely dotted] (0,-2-6)--++(90:-2);

   \draw[thick] (-1,-2-9)--++(90:1)--++(0:1)--++(-90:1)--++(180:1);
 \draw(-0.5,-2-8.5) node {$a\text{--}1 $};

 \end{tikzpicture}
\end{minipage}
\qquad
\qquad
\begin{minipage}{5.75cm} \begin{tikzpicture}  [xscale=0.5,yscale=0.4]

 \draw[thick](0,0)--++(90:1)--++(0:1)--++(-90:1)--++(180:1);
 \draw(0.5,0.5) node {1};
 \draw[thick](1,0)--++(90:1)--++(0:1)--++(-90:1)--++(180:1);
 \draw(1.5,0.5) node {1};
  \draw[thick](2,0)--++(90:1)--++(0:1)--++(-90:1)--++(180:1);
 \draw(2.5,0.5) node {1};
  \draw[thick](5,0)--++(90:1)--++(0:1)--++(-90:1)--++(180:1);
 \draw(5.5,0.5) node {1};

\draw[thick, densely dotted] (3,0)--++(90:1)--++(0:2)--++(-90:1)--++(180:2);
   \draw[thick](6,0)--++(90:1)--++(0:1)--++(-90:1)--++(180:1);
 \draw(6.5,0.5) node {1};

  \draw[thick] (0,-1)--++(90:1)--++(0:1)--++(-90:1)--++(180:1);
 \draw(0.5,-0.5) node {2};

    \draw[thick,densely dotted,fill=gray!30](7,1)--++(0:0)--++(90:1)--++(180:1)--++(180:7)--++(-90:7)--++(0:1)--
   ++(90:6);

    \fill[ fill=gray!30](-0.1,1)--++(0:2.1)--++(-90:1)--++(180:2.1)--++(90:1);

    \draw[thick](0,0)--++(0:2)--++(90:1);

     \draw[thick,fill=white] (7,1)--++(90:1)--++(0:1)--++(-90:1)--++(180:1);
  \draw(7.5,1.5) node {1};

\draw[thick] (1,-1)--++(90:1)--++(0:1)--++(-90:1)--++(180:1);
 \draw(1.5,-0.5) node {2};

   \draw[thick,fill=white] (-1,-2-3)--++(90:1)--++(0:1)--++(-90:1)--++(180:1);
 \draw(-0.5,-2-2.5) node {1};

   \draw[thick] (-1,-2-4)--++(90:1)--++(0:1)--++(-90:1)--++(180:1);
 \draw(-0.5,-2-3.5) node {2};

  \draw[thick] (-1,-2-5)--++(90:1)--++(0:1)--++(-90:1)--++(180:1);
 \draw(-0.5,-2-4.5) node {3};
    \draw[thick] (-1,-2-6)--++(90:1)--++(0:1)--++(-90:1)--++(180:1);
 \draw(-0.5,-2-5.5) node {4};

  \draw[thick,densely dotted] (-1,-2-6)--++(90:-2);
    \draw[thick,densely dotted] (0,-2-6)--++(90:-2);

   \draw[thick] (-1,-2-9)--++(90:1)--++(0:1)--++(-90:1)--++(180:1);
 \draw(-0.5,-2-8.5) node {$a\text{--}1 $};

 \end{tikzpicture}
\end{minipage}
$$
\caption{The 
   Littlewood--Richardson 
   tableaux of shape $(a+2,a+1,3,1^{2a-4}) \setminus ss[(a ,1)]$
 and weight $ss[(a-1,1)]$.   }
\label{643}

\end{figure}

Now we consider the 2-column case  $\nu=(2^a,1^{b})$.
For $a,b >1$ the result follows  from \cref{usefulrestulwhichstopproblem}.
Let  $\nu=(2^a,1)$.  We claim that
\begin{align}
& {\color{white} = }
  \langle s_{(2^a,1 )}\circ s_{(2)}	\mid s_{(a+2,a+1,3,1^{2a-4})}\rangle
  \\
   \label{6and4}
   &=  \langle e_{(a+1,a)} \circ s_{(2)} \mid s_{(a+2,a+1,3,1^{2a-4})}\rangle-
    \langle e_{(a+2,a-1)}  \circ s_{(2)}\mid s_{(a+2,a+1,3,1^{2a-4})}\rangle
    \\
&=6-4=2.
\end{align}
The 6 Littlewood--Richardson tableaux arising from the first term in \cref{6and4} are depicted in \cref{641,642,643} and
the 4 Littlewood--Richardson tableaux  arising from the second term in \cref{6and4} are depicted in \cref{644,645}.

Let  $a=1$ and $n\ge 9$.
  We claim that $$\langle s_{(2,1^{n-2})}\circ s_{(2)}	\mid s_{ss[n-4,3,1]}\rangle =2.$$
To see this, we set
  $$
  \beta_1=(n-5,3,1)
  \quad
  \beta_2=(n-4,2,1)
  \quad
  \beta_3=(n-4,3)
      $$
and we note that
      $$
\langle     s_{ss[\beta_i]}\boxtimes  s_{(2)}
\mid s_{ss[n-4,3,1]}\rangle =1$$
for $i=1,2,3$,  
whereas $
\langle     s_{ss[\gamma]}\boxtimes  s_{(2)}
\mid s_{ss[n-4,3,1]}\rangle =0$ for all other partitions $\gamma \vdash n-1$ with distinct parts.
\color{black}
  Now, simply note that
\begin{align*}
s_{(2,1^{n-2})}\circ s_{(2)}		
  &=
 e_{(n-1,1)} \circ s_{(2)} -e_{(n)}  \circ s_{(2)}\\
\intertext{and therefore} 
  \langle    s_{(2,1^{n-2})}\circ s_{(2)} \mid s_{ss[n-4,3,1]}\rangle	\color{black}
	 &=
\sum_{1\leq i \leq 3}\langle     s_{ss[\beta_i]}\boxtimes  s_{(2)}
\mid s_{ss[n-4,3,1]}\rangle
-
\langle    s_{ss[n-4,3,1]}
\mid s_{ss[n-4,3,1]}\rangle \\
& =3-1=2\end{align*}

 \begin{figure}[ht!]
 $$ \scalefont{0.7}
 \begin{minipage}{5.75cm} \begin{tikzpicture}  [xscale=0.5,yscale=0.4]

 \draw[thick](0,0)--++(90:1)--++(0:1)--++(-90:1)--++(180:1);
 \draw(0.5,0.5) node {1};
 \draw[thick](1,0)--++(90:1)--++(0:1)--++(-90:1)--++(180:1);
 \draw(1.5,0.5) node {1};
  \draw[thick](2,0)--++(90:1)--++(0:1)--++(-90:1)--++(180:1);
 \draw(2.5,0.5) node {1};
  \draw[thick](5,0)--++(90:1)--++(0:1)--++(-90:1)--++(180:1);
 \draw(5.5,0.5) node {1};

\draw[thick, densely dotted] (3,0)--++(90:1)--++(0:2)--++(-90:1)--++(180:2);
   \draw[thick](6,0)--++(90:1)--++(0:1)--++(-90:1)--++(180:1);
 \draw(6.5,0.5) node {1};

  \draw[thick] (0,-1)--++(90:1)--++(0:1)--++(-90:1)--++(180:1);
 \draw(0.5,-0.5) node {1};

       \draw[thick,densely dotted,fill=gray!30](7,1)--++(0:1)--++(90:1)--++(180:1)--++(180:8)--++(-90:7)--++(0:1)--
   ++(90:6);

    \fill[ fill=gray!30](-0.1,1)--++(0:2.1)--++(-90:1)--++(180:2.1)--++(90:1);

    \draw[thick](0,0)--++(0:2)--++(90:1);


\draw[thick] (1,-1)--++(90:1)--++(0:1)--++(-90:1)--++(180:1);
 \draw(1.5,-0.5) node {2};

   \draw[thick,fill=white] (-1,-2-3)--++(90:1)--++(0:1)--++(-90:1)--++(180:1);
 \draw(-0.5,-2-2.5) node {1};

   \draw[thick] (-1,-2-4)--++(90:1)--++(0:1)--++(-90:1)--++(180:1);
 \draw(-0.5,-2-3.5) node {3};

  \draw[thick] (-1,-2-5)--++(90:1)--++(0:1)--++(-90:1)--++(180:1);
 \draw(-0.5,-2-4.5) node {4};
    \draw[thick] (-1,-2-6)--++(90:1)--++(0:1)--++(-90:1)--++(180:1);
 \draw(-0.5,-2-5.5) node {5};

  \draw[thick,densely dotted] (-1,-2-6)--++(90:-2);
    \draw[thick,densely dotted] (0,-2-6)--++(90:-2);

   \draw[thick] (-1,-2-9)--++(90:1)--++(0:1)--++(-90:1)--++(180:1);
 \draw(-0.5,-2-8.5) node {$a\text{--}1$};

 \end{tikzpicture}
\end{minipage}
\qquad
\qquad
\begin{minipage}{5.75cm} \begin{tikzpicture}  [xscale=0.5,yscale=0.4]

 \draw[thick](0,0)--++(90:1)--++(0:1)--++(-90:1)--++(180:1);
 \draw(0.5,0.5) node {1};
 \draw[thick](1,0)--++(90:1)--++(0:1)--++(-90:1)--++(180:1);
 \draw(1.5,0.5) node {1};
  \draw[thick](2,0)--++(90:1)--++(0:1)--++(-90:1)--++(180:1);
 \draw(2.5,0.5) node {1};
  \draw[thick](5,0)--++(90:1)--++(0:1)--++(-90:1)--++(180:1);
 \draw(5.5,0.5) node {1};

\draw[thick, densely dotted] (3,0)--++(90:1)--++(0:2)--++(-90:1)--++(180:2);
   \draw[thick](6,0)--++(90:1)--++(0:1)--++(-90:1)--++(180:1);
 \draw(6.5,0.5) node {1};

  \draw[thick] (0,-1)--++(90:1)--++(0:1)--++(-90:1)--++(180:1);
 \draw(0.5,-0.5) node {1};

       \draw[thick,densely dotted,fill=gray!30](7,1)--++(0:1)--++(90:1)--++(180:1)--++(180:8)--++(-90:7)--++(0:1)--
   ++(90:6);

    \fill[ fill=gray!30](-0.1,1)--++(0:2.1)--++(-90:1)--++(180:2.1)--++(90:1);

    \draw[thick](0,0)--++(0:2)--++(90:1);

\draw[thick] (1,-1)--++(90:1)--++(0:1)--++(-90:1)--++(180:1);
 \draw(1.5,-0.5) node {1};

   \draw[thick,fill=white] (-1,-2-3)--++(90:1)--++(0:1)--++(-90:1)--++(180:1);
 \draw(-0.5,-2-2.5) node {2};

   \draw[thick] (-1,-2-4)--++(90:1)--++(0:1)--++(-90:1)--++(180:1);
 \draw(-0.5,-2-3.5) node {3};

  \draw[thick] (-1,-2-5)--++(90:1)--++(0:1)--++(-90:1)--++(180:1);
 \draw(-0.5,-2-4.5) node {4};
    \draw[thick] (-1,-2-6)--++(90:1)--++(0:1)--++(-90:1)--++(180:1);
 \draw(-0.5,-2-5.5) node {5};

  \draw[thick,densely dotted] (-1,-2-6)--++(90:-2);
    \draw[thick,densely dotted] (0,-2-6)--++(90:-2);

   \draw[thick] (-1,-2-9)--++(90:1)--++(0:1)--++(-90:1)--++(180:1);
 \draw(-0.5,-2-8.5) node {$a\text{--}1$};

 \end{tikzpicture}
\end{minipage}
$$
\caption{The 
 Littlewood--Richardson \color{black}  tableaux of shape 
$(a+2,a+1,3,1^{2a-4}) \setminus  ss[(a+1 ,1)]$
   \color{black}
 and weight $ss[(a-1)]$.   }
\label{644}

\end{figure}
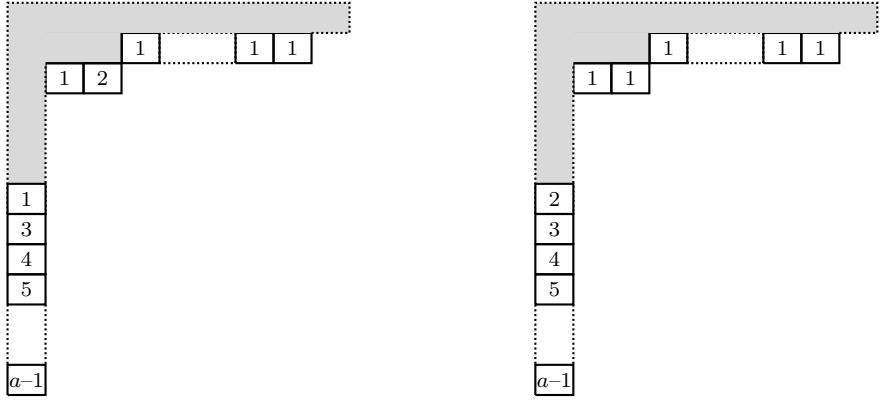

 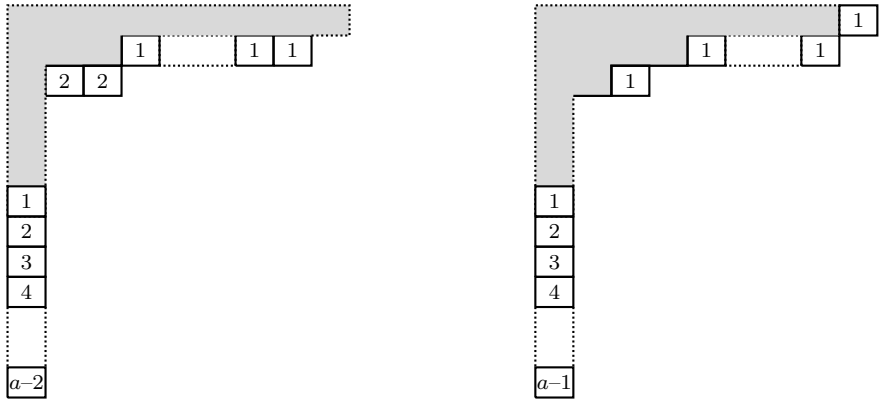
\begin{figure}[ht!]
 $$ \scalefont{0.7}
 \begin{minipage}{5.75cm} \begin{tikzpicture}  [xscale=0.5,yscale=0.4]

 \draw[thick](0,0)--++(90:1)--++(0:1)--++(-90:1)--++(180:1);
 \draw(0.5,0.5) node {1};
 \draw[thick](1,0)--++(90:1)--++(0:1)--++(-90:1)--++(180:1);
 \draw(1.5,0.5) node {1};
  \draw[thick](2,0)--++(90:1)--++(0:1)--++(-90:1)--++(180:1);
 \draw(2.5,0.5) node {1};
  \draw[thick](5,0)--++(90:1)--++(0:1)--++(-90:1)--++(180:1);
 \draw(5.5,0.5) node {1};

\draw[thick, densely dotted] (3,0)--++(90:1)--++(0:2)--++(-90:1)--++(180:2);
   \draw[thick](6,0)--++(90:1)--++(0:1)--++(-90:1)--++(180:1);
 \draw(6.5,0.5) node {1};

  \draw[thick] (0,-1)--++(90:1)--++(0:1)--++(-90:1)--++(180:1);
 \draw(0.5,-0.5) node {2};

       \draw[thick,densely dotted,fill=gray!30](7,1)--++(0:1)--++(90:1)--++(180:1)--++(180:8)--++(-90:7)--++(0:1)--
   ++(90:6);

    \fill[ fill=gray!30](-0.1,1)--++(0:2.1)--++(-90:1)--++(180:2.1)--++(90:1);

    \draw[thick](0,0)--++(0:2)--++(90:1);

\draw[thick] (1,-1)--++(90:1)--++(0:1)--++(-90:1)--++(180:1);
 \draw(1.5,-0.5) node {2};

   \draw[thick,fill=white] (-1,-2-3)--++(90:1)--++(0:1)--++(-90:1)--++(180:1);
 \draw(-0.5,-2-2.5) node {1};

   \draw[thick] (-1,-2-4)--++(90:1)--++(0:1)--++(-90:1)--++(180:1);
 \draw(-0.5,-2-3.5) node {2};

  \draw[thick] (-1,-2-5)--++(90:1)--++(0:1)--++(-90:1)--++(180:1);
 \draw(-0.5,-2-4.5) node {3};
    \draw[thick] (-1,-2-6)--++(90:1)--++(0:1)--++(-90:1)--++(180:1);
 \draw(-0.5,-2-5.5) node {4};

  \draw[thick,densely dotted] (-1,-2-6)--++(90:-2);
    \draw[thick,densely dotted] (0,-2-6)--++(90:-2);

   \draw[thick] (-1,-2-9)--++(90:1)--++(0:1)--++(-90:1)--++(180:1);
 \draw(-0.5,-2-8.5) node {$a\text{--}2$};

 \end{tikzpicture}
\end{minipage}
\qquad\qquad
 \begin{minipage}{5.75cm} \begin{tikzpicture}  [xscale=0.5,yscale=0.4]

 \draw[thick](0,0)--++(90:1)--++(0:1)--++(-90:1)--++(180:1);
 \draw(0.5,0.5) node {1};
 \draw[thick](1,0)--++(90:1)--++(0:1)--++(-90:1)--++(180:1);
 \draw(1.5,0.5) node {1};
  \draw[thick](2,0)--++(90:1)--++(0:1)--++(-90:1)--++(180:1);
 \draw(2.5,0.5) node {1};
  \draw[thick](3,0)--++(90:1)--++(0:1)--++(-90:1)--++(180:1);
 \draw(3.5,0.5) node {1};

\draw[thick, densely dotted] (4,0)--++(90:1)--++(0:2)--++(-90:1)--++(180:2);
   \draw[thick](6,0)--++(90:1)--++(0:1)--++(-90:1)--++(180:1);
 \draw(6.5,0.5) node {1};

  \draw[thick] (0,-1)--++(90:1)--++(0:1)--++(-90:1)--++(180:1);
 \draw(0.5,-0.5) node {2};

   \draw[thick,densely dotted,fill=gray!30](7,1)--++(0:0)--++(90:1)--++(180:1)--++(180:7)--++(-90:7)--++(0:1)--
   ++(90:6);

    \fill[ fill=gray!30](-0.1,1.1)--++(0:3.1)--++(-90:1.1)--++(180:3.1)--++(90:1);
    \fill[ fill=gray!30](-0.1,-1)--++(00:1.1)--++(90:1)--++(180:1.1);

    \draw[thick](0,-1)--++(0:1)--++(90:1)--++(0:2)--++(90:1);

 \draw[thick] (1,-1)--++(90:1)--++(0:1)--++(-90:1)--++(180:1);
 \draw(1.5,-0.5) node {1};

  \draw[thick] (7,1)--++(90:1)--++(0:1)--++(-90:1)--++(180:1);
   \draw(7.5,1.5) node {1};

 \draw[thick,fill=white] (-1,-2-3)--++(90:1)--++(0:1)--++(-90:1)--++(180:1);
 \draw(-0.5,-2-2.5) node {1};

   \draw[thick] (-1,-2-4)--++(90:1)--++(0:1)--++(-90:1)--++(180:1);
 \draw(-0.5,-2-3.5) node {2};

  \draw[thick] (-1,-2-5)--++(90:1)--++(0:1)--++(-90:1)--++(180:1);
 \draw(-0.5,-2-4.5) node {3};
    \draw[thick] (-1,-2-6)--++(90:1)--++(0:1)--++(-90:1)--++(180:1);
 \draw(-0.5,-2-5.5) node {4};

  \draw[thick,densely dotted] (-1,-2-6)--++(90:-2);
    \draw[thick,densely dotted] (0,-2-6)--++(90:-2);

   \draw[thick] (-1,-2-9)--++(90:1)--++(0:1)--++(-90:1)--++(180:1);
 \draw(-0.5,-2-8.5) node {$a\text{--}1$};

 \end{tikzpicture}
\end{minipage}
$$
\caption{The left  tableau is the unique Littlewood--Richardson tableau of shape
 $(a+2,a+1,3,1^{2a-4}) \setminus  ss[(a+1 ,1)]$\color{black}
 and weight $ss[(a-2,1)]$. The right tableau
 is the unique Littlewood--Richardson tableau of shape 
  $(a+2,a+1,3,1^{2a-4}) \setminus  ss[(a ,2)]$
 \color{black}
 and weight $ ss[(a-1)]$ }
\label{645}
 \qedhere

\end{figure}
       \end{proof}

We have now already considered  all partitions $\nu$ except hooks and fat hooks.
Firstly, we consider hooks.
As 2-line partitions have already been discussed, we need only consider hooks of length and width at least~3.

 \begin{prop}\label{hooky1}\label{hooky}
 If   $\nu=(n-a,1^a)$ for $2\leq a < n-2$, then $p(\nu,\mu)>1$
 for all $\mu\vdash m>1$ except for the cases listed in \cref{conj}.
  \end{prop}
\begin{proof}
  By   \cref{Brion,2liners} it suffices to consider partitions $\nu$ of the form $(3,1^a)$ for $a=2,3,4,5, 6$.
   In this case we obtain 5 small rank seeds of multiplicity as follows:
  $$\langle s_{(3,1^a)} \circ s_{(2)} \mid s_{(4+a,3,1^{a-1})}\rangle =2$$
  for $a=2,3,4,5, 6$ (by computer calculation).
We   hence  deduce $p(\nu,(2))>1$ whenever $\nu$ is a proper hook not listed in \cref{conj}.
 Since the set of hooks under consideration is closed under conjugation,
we deduce the result using \cref{cor:2-important}.
\end{proof}

 \begin{prop}\label{coolcor}\label{fathooks2}
Let  $\nu$ be a  proper fat hook.
 Then $p(\nu,\mu)>1$ for any  partition
$\mu  $ such that   $|\mu|> 1$.
 \end{prop}

 \begin{proof}
Let $\mathcal N$ be the set of all proper fat hooks.
Let $\nu \vdash n$ be in $\mathcal N$.
 If $\nu_1=\nu_2$ then \cref{yaaaas2} shows that
$$2 \leq \langle s_\nu\circ s_{(2)} \mid s_{\overline\nu -\varepsilon_1+\varepsilon_c}\rangle
$$
for any $\varepsilon_c \in {\rm Add}(\overline{\nu}-\varepsilon_1)$ with $c>2$.  
Otherwise, $\nu$ is a near rectangle of the form $\nu=(a+k,a^b)$ with $k\geq 1$ and $a,b\geq 2$.
\color{black}
In this latter case, we apply \cref{the latter} to the rectangle $\rho =(a^{b+1})\vdash r$ and
obtain by \cref{Brion} for $\nu=\rho+(k)$:
$$
2 = p(\rho,(2),\bar\rho -2\varepsilon_1+2\varepsilon_2)
\leq p(\nu,(2),\bar\rho+(2k)-2\varepsilon_1+2\varepsilon_2).
$$
Thus, in any case $p(\nu,(2))>1$.
As $\mathcal N$ is closed under conjugation, the result now follows by
\cref{cor:2-important}.
 \end{proof}

  \begin{prop}\label{2222}
 Let $\nu\vdash 2$.  Then $p(\nu,\mu)>1$ for all
 $\mu$ not appearing  in the exceptional cases of \cref{conj}$(ii)$.
  \end{prop}
\begin{proof}
  We have checked that the result is true for all partitions $\mu$ of size
 at most 10 by computer calculation.
Now, we let $\nu \vdash 2$ and suppose that
 $\mu$ is  either
 \begin{itemize}[leftmargin=*]
 \item[$(i)$]   a
  fat hook not equal to  $(a^b)$, $(a+1,a^{b-1})$, $(a^b,1)$,   $(a^{b-1},a-1)$, or a hook;

 \item[$(ii)$]  a partition with at least 3 removable nodes;
 \end{itemize}
 we will show that $p(\nu,\mu)>1$.

   We first assume that $\mu$  satisfies $(i)$.
   We wish to use the semigroup property of \cref{PW2}  to remove
     columns of  $\mu$ and then conjugate (note that   the condition on $\nu$ is conjugation invariant) and again remove more columns until we obtain  a list of  the smallest possible fat hook partitions $\widehat{\mu}$ such that $s_\nu \circ s_{\widehat{\mu}}$ contains multiplicities.  Up to conjugation,  the partition   $(4,2)$ is the unique smallest fat hook  which is
      not
     equal to an almost rectangle  or a hook.  However $(4,2)$ is on our list of exceptional products for which $s_{\nu}\circ s_{(4,2)}$ {\em is} multiplicity-free --- and so if we reach $\widehat{\mu}=(4,2)$ (or its conjugate) we have removed a row or column too many from $\mu$.  Therefore our list of seeds is given by the four fat hook partitions obtained by   adding a row or column to $(4,2)$, namely
   $\widehat{\mu}=(5,2)$, $(5,3)$,   $(4^2,2)$, or $(3^2,1^2)$ up to conjugation.
 Now such $\widehat{\mu}$ has $|\widehat{\mu}|\leq 10$ and
 hence is covered by computer calculation.
 Thus we deduce that
 any product  $s_\nu\circ s_{\mu}$ can be seen to have multiplicities  by reducing it to one of the form $s_\nu\circ s_{\widehat{\mu}}$  using  \cref{cor:semigroup}.

Now suppose that $\mu$ satisfies $(ii)$.
Using \cref{PW2} we can remove successive {columns} from anywhere in $\mu$ until we obtain  a  3 column partition  $\widehat{\mu}$ with 3 removable nodes (it does not matter how we do this).
 We then conjugate (as the condition on $\nu$ is conjugation invariant) using \cref{conjugate}
  and again remove successive columns until we obtain the partition
  $\overline{\mu}=(3,2,1)$.
 Finally we note that
$$
 2=
 \langle s_{\nu} \circ s_{(3,2,1) }\mid  s_{
 (5,4,2,1)} \rangle
  $$
for $\nu\vdash 2$  and so the result follows.
 \end{proof}

  \begin{prop}\label{finallinearcase}
 Let
 $\nu$ be a linear partition of $n\ge 3$.
  Then $p(\nu,\mu)>1$ for all
 $\mu$ not appearing  in the exceptional cases of \cref{conj}.
 \end{prop}

\begin{proof}
Let $\mu$ be a partition of $m$.
We already know that for
$m \le 2$ we have $p(\nu,\mu)=1$,
so we assume now that $m\ge 3$.
We also note that for $m+n\leq 8$ the claim
is checked by computer (see \cref{Data}).
So from now on, we assume that $m+n\geq 9$.

We first suppose that $\mu$ is also a  linear partition.

 We now first consider the case when $\nu=(n)$.
 We can use \cref{cor:semigroup}
 to remove boxes
 from $\nu$ and $\mu$ until we obtain a seed 
  (see \cref{Data})
 \color{black}
 of the form
$$
s_{(3)}\circ s_{(6)}
 \quad
s_{(4)}\circ s_{(4)}
  \quad s_{(5)}\circ s_{(3)},
 $$
$$
s_{(3)}\circ s_{(1^6)}
 \quad
 s_{(4)}\circ s_{(1^4)}
 \quad s_{(6)}\circ s_{(1^3)}.
$$
We now proceed to the  case when $\nu=(1^n)$.
  If $m$ is odd, then by  \cref{max-conjugate}
 we have $p((1^n),\mu)=p((n),\mu^T)$ and so the result follows from the above.
 If $m$ is even, then we can remove a box from $\mu$ using
 \cref{cor:semigroup}	and then the result follows from the $m$ odd case
 if $m+n>9$ 	
	(note that $m-1\geq 3$ if $m$ is even and so this is fine);
if $m+n=9$ we only need to check by computer that we have the seed
$$s_{(1^5)}\circ s_{(4)} .
$$

Next suppose that $\mu$ is an arbitrary non-linear  rectangle $(a^b)$.
 If $a,b\geq 3$ then
we remove rows and column of $\mu$ using \cref{cor:semigroup} until
we obtain the partition $\widehat{\mu}=(3^3)$, with $p(\nu,(3^3))\le p(\nu,\mu)$.
  Since $9$ is odd,
 using \cref{max-conjugate} reduces to showing that $p((n),(3^3))>1$.

     Using \cref{cor:semigroup} again, we have
    $p((n),(3^3))\ge p((3),(3^3))>1$,
    and the result follows for $\mu=(a^b)$ for $a,b\geq 3$.      By \cref{max-conjugate} it only remains to consider 2-line rectangles $\mu=(a^2)$, $a\ge 2$.      Using \cref{cor:semigroup}
    once more
      we find
		$p((n),(a^2)) \ge p((3),(2^2))=2$ for $n\ge 3$ and $a\ge 2$,
    $p((1^n),(a^2)) \ge p((1^4),(2^2))=3$ for $n\ge 4$ and $a\ge 2$,
		and $p((1^3),(a^2)) \ge p((1^3),(3^2))=2$ for $a\ge 3$.
    Thus the result follows in this case.

Finally, suppose that  $\mu$ is not a rectangle.
We now use all parts of \cref{cor:semigroup} in turn, i.e.,
we remove all rows above the last non-linear hook of~$\mu$,
all  columns to the left of this hook, and then almost all
boxes in the arm and almost all boxes in the leg, and we find
$$p(\nu,\mu) \ge p(\nu,(2,1))=p(\nu^T,(2,1))\ge p((3),(2,1))=2.$$
Hence the result follows.
 \end{proof}

Since $\nu$ must be a linear partition, or a 2-line partition, or a hook, or a rectangle or a proper fat hook, or have (at least) 3 removable nodes  --- and we have proven \cref{conj} for  each of these different cases in turn  ---  the proof of  \cref{conj} is now complete.

\bibliographystyle{amsalpha}
\bibliography{master}

 \section{Data}\label{Data}

We now provide   all pairs of partitions $\nu\vdash n,\mu\vdash m$ with $n+m\leq 8$ 
for which the plethysm $s_\nu\circ s_\mu$ is not multiplicity-free, together with the corresponding value $p(\nu,\mu)>1$ (values 1 are suppressed in the tables below);
 for succinctness, we do not list the products which can be deduced by conjugation (as in \cref{conjugate}).  

Using monotonicity properties, in the main body of this paper
pairs in this region and slightly beyond serve as seeds for plethysms which are not multiplicity-free.
Hence, we also add further values for some pairs $\nu,\mu$ which are used as seeds for multiplicity
in the arguments.

\bigskip

\scalefont{0.9} $
\begin{array}{c||c|c||c|c||c|} 
\nu \backslash \mu
    & (4,2) & (3,2,1) & (5,2) & (4,2,1) &
		(4,2^2)
		\\
\hline
(2) & 2 & 2 & 2 & 3 & 2 \\
(1^2) & & 2 & 2 & 3 & 2 \\
\end{array}
$

\smallskip


$
\begin{array}{l||c||c|c|c||c|c|c|c||c|c||c|}
\nu \backslash \mu
    & (2,1) & (4) & (3,1) & (2^2) & (5) & (4,1) & (3,2) & (3,1^2) & (6) &(3^2) &(3^3)     \\
\hline
(3)   & 2 &   & 4 & 2 &   & 6 & 6 & 7 & 2	&&9\\
(2,1) & 3 & 2 & 7 & 2 & 2 & 10& 11& 12 & 2&&\\
(1^3) & 2 &   & 3 &   &   & 5 & 6 & 7 &	&2	&
\end{array}
$

\smallskip

$
\begin{array}{l||c|c||c|c|c||c|}
\nu \backslash \mu
    & (3) & (2,1) & (4) & (3,1) & (2^2) & (5)   \\
\hline
(4)   &   & 4  & 2 & 15 & 3  & 3  \\
(3,1) & 2 & 12 & 4 & 46 & 9  & 6  \\
(2^2) & 2 & 9  & 3 & 31 & 6 & 5  \\
(2,1^2)& 2& 12 & 4 & 46 & 9 & 6 \\
(1^4) &   & 4  &   & 15 & 3 & 2
\end{array}
$

\smallskip

$
\begin{array}{l||c||c|c||c|}
\nu \backslash \mu
    & (2) & (3) & (2,1) & (4)   \\
\hline
(5)   &   & 2  & 12 & 4      \\
(4,1) &  & 4 & 49  & 10     \\
(3,2) & 2 & 5  & 60 & 13 \\
(3,1^2)& 2& 6 & 72 & 17 \\
(2^2,1) &   & 4  & 60 & 14 \\
(2,1^3) &  & 4 & 49 & 12 \\
(1^5) &  & & 12 & 3 \\
\end{array}
$

\smallskip

$
\begin{array}{l||c||c|}
\nu \backslash \mu
    & (2)  & (3)  \\
\hline
(6) &    & 2    \\
(5,1) &  2   & 7     \\
(4,2) & 3  & 14  \\
(4,1^2)& 2  & 16 \\
(3^2) &     & 8  \\
(3,2,1) & 4  & 25 \\
(3,1^3) &  2 & 18 \\
(2^3) & 2 &  8 \\
(2^2,1^2) & 2 & 15 \\
(2,1^4) &    &  8  \\
(1^6) &     &  2 \\
\end{array}
$

\bigskip\bigskip

 \end{document}